\documentclass[11pt, preprint]{imsart}
\usepackage{xr}
\usepackage{parskip}
\RequirePackage{amsthm,amsmath,amsfonts,amssymb,mathtools}
\usepackage[nothm]{yk}
\usepackage{mathrsfs,dsfont}
\usepackage[margin=3.5cm]{geometry}

\usepackage[framemethod=TikZ]{mdframed}
\RequirePackage[comma,sort&compress, numbers]{natbib}
\usepackage[colorlinks]{hyperref}
\hypersetup{
	colorlinks,%
	citecolor=blue,%
	filecolor=black,%
	linkcolor=blue,%
	urlcolor=blue
}
\RequirePackage{graphicx}

\startlocaldefs

\newcommand{\hum}{\hat{\mu}_n}
\newcommand{\hun}{\hat{\nu}_n}

\newcommand{\EE}{\mathbb{E}}
\newcommand{\R}{\mathbb{R}}
\newcommand{\mcx}{\mathcal{X}}
\newcommand{\bnir}[1]{\mathscr{N}_{[\,]}(#1,\cF,\lVert\cdot\rVert_r)}
\newcommand{\benr}[1]{\mathscr{H}_{[\,]}(#1,\cF,\lVert \cdot \rVert_r)}
\newcommand{\bnif}[1]{\mathscr{N}_{[\,]}(#1,\cF,\lVert\cdot\rVert_{\infty})}
\newcommand{\benf}[1]{\mathscr{H}_{[\,]}(#1,\cF,\lVert \cdot \rVert_{\infty})}
\newcommand{\benfg}[1]{\mathscr{H}_{[\,]}(#1,\cG,\lVert \cdot \rVert_{\infty})}

\newcommand{\beno}[1]{\mathscr{H}_{[\,]}(#1,\cF,\lVert\cdot\rVert_{1})}
\newcommand{\benoc}[1]{\mathscr{H}_{[\,]}(#1,\{\mathbf{1}(C):\ C\in\mathcal{C}\},\lVert\cdot\rVert_{1})}
\newcommand{\bent}[1]{\mathscr{H}_{[\,]}(#1,\cF,\lVert\cdot\rVert_{2})}

\newcommand{\bentgd}[1]{\mathscr{H}_{[\,]}(#1,\cG_{\delta},\lVert\cdot\rVert_{2})}
\newcommand{\ornm}[1]{\left\lVert #1 \right\rVert_{\phi,P}}
\newcommand{\lrp}[1]{\left\lVert #1 \right\rVert_{r,P}}
\newcommand{\ltp}[1]{\left\lVert #1 \right\rVert_{2,P}}
\newcommand{\gjl}[1]{g_j^{#1,L}}
\newcommand{\gju}[1]{g_j^{#1,U}}
\newcommand{\gl}[1]{g_{j_f}^{#1,L}}
\newcommand{\gu}[1]{g_{j_f}^{#1,U}}
\newcommand{\tgl}[1]{\tilde{g}_{f}^{#1,L}}
\newcommand{\tgu}[1]{\tilde{g}_{f}^{#1,U}}
\newcommand{\ppbi}{\pi_{\infty,\beta}}
\newcommand{\mL}{\mathcal{L}}
\newcommand{\Rg}{\R^{\ge 0}}
\newcommand{\tq}[1]{\tilde{q}_{n,\beta}(#1)}
\newcommand{\tqg}[1]{\tilde{q}_{n,\gamma}(#1)}
\newcommand{\orin}[1]{\lVert #1 \rVert_{\infty,\mcx}}
\newcommand{\Lbp}{\Lambda_{\phi,\beta}}
\newcommand{\cph}{c_{\phi}}
\newcommand{\got}{\Theta^{(1)}_{f,s}}
\newcommand{\gotz}{\Theta^{(1)}_{f,0}}
\newcommand{\gtt}{\Theta^{(2)}_{f,s}}

\newcommand{\Ld}[1]{\Lambda^{\star}_n({#1})}

\newcommand{\Ldi}[1]{\Lambda^{\star}_{n,\infty}({#1})}
\newcommand{\mbH}[1]{\mathbb{H}_{\phi}({#1},\cF)}
\newcommand{\mbHi}[1]{\mathbb{H}_{\infty}({#1},\cF)}

\newcommand{\vep}{\varepsilon}
\newcommand{\ppb}{\pi_{\phi,\beta}}
\newcommand{\bnum}[1]{\mathscr{N}_{[\,]}(#1,\cF,\lVert\cdot\rVert_{\phi})}
\newcommand{\benp}[1]{\mathscr{H}_{[\,]}(#1,\cF,\lVert\cdot\rVert_{\phi})}
\newcommand{\df}[1]{\Delta_f^{#1}}
\newcommand{\rno}{R_{n,\beta}^{(1)}}

\newcommand{\rto}{R_{n,\beta}^{(2)}}

\newcommand{\indep}{\perp \!\!\! \perp}

\newcommand{\add}{\mathrm{add}}

\DeclarePairedDelimiter\floor{\lfloor}{\rfloor}

\newcounter{theo}[section]
\renewcommand{\thetheo}{\arabic{section}.\arabic{theo}}

\usepackage[noabbrev, capitalize]{cleveref}
\allowdisplaybreaks

\theoremstyle{plain}

\newtheorem{theorem}{Theorem}
\newtheorem{lemma}{Lemma}
\newtheorem{prop}{Proposition}
\newtheorem{assumption}{Assumption}
\newtheorem{cor}{Corollary}

\theoremstyle{remark}
\newtheorem{definition}{Definition}

\newtheorem{rem}{Remark}

\numberwithin{assumption}{section}
\numberwithin{theorem}{section}
\numberwithin{rem}{section}
\numberwithin{example}{section}
\numberwithin{definition}{section}
\numberwithin{lemma}{section}
\numberwithin{cor}{section}
\numberwithin{prop}{section}
\endlocaldefs
\begin{document}
\begin{frontmatter}
\title{Trade-off Between Dependence and Complexity for   Nonparametric Learning --- an Empirical Process Approach}

\begin{aug}
\author[A]{\fnms{Nabarun} \snm{Deb}\thanksref{t1}\ead[label=e1]{nabarun.deb@chicagobooth.edu}}
\and
\author[B]{\fnms{Debarghya} \snm{Mukherjee}\thanksref{t1}\ead[label=e2]{
mdeb@bu.edu}}
\thankstext{t1}{Both authors have equal contribution}
\address[A]{Econometrics and Statistics,
University of Chicago Booth School of Business, 
\printead{e1}}

\address[B]{Department of Mathematics and Statistics,
Boston University,
\printead{e2}}
\end{aug}

\runtitle{Maximal inequalities for empirical processes under dependence}
\runauthor{Deb N. and Mukherjee D.}

\begin{abstract}
Empirical process theory for i.i.d.  observations has emerged as a ubiquitous tool for understanding the generalization properties of various statistical problems. However, in many  applications where the data exhibit \emph{temporal dependencies} (e.g., in finance, medical imaging, weather forecasting etc.), the corresponding empirical processes are much less understood. Motivated by this observation, we present a general bound on the expected supremum of empirical processes under standard $\beta/\rho$-mixing assumptions. Our bounds take the form of weighted square root bracketing entropy integrals where the weighing function captures the strength of dependence. Unlike most prior work, our results cover both the long and the short-range regimes of dependence. Our main result shows that the learning rate in a large class of nonparametric problems is characterized by a non-trivial trade-off between the complexity of the underlying function class and the dependence among the observations. This trade-off reveals a new phenomenon, namely that \emph{even under long-range dependence, it is possible to attain the same rates as in the i.i.d. setting, provided the underlying function class is complex enough}. We demonstrate the practical implications of our findings by analyzing various statistical estimators in both fixed and growing dimensions. Our main examples include a comprehensive case study of generalization error bounds in nonparametric regression over smoothness classes in \emph{fixed as well as growing dimension}  using neural nets, shape-restricted multivariate convex regression, estimating the optimal transport (Wasserstein) distance between two probability distributions, and classification under the Mammen-Tsybakov margin condition --- all under appropriate mixing assumptions. In the process, we also develop bounds on $L_r$ ($1\le r\le 2$)-localized empirical processes with  dependent observations, which we then leverage to get faster rates for (a) tuning-free adaptation, and (b) set-structured learning problems. 
\end{abstract}


\begin{keyword}[class=MSC]
\kwd[Primary ]{37A25}
\kwd{62G05}
\kwd[; Secondary ]{62M10}
\kwd{60E15}
\end{keyword}

\begin{keyword}
\kwd{$\beta$ and $\rho$-mixing}
\kwd{classification}
\kwd{empirical process for dependent data}
\kwd{empirical risk minimization}
\kwd{non-Donsker}
\kwd{rate of convergence}
\end{keyword}

\end{frontmatter}
\tableofcontents

\section{Introduction}
\label{sec:intro}
In modern statistical applications, empirical process theory is a vital tool for analyzing model generalization errors. It helps us understand the trade-off between the number of training samples ($n$) and the complexity of the underlying model (usually through a hypothesis function class $\cF$), which affects how well a fitted model performs on unseen test data. This typically involves maximal inequalities that provide finite sample bounds on 
\begin{equation}\label{eq:target}
    \EE\sup_{f\in\cF} \big|\mathbb{G}_n(f)\big|, \ \mbox{where} \ \mathbb{G}_n(f):=\frac{1}{\sqrt{n}}\sum_{i=1}^n (f(X_i)-\EE f(X_i)).
\end{equation}
Here $X_1,X_2,\ldots , X_n$ is drawn according to some data generating process (DGP). When this DGP yields  \emph{independent} observations, i.e., when the $Z_i$s are independent, tight bounds on \eqref{eq:target} are well understood; see \cite{van1996weak,van2000empirical,wainwright2019high,han2021set} and the references therein for details. Perhaps the most popular example of $\cF$ is $\{\mathbf{1}_{(-\infty,x]}(\cdot), \, \forall \, x\in\R\}$, in which case, bounding \eqref{eq:target} yields the Glivenko-Cantelli theorem~\cite[Theorem 19.1]{vaart_1998}. It states that the empirical distribution converges to the population distribution function uniformly at $n^{-1/2}$ rate. For more complex $\cF$, such bounds have been used extensively to obtain rates of convergence in various statistical problems such as classification \cite{tsybakov2004optimal,GineKoltchinskii2006}, nonparametric regression \cite{Birge1993,Sabyasachi2018}, metric estimation \cite{chizat2020faster,sriperumbudur2009integral}, generative adversarial networks \cite{liang2021well,singh2018nonparametric}, etc.

However, the \emph{independence assumption} on the underlying DGP can be quite restrictive in various practical learning scenarios where temporal dependencies exist among the observed samples, as seen in time-series data, weather reports, stock prices, subscriber counts of streaming platforms over time, etc., see \cite{boente1989robust,jiang2001robust,Vidyasagar2003,dundar2007learning,hanneke2021learning} and the references therein. Due to the \emph{lack of exchangeability} among the $Z_i$s, standard techniques for bounding \eqref{eq:target}, which typically proceed using symmetrization and Rademacher complexity bounds (see \cite{van1996weak}), are usually not feasible. Consequently, to the best of our knowledge, upper bounds of \eqref{eq:target} for \emph{general $\cF$} (both Donsker and non-Donsker classes, to be defined later) and across \emph{short and long-range} dependence (see \cref{def:ShortVsLongDep}) have not been studied in the literature in a \emph{fully nonparametric} setting. This raises the natural question:

\begin{center}
    \emph{How does the rate of convergence of \eqref{eq:target} depend on the complexity of $\cF$ and the strength of dependence among the $X_i$s?}
\end{center}

In this paper, we provide general upper bounds of \eqref{eq:target} under popular \emph{mixing assumptions}, namely $\beta$ and $\rho$-mixing (see Definitions \ref{def:beta_mixing} and \ref{def:rho_mixing} below), to address this question. Our maximal inequalities do not impose parametric modeling assumptions on the DGP and consequently cover function classes $\cF$ of all complexities. Furthermore, we do not assume the summability of the mixing coefficients, which enables us to cover both short and long-range dependence (see \cref{def:ShortVsLongDep}). In the process, we uncover a new phenomenon ---  \emph{even under long-range dependence, it is possible to attain the same rates as in the i.i.d. setting, provided the underlying function class is complex enough}. This is particularly common in nonparametric empirical risk minimization (ERM) problems when the underlying intrinsic dimension of the DGP is large. The flexibility of our general bounds allows us to cover a large class of statistical applications in both \emph{low and high dimensions}, such as regression, classification, estimating probability metrics, etc., under the appropriate mixing assumptions.

\subsection{Main contributions}
\label{sec:maincontrib}
As mentioned above, we study the problem of bounding the expected supremum of empirical processes involving data from a strictly stationary $\beta/\gamma$ mixing sequence, over function classes with general complexity quantified by their bracketing entropy (see \cref{def:bracketing}). 
Our results can cover various estimation problems simultaneously. 
An implication of our main results is that the more complex the function class, the more dependence the problem can withstand up to which the rates are exactly the same as in the i.i.d. case. For ease of exposition, we provide an informal description of the trade-off below, in the context of bounding empirical processes with uniform bracketing. 

\emph{An illustrative example of the trade-off.}  Suppose $\cF$ is a collection of bounded functions on some space $\mathcal{X}$. Assume that the bracketing entropy (see \cref{def:bracketing} for details) of $\cF$ is bounded above by $C\delta^{-\alpha}$ for some $\alpha,C>0$. Finally suppose that $X_1,\ldots , X_n$ is a strictly stationary $\beta$-mixing sequence (see \cref{def:beta_mixing} for details) with mixing coefficient $\beta_k\le c(1+k)^{-\beta}$ for some $\beta>0$. Then \cref{thm:infmainres}, in particular \cref{cor:infmainres}, yields the following: if $\beta>1$, then:
    $$
 \EE\sup_{f\in \cF}|\mathbb{G}_n(f)|\lesssim \begin{cases}
       1, & \text{ if } 0 < \alpha < 2 \,, \\
        n^{\frac{1}{2} - \frac{1}{\alpha}}, & \text{ if }\alpha > 2\,,
    \end{cases}
    $$
    and if $\beta< 1$, then: 
    $$
 \EE\sup_{f\in \cF}|\mathbb{G}_n(f)|\le \begin{cases}
        C n^{\frac{1-\beta}{2(1+\beta)}} , & \text{ if } 0 < \alpha < \frac{1+\beta}{\beta} \,, \\
        C n^{\frac{1}{2} - \frac{1}{\alpha}}, & \text{ if }\alpha > \frac{1+\beta}{\beta}\,.
    \end{cases}
    $$
Here, the $\lesssim$ hides constants free of $n$. Note that the above result implies that for $\alpha>2$ (which implies $1/(\alpha-1)<1$), and $\beta\in (1/(\alpha-1),1)$, we get the same rates of convergence $n^{\frac12-\frac1\alpha}$ as in the i.i.d. setting. Note that $\beta\in (1/(\alpha-1),1)$ falls in the long-range dependence regime as per Definition \ref{def:ShortVsLongDep}. Also, the seminal paper of \cite{Birge1993} shows that this rate is not improvable in general even in the i.i.d. setting. This phenomenon is also illustrated in \cref{fig:curve1}.
    
\begin{figure}[ht]
    \centering
    \includegraphics[width=0.60\textwidth]{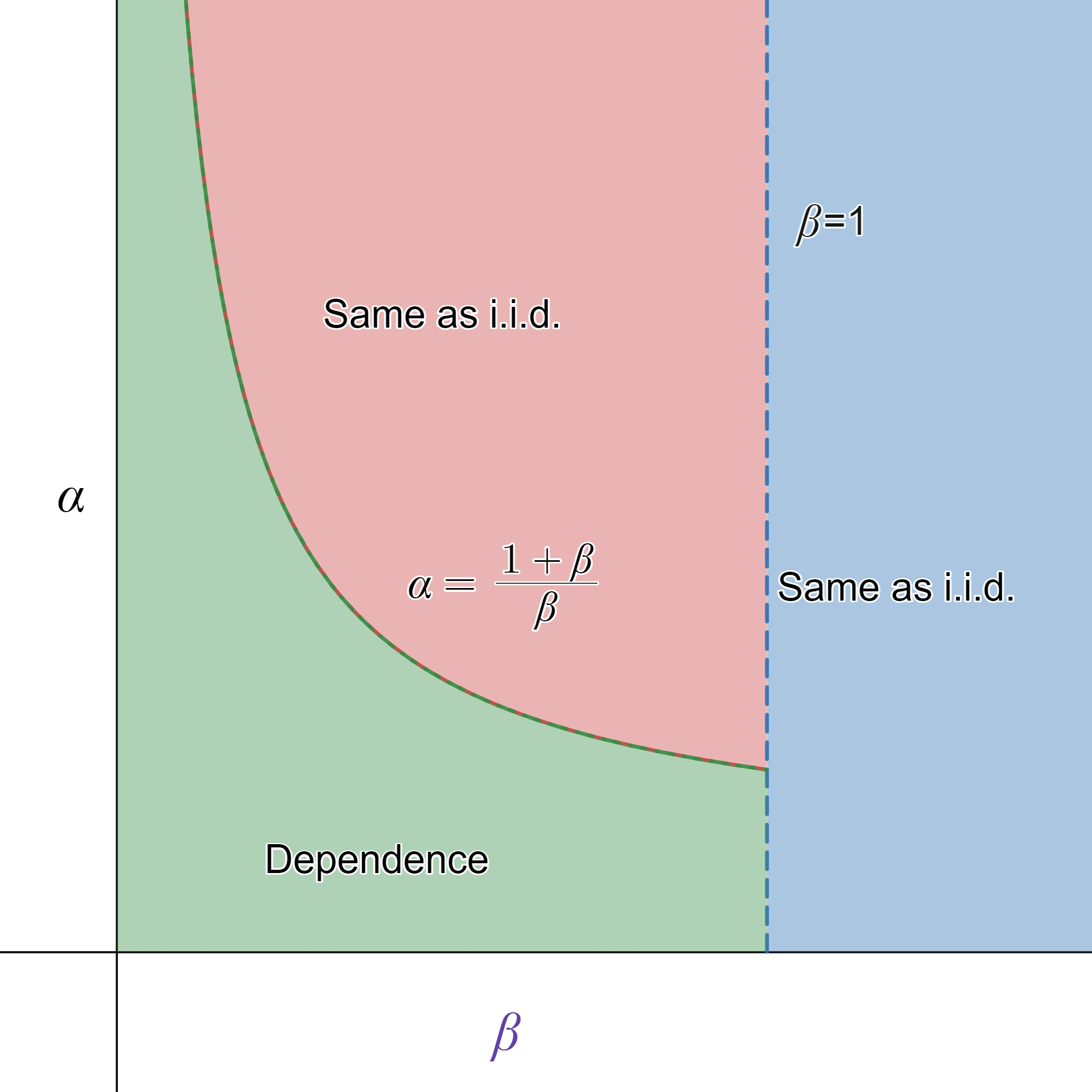}
    \caption{Phase transition curve in terms of $\alpha$ (complexity of the function class) and $\beta$ (the level of dependence) for $\cL_\infty$ covers. When the parameters fall in the red or the blue parts of the picture, the rates are the same as in the i.i.d. case whereas in the green region, the effect of dependence shows up. As is evident the phase transition curve is a smooth parabola up to $\beta=1$ after which the rates are always the same as for i.i.d. observations.}
    \label{fig:curve1}
\end{figure}

Let us now outline our main contributions. 
\begin{enumerate}
    \item \emph{Maximal inequalities under $\beta$-mixing assumptions  with $L_r$-bracketing, $2<r<\infty$}. In \cref{thm:phimainres}, we present a maximal inequality  when the underlying DGP is a strictly stationary $\beta$-mixing sequence, for a general function class with respect to its $L_r$ bracketing entropy, for $2<r<\infty$. In fact, the result is presented for brackets with respect to more general Orlicz type norms. We illustrate the trade-off between dependence and complexity in an interpretable setting in \cref{cor:nondonsk_global}. Note that $L_r$-bracketing for $r<\infty$ is of significant importance, as certain function classes such as the space of bounded multivariate convex functions (without Lipschitz constraints) does not have a finite bracketing number for $r=\infty$.
    \item \emph{Maximal inequalities under $\beta$-mixing assumptions  with $L_\infty$-bracketing}. In the special case where $r=\infty$, the trade-off simplifies significantly as illustrated in the example above. Many function classes of practical relevance such as the space of fully connected deep neural nets admit finite bracketing entropy with respect to the $L_\infty$ norm. The technical details are deferred to \cref{thm:infmainres} and \cref{cor:infmainres}. 
    \item \emph{Maximal inequalities under $\gamma$-mixing assumptions  with $L_r$-bracketing, $1\le r\le 2$}. An inspection of the proofs of Theorems \ref{thm:phimainres} and \ref{thm:infmainres} shows that the proof techniques break down for $1\le r\le 2$. This regime is of particular interest at the intersection of empirical process theory and empirical risk minimization (ERM) as it allows us to get faster rates for ERM estimators even under dependence, than by naive empirical process bounds. In \cref{thm:roc}, under a stronger $\gamma\equiv \beta\vee\rho$-mixing assumption, we show that the following equation 
    $$\sqrt{n}\delta_n^2\sim \EE\sup_{f\in\cF:\ \lVert f\rVert_2\le \delta_n}\mathbb|{G}_n(f)|,$$
    \emph{characterizes the rate of convergence of ERM estimators under $\gamma$-mixing}. 
    The occurrence of the $L_2$ norm is crucial in the above equation. This motivates us to study maximal inequalities under $\gamma$-mixing with $L_r$ bracketing, $1\le r\le 2$ in \cref{thm:gamma_mixing}; also see Corollaries \ref{cor:nonpbd} and \ref{cor:vcbd}. These results lead to rates of convergence faster than $n^{-1/2}$ in \cref{cor:fast_error} even for less complex function classes. Two other important corollaries that illustrate faster rates in learning problems are also provided: (a) in tuning-free adaptation, see \cref{cor:adaptbd} and (b) in set-structured problems, see \cref{cor:fasterl1}.
    \item We provide various examples to illustrate the  applicability of our results under appropriate mixing assumptions --- 
    
    (a) For smooth function estimation using neural nets in both fixed and growing dimension, see Theorems \ref{thm:dnn_mixing} and \ref{thm:additive_rate}. To the best of our knowledge, these are the first theoretical results for neural networks in a high dimensional setting, under dependence. 
    
    (b) For shape restricted multivariate convex least squares regression with tuning-free adaptation, see \cref{thm:shapeconv}. We show that the optimal adaptation rates for least squares estimators are attainable even under long range dependence. 
    
    (c) For Wasserstein distance estimation to highlight the benefits of regularization under strong dependence, see \cref{th:otbalance}.  
    
    (d) For classification under the Mammen-Tsybakov margin condition to highlight that our results are applicable for ERMs with non-convex and non-smooth loss functions, see \cref{prop:classerr}.

\end{enumerate}

\subsection{Literature review} 
Estimating parametric or nonparametric components of statistical models under temporal dependency is a longstanding problem, engaging the efforts of numerous statisticians over the years. There are various ways to quantify the dependence among the observations, of which the most popular notions are so-called strong mixing conditions, namely $\alpha$-mixing (\cite{rosenblatt1956central}), $\beta$-mixing (\cite{volkonskii1959some}), $\rho$-mixing (\cite{kolmogorov1960strong}), $\psi$-mixing (\cite{blum1963strong}), $\phi$-mixing (\cite{ibragimov1959some}), etc.  
Roughly speaking, these notions of mixing conditions quantify how dependence between two observations at two time points $(t_1, t_2)$ depends on $|t_1 - t_2|$, especially how it decays as $|t_1 - t_2| \to \infty$. 
A classical example is \cite{Roussas1990}, which analyzed nonparametric kernel regression under various mixing assumptions. 
In this paper, we mostly focus on $\beta$-mixing and $\rho$-mixing condition (see Definitions \ref{def:beta_mixing} and \ref{def:gammamix}) as they are popularly studied in the literature. 

For ease of presentation, we mostly focus here on $\beta$-mixing sequences. Informally speaking, $\beta_q$ denotes the dependence among two observations that are $q$ time units apart (i.e. say $X_j$ and $X_{q + j}$) (see \cite[Section 5]{Samorodnitsky2016}, \cite[Section 3.3]{Guegan2005}). Typically, the dependence is classified into two categories --- short-range or long-range depending on whether the mixing coefficients (say $\beta_q$s) are summable or not (see \cref{def:ShortVsLongDep}). Under short-range dependency, several results, that hold in i.i.d setup also continue to hold. 
For example, by \cite[Theorem 1.2]{rio1993covariance},  $\limsup_n n^{-1}\mbox{Var}(S_n)<\infty$ (where $S_n = \sum_{j = 1}^n X_i$), which implies that central limit theorem and Donsker Invariance principle type results (see \cite[Theorem 1]{doukhan1995invariance}) continue to hold.
On the other hand, there are examples of stationary $\beta$-mixing sequences with long-range dependence, where $\liminf_n n^{-1}\mbox{Var}(S_n)=\infty$ (see \cite{Davydov1969,Davydov1973}; also  \cite{Bradley2007rocky,Bradley1989}). In these examples, by Lamperti's Theorem (see \cite[Theorem 2.1.1]{Embrechts2002}), Donsker Invariance principle ceases to hold. 
Therefore, analyzing observations under long-range dependency is technically much harder than under short-range dependency. Note that long-range dependency is not only a mere technical construct, it also occurs in various practical applications such as hydrology and economics as has been argued in \cite[Chapter 12.4]{brockwell1991time}; also see \cite{hurst1951long,mcleod1978preservation,lawrance1977stochastic,granger1980long} and the references therein.

 
In modern applications, complex functions involving observed data beyond $S_n$ are often explored under generic nonparametric assumptions. Addressing inferential questions in this context involves bounding the supremum of specific empirical processes over function classes, yielding convergence rate results. The impact of dependence on these rates, especially influenced by the extent of $\beta$-mixing, becomes apparent under strong enough dependence. Initial exploration of dependence in empirical processes can be found in works like \cite{doukhan1994functional,yu1994rates,doukhan1995invariance,nobel1993note}. 
For example, \cite{nobel1993note} shows that any Glivenko-Cantelli (GC) classes of function under i.i.d. observations is also (GC) under any $\beta$-mixing sequence as long as $\beta_q \to 0$ as $q \uparrow \infty$. 
\cite{nobel1999limits} demonstrates that ergodicity alone is insufficient for obtaining consistent regression estimates or classification rules. We refer the reader to \cite{irle1997consistency,berti1997glivenko,steinwart2009learning,hansen2008uniform} for additional consistency results in regression settings under dependence. 

Theoretical machine learning problems under dependence have been investigated in various domains, including optimal transport \cite{Connor2022, Fournier2015}, neural networks \cite{kurisu2022adaptive,kengne2023deep, ma2022theoretical,kurisu2022adaptive,kengne2023deep}, and generalization bounds in regression, classification, support vector machines, etc \cite{Mohri2010,hanneke2019statistical}. 
In learning theory, \cite{hanneke2021learning,dawid2020learnability} establishes consistency results under minimal assumptions for data from a given stochastic process. \cite{hansen2008uniform} explores learning rates under geometric $\alpha$-mixing and \cite{roy2021empirical} obtains rates for the heavy-tailed setting under geometric $\beta$-mixing. In another line of work, authors have focused on non-stationary time series but with short-range $\beta$-mixing \cite{hanneke2019statistical,Kuznetsov2014,barrera2021generalization,Farahmand2012}. 
For a general suite of techniques at the intersection of empirical process theory and data dependency, readers can refer to \cite{Dehling2002book}. 
However most of the existing work either focuses on exponentially decaying mixing conditions (e.g. $\beta_q \sim e^{-cq}$) or/and on function classes of ``low" complexity (such as those with finite Vapnik-Chervonenkis (VC) dimension). 
While \cite{yu1994rates} explored the most general setting ($\beta_q \sim (1+q)^{-\beta}$ for any $\beta > 0$), our results provide a broader perspective on the trade-off between the complexity of the underlying function class and dependency (refer to \cref{rem:yu} for details). In a different paper, \cite{Chen2017} can handle long-range dependence but for structured linear time series models using martingale techniques. As mentioned in \cite[Remark 17]{Chen2017}, their assumptions are not comparable to mixing assumptions.

\subsection{Notation}
\label{sec:nota}
Define a class of function $\Phi$ as follows: 
\begin{align*}
\Phi:=\bigg\{\phi:\Rg\to& \Rg,\ \phi\ \mbox{is convex, increasing and differentiable},\\ & \phi(0)=0,\ \lim_{x\to\infty}\frac{\phi(x)}{x}\to\infty,\ \int_0^1 \phi^{-1}\left(\frac{1}{u}\right)\,du<\infty\bigg\}.
\end{align*}
Here $\Rg$ denotes the space of non-negative reals. 
Let $P\in\mathcal{P}(\mathcal{X})$ where $\mathcal{P}(\mcx)$ denotes the set of probability measures  supported on some Polish space $\mcx$. Given a function $h:\mcx\to\R$ and $\phi\in\Phi$, we define the Orlicz norm of $h$ as follows:
\begin{equation}
\label{eq:orlicz_norm}
\ornm{h}:=\inf\left\{t>0:\ \EE_{X\sim P}\left[\phi\left(\frac{h^2(X)}{t^2}\right)\right]\le 1\right\} \,.
\end{equation}
Here the infimum of an empty set is $\infty$. The Orlicz space corresponding to $\phi$ is then defined as:
$$\mL_{\phi}(P):=\{h:\mcx\to\R: \ornm{h}<\infty\}.$$
As as example, if we take $\phi(x) =  x^{r/2}$ for some $r >  2$, then 
$$
\|h\|_{\phi, P} = \|h\|_{L_r(P)} := \lrp{h}:=\left(\int |h(x)|^r\,d P(x)\right)^{\frac{1}{r}} \,
$$ 
and it is easy to check that $\int_0^1 \phi^{-1}(1/u)\,du<\infty$. 
Another useful norm for our purposes would be $\orin{\cdot}$ where $\orin{h}:=\sup_{x\in\mcx}|h(x)|$. 
Furthermore, given any nonincreasing and c\`{a}dl\`{a}g function $f:\R\to\R$, define its inverse $f^{-1}$ as follows:
$$
f^{-1}(u):=\inf\{t:f(t)\le u\} \,.
$$
The basic property of $f^{-1}$ we will use is that 
$$t<f^{-1}(u)\quad \mbox{if and only if}\quad f(t)>u.$$
For $\phi\in\Phi$ and $h\in \mL_{\phi}(P)$, we define the standard Legendre-Fenchel conjugate as
$$h^c(y)=\sup_{x\ge 0}\ (xy-h(x))$$
for $y\ge 0$. Given any $t\in\R$, we use $\floor*{t}$ to be the largest integer less than or equal to $t$. Given any function $h\in \mL_{\phi}(P)$ for $\phi\in \Phi$, define $Q_{h}(t)$ to be the inverse of the tail probability function $t\mapsto P(|f(X)|>t)$. We write $\overset{\mathscr{L}}{=}$ to denote equality in distribution. Given a natural number $n\in\bbN$, let $[n]:=\{1,2,\ldots ,n\}$. We also use $a\vee b$ and $a\wedge b$ to denote the maximum and the minimum between two real numbers $a$ and $b$. Given a real number $x$, set $x_{+}:=x\ \vee \ 0$. We use $a_n \lesssim b_n$ and $a_n \gtrsim b_n$ to denote $a_n<=C b_n$ and $a_n >= C b_n$, where $C>0$ hides constants free $n$. If $a_n<=C b_n$ and $a_n>=c b_n$, then we write $a_n\sim b_n$. Also $\tilde{O}$ will denote the standard ``Big O'' notation with logarithmic factors suppressed. We will write $\mathrm{Id}$ for the identity function and $\mathbf{1}$ for the indicator function. Further, for two probability measures $P$ and $Q$, let $\lVert P-Q\rVert_{\mathrm{TV}}$ denote the total variation distance between them.

\section{Maximal inequalities with $L_r$ bracketing, $2<r<\infty$}
\label{sec:mainres}
In this section, we present our main result(s) to bound the expected supremum of an empirical process under appropriate $\beta$-mixing conditions. We start with a couple of definitions.
\begin{definition}[$\beta$-mixing (see \cite{doukhan2012mixing})]
    \label{def:beta_mixing}
    For a paired random element $(X, Y)$ the $\beta$-dependence or $\beta$-mixing coefficient between them is defined as: 
    $$
    \beta(X, Y) := \|P_{X, Y} - P_X \otimes P_Y\|_{TV} \,
    $$
    where $P_{XY}$, $P_X$, and $P_Y$ denote the joint distribution of $X$ and $Y$, and the marginal distributions of $X$ and $Y$. 
    A strictly stationary (see \cite[Definition 1.3.3]{brockwell1991time}) sequence $\{X_n\}_{n \ge 0}$ is said to satisfy $\beta$-mixing condition if the sequence $\beta_q$ defined as: 
    $$
    \beta_q := \sup_k\left\|P_{\mathscr{F}_{0:k}, \mathscr{F}_{(k+q+1):\infty}} - P_{\mathscr{F}_{0:k} \otimes \mathscr{F}_{(k+q+1):\infty}}\right\|_{TV}\to 0, \, \mbox{as}\, q\to\infty,
    $$
    where $\mathscr{F}_{a:b}$ denotes the sigma-field generated by $X_a, X_{a+1}, \dots, X_b$. 
    \end{definition}
    In other words, the sequence $\beta_n$ denotes the total variation norm between the joint distribution of the sequence $\{X_1, \dots, X_k, X_{k+n+1}, \dots, X_\infty\}$ and the product of the distributions of $\{X_1, \dots, X_k\}$ and $\{X_{k+n+1}, \dots, X_\infty\}$. We also adopt the convention $\beta_0 = 1$. For other equivalent definitions of $\beta$-mixing, see \cite{bradley2005basic,yu1994rates}.
    Our main assumption on the data-generating process is as follows: 
\begin{assumption}
\label{assm:dgp}
The observed data $\{X_1, \dots, X_n\}$
is strictly stationary $\beta$-mixing with coefficients $\{\beta_q\}_{q\ge 0}$. 
\end{assumption}
In most standard datasets with temporal dependence, $\beta_q$ is a decreasing function of $q$. Depending on how fast $\beta_q$ decays as $q$ grows, the structure of dependency can be broadly categorized into two regimes: 
\begin{definition}[short-range vs long-range dependence]
\label{def:ShortVsLongDep}
    If the dependency decays sufficiently slow, i.e. $\sum_{j = 1}^\infty \beta_j = \infty$ then the dependency among the observations is termed as \emph{long-range dependency}. On the other hand, if  $\sum_{j = 1}^\infty \beta_j < \infty$, then it is called \emph{short-range dependency}. 
\end{definition}
One can similarly define short and long-range dependence with other mixing assumptions; see \cite{RACHEVAIOTOVA2003641}. Several researchers have extended the results available for i.i.d setup to exponentially $\beta$-mixing data (where $\beta_q \le Ce^{-cq}$ for some constant $C, c$, a particular instance of short-term dependency), e.g. \cite{krebs2018large, baraud2001adaptive, Mohri2010, Vidyasagar2003} to name a few. The crucial technical tool for analyzing stationary $\beta$-mixing sequences is a blocking technique usually attributed to \cite{yu1994rates}, along with Berbee's coupling lemma \cite{Berbee1979, Goldstein1979}, which we will also use in our analysis. 


As mentioned in the previous section, the aim of this paper is to 
understand how the size of \eqref{eq:target} changes depending on  the interplay between the dependency and the complexity of the underlying function class. We measure this complexity  by the growth of the bracketing number, a popular choice in empirical process theory, as defined below:

\begin{definition}[Bracketing number]
\label{def:bracketing} 
Given a collection of functions $\cF$ and any $\delta > 0$, the bracketing number $\bnum{\delta}$ with respect to Orlicz norm $\|\cdot\|_\phi$ as defined in \eqref{eq:orlicz_norm}, is the minimum number $N$ such that there exists $N$ pairs of measurable functions $\{F_i^{(L)},F_i^{(U)}\}_{i=1}^N$ satisfying:
    \begin{enumerate}
        \item Given any $f\in\cF$, $\exists \ i\in [N]$ with $F_i^{(L)}(x)\le f(x)\le F_i^{(U)}(x)$ for all $x\in\mcx$.
        \item $\|F_i^{(U)}-F_i^{(L)}\|_\phi \le \delta$ for all $1 \le i \le N$. 
    \end{enumerate}
\end{definition}
The $\delta$-bracketing entropy is defined as $\benp{\delta} :=\log{\bnum{\delta}}$. 
Furthermore, if $\|f\|_\infty \le b$ for all $f \in \cF$, then  given any $\delta$-bracket of $\cF$, we can replace each $F_i^{(L)}$ by $F_i^{(L)}\vee (-b)$ and $F_i^{(U)}$ by $F_i^{(U)}\wedge b$ to construct a new $\delta$-bracket. Therefore, without loss of generality, we can always assume that all functions in a $\delta$-bracket of $\cF$ are also uniformly bounded by $b$. When $\phi$ is replaced by the $L_r(P)$ norm, $1\le r\le \infty$, the corresponding  $\delta$-bracketing number (resp. $\delta$ bracketing entropy) is denoted by $\bnir{\delta}$ (resp. $\benr{\delta}$). For $r=\infty$, as both $\sigma$ and $b$ denote bounds on the $\infty$ norm, we simplify the notation as $\bnif{\delta}$ (resp. $\benf{\delta}$).

In practice, often relatively sharp upper bounds are known on the bracketing numbers as defined above. For example, if $\cF$ is the class of all bounded monotone functions on $[0, 1]$, then we know $\bent{\delta} \le C\delta^{-1}$ for some constant $C$. Henceforth, we assume that there exists a function $\mathbb{H}_{(\cdot)}(\cdot,\cdot)$ which is non-increasing in the first parameter such that 
\begin{equation}\label{eq:asnubd}
     \benp{\delta}\le \mbH{\delta}\qquad \mbox{for}\ \delta>0.
\end{equation}
For the collection of monotone functions $\bbH_{\phi}(\delta, \cF) = C_r\delta^{-1}$ for $\phi(x)=x^{r/2}$ for all $r>2$. The function $\bbH$ quantifies the complexity of $\cF$ through its bracketing entropy.

For dependency, we use $\beta$-mixing as defined in Definition \ref{def:beta_mixing}. 
We use the notation $\beta(t)$ to denote the following step function: $\beta(t)=\beta_{\floor*{t}}$ for $t\ge 0$. Given a function $\phi \in \Phi$ (see Subsection \ref{sec:nota}) and a function class $\cF$ with $\|f\|_{\phi, P} \le \sigma$ and $\|f\|_\infty \le b$ for all $f \in \cF$, we define two pivotal quantities, $\tq{\cdot}$ and $\Lbp(n)$ which are imperative to our analysis: 
\begin{align}
\label{eq:tqd}
\tq{\delta} & :=\inf\left\{0\le q\le n:\ \beta_q\le \frac{q}{n}\bigg(1+\sum_{k\ge 0:\ 2^{-k}\sigma\ge \delta} \mbH{2^{-k}\sigma}\bigg)\right\}\,, \\
\label{eq:piv}
\Lbp(n) & :=\sum_{i=0}^{n} \int_0^{\beta_i} \phi^{-1}\left(\frac{1}{u}\right)\,du \,.
\end{align}
Here $\tq{\cdot}$ is defined for $\delta\le \sigma$. 
The definition of $\tq{\delta}$ can be thought as a version of \emph{dependence-complexity} trade-off; as $q$ increases, $\beta_q$ decreases, however, the complexity term (right-hand side of the inequality in \eqref{eq:tqd}) increases and $\tq{\delta}$ balances these two opposite forces. 
The other function $\Lambda_{\phi, \beta}$ relates the $\beta$-mixing coefficients with the Orlicz norm $\phi$. The following lemma establishes some basic characteristics of $\tq{\cdot}$ and $\Lbp$:  
\begin{lemma}
\label{lem:basefn}
The function $\tq{\cdot}$ is well-defined, and always greater than or equal to $1$. 
Furthermore, both $\tq{\cdot}$ and $\Lbp(\tq{\cdot})$  are non-decreasing in $\delta$.
\end{lemma}

The following assumption imposes some mild restrictions on the product of $\Lbp(\tq{\cdot})$ and $\mbH{\cdot}$:

\begin{assumption}\label{assm:bsbd}
Let $\cF$ be a family of functions such that $\lVert f\rVert_{\phi}\le \sigma$, $\lVert f\rVert_{\infty}\le b$ for some $\phi\in\Phi$ and $\sigma>0$, $b\ge 1$.  
Assume that there exist functions $\Ld{\cdot}$ and  $\rno$ from $[0, \sigma]$ to $\reals_+$, 
with $\Ld{\cdot}$ being non-decreasing, and with  $\rno$ being non-increasing, which satisfy the following properties: 
\begin{align}
   \label{eq:asncall} &\Lbp(\tq{\delta})\le\Ld{\delta}\,,  \\
    \label{eq:asnbd1}
& \Ld{\delta}\left(1 + \mbH{\delta}\right)\le \rno(\delta)\,.
\end{align}
\end{assumption}
In \cref{assm:bsbd}, we may have as well taken $\Ld{\cdot}=\Lbp(\tq{\cdot})$. However, in practice, as $\tq{\cdot}$ needs to be an integer, due to rounding issues it is often convenient to apply \cref{thm:phimainres} with an upper bound, namely $\Ld{\cdot}$ in our case. The assumption of monotonicity on $\Ld{\cdot}$ is natural because $\Lbp(\tq{\cdot})$ is  monotonic by \cref{lem:basefn}. The assumption in \eqref{eq:asnbd1} is slightly more technical but it holds in all our problems of interest, as will be evident in the examples discussed throughout the rest of the paper. Roughly speaking, we need a product of non-decreasing function $\Ld{}$ and non-increasing function $\bbH$ to be bounded by some non-increasing function $\rno$, i.e. the growth of $\Ld{\cdot}$ is dominated by the decay of $\bbH$. 

We are now in a position to state our main theorem. The collection of functions $\cF$ is always assumed to be countable to avoid measurability digressions.

\begin{theorem}[Main theorem]
\label{thm:phimainres}
Consider Assumptions \ref{assm:dgp} and \ref{assm:bsbd}. Choose $a > 0$ such that: 
\begin{equation}
\label{eq:basedef1}
C_0\int_{\frac{a}{2^6\sqrt{n}}}^{\sigma} \sqrt{\rno(u)}\,du\ \le \ a \ \le \ 8\sqrt{n}\sigma \,,
\end{equation}
for some $C_0\ge 1$. Then we have:
$$
\EE\sup_{f\in\cF}\big|\mathbb{G}_n(f)\big|\leq C \left(a+b \tq{\sigma}\frac{1+\mbH{\sigma}}{\sqrt{n}}\right),
$$
where $C>0$ is an absolute constant depending only on $C_0$ and $\phi$.
\end{theorem}

\begin{rem}[Comparison in the i.i.d. case]\label{rem:comvdv}
In \cite[Theorem 8.13]{van2000}, the author presents a maximal inequality for empirical processes of independent variables. Our main result \cref{thm:phimainres} can be viewed as a generalization of this result to the non i.i.d. case under a $\beta$-mixing assumption. To see this, we note in particular that when $\beta(k)=0$ for all $k\ge 1$, it is easy to check that for all $\delta>0$, we can choose
$$\tq{\delta}=1,\qquad \mbox{and}\qquad \Ld{\delta}=1,$$
irrespective of the function $\benp{\cdot}$. Therefore, by \eqref{eq:asnbd1}, we have $\rno(u)\ge 1\vee \mbH{u}$. Hence, from the lower bound in \eqref{eq:basedef1}, we have: 
\begin{align*}
    a\ge C_0\int_{\frac{a}{2^6\sqrt{n}}}^{\sigma} \sqrt{\rno(u)}\,du \ge C_0\int_{\frac{a}{2^6\sqrt{n}}}^{\sigma} 1\,du\implies a\ge  C_0\sigma,
\end{align*}
and 
\begin{align*}
    a\ge C_0\int_{\frac{a}{2^6\sqrt{n}}}^{\sigma} \sqrt{\rno(u)}\,du \ge C_0\int_{\frac{a}{2^6\sqrt{n}}}^{\sigma} \sqrt{\mbH{u}}\,du \,.
\end{align*}
Combining both the bounds above, $a$ should satisfy: 
\begin{equation}
\label{eq:a_lower_bound}
a \ge C_0 \max\left\{\int_{\frac{a}{2^6\sqrt{n}}}^{\sigma} \sqrt{\mbH{u}}\,du, \ \ \sigma \right\} \,.
\end{equation}
As $b\tq{\sigma}(1+\mbH{\sigma})/\sqrt{n} = b (1+\mbH{\sigma})/\sqrt{n}$, \cref{thm:phimainres} then implies: 
\begin{align}
\label{eq:newbasedef2}
\EE\sup_{f\in\cF}\big|\mathbb{G}_n(f)\big|\leq C \left(a +b \frac{1+\mbH{\sigma}}{\sqrt{n}}\right), 
\end{align}
for all $ a\le 8\sqrt{n}\sigma$, satisfying \eqref{eq:a_lower_bound}. 
It is now easy to check that, with different implicit constants, \eqref{eq:basedef1} provides an analog of \cite[Theorem 8.13]{van2000}; see in particular the version in \cite[Lemma 7]{hanisotonic2019}. Hence one can recover the standard bound for the i.i.d. case from \cref{thm:phimainres} by simply putting $\beta_k = 0$.  
\end{rem}


    In practice, the $\beta$-mixing coefficients (or upper bounds thereof) are often unknown. In that case, one could estimate the coefficients based on the binning method in \cite{mcdonald2015estimating}. In combination with our results, this should enable us to get high probability bounds involving the estimated $\beta$-mixing coefficients. We reserve this direction for future research.

The above theorem is presented in a rather general form and is applicable in a variety of settings involving choices of $\phi$, $\beta$, and $\mbH{\delta}$. For ease of exposition let us now illustrate \cref{thm:phimainres} in a popular example widely used in learning theory. More specific examples will be provided in \cref{sec:app}. We begin with some informal computations. In the sequel, we will use the $\approx$ sign for natural approximations (that hide universal constants). Consider, as before, that the observed sequence $\{X_1, \dots, X_n\}$ satisfies~\cref{assm:dgp} with $\beta_q \approx (1+q)^{-\beta}$   and $\benp{\delta} \approx \delta^{-\alpha} :=\mbH{\delta}$, $\sigma,b$ bounded away from $0$, and $\phi(x) = x^{r/2}$ with $r > 2$. Assume that we are in the non-Donsker sub regime $\alpha>4$ and the long-range dependence regime $$0<\frac{r}{\alpha r - \alpha - r}<\beta<1.$$
Then for $u>0$ not ``too small or large", we obtain $\tq{u}$ by solving 
$$\beta_q\approx \frac{q}{n}\left(1+\sum_{k\ge 0:\ 2^{-k}\sigma\ge u} \mbH{2^{-k}\sigma}\right)\, \implies \tq{u}\approx \left(n u^{\alpha}\right)^{\frac{1}{1+\beta}}.$$
Consequently, 
$$\Ld{u}\lesssim \left(n u^{\alpha}\right)^{\frac{1-\beta\left(1-\frac{2}{r}\right)}{1+\beta}}, \quad \mbox{and}\quad \rno(u)\lesssim n^{\frac{1-\beta\left(1-\frac{2}{r}\right)}{1+\beta}}u^{-\frac{2\alpha \beta (r-1)}{r(1+\beta)}}.$$
As $b\tq{\sigma}(1+\mbH{\sigma})/\sqrt{n}\lesssim n^{\frac{1}{2}-\frac{\beta}{1+\beta}}$ and 
$$a\ge \int_{\frac{a}{\sqrt{n}}}^{\sigma} \sqrt{\rno(u)}\,du,\, \implies \ a\geq \bigg(\frac{a}{\sqrt{n}}\bigg)^{1-\frac{\alpha \beta (r-1)}{r (1+\beta)}} n^{\frac{1-\beta\left(1-\frac{2}{r}\right)}{2(1+\beta)}},\, \implies \ a\approx n^{\frac{1}{2}-\frac{1}{\alpha}},$$
we get from \cref{thm:phimainres} that 
$$
\EE\sup_{f\in\cF_{\phi, \sigma, b}}\big|\mathbb{G}_n(f)\big|\lesssim n^{\frac{1}{2}-\frac{1}{\alpha}}.
$$
The right hand side above is free of $\beta$ and is in fact the same rates as one would get if the underlying DGP yielded i.i.d. observations (see \cite{Birge1993,gine2021mathematical}). This reveals the striking phenomenon that it is possible to get i.i.d. like rates even under long-range dependence provided the complexity of the underlying hypothesis class $\cF$ is large enough. The above informal computations are formalized in the following corollary which provides quantitative bounds on \eqref{eq:target} for a much larger class of DGPs. 

\begin{cor}
\label{cor:nondonsk_1}
Consider the same notation as in \cref{thm:phimainres} with $\sigma<1$ and $b\ge 1$, and suppose that \cref{assm:dgp} holds. Further, assume that $\phi(x)=x^{r/2}$ for some $r>2$. Also suppose that 
$\benp{\delta}\le K\delta^{-\alpha}$ for all $\delta>0$ and some $K>0$ depending on $b$ but not on $\delta$ or $\sigma$. 
Finally let $\beta_q\le c (1+q)^{-\beta}$ for some $c, \beta>0$ and all integers $q\ge 0$. 
Then there exists $C>0$ depending only on $c$ such that the following conclusions hold for all $\sigma \gtrsim n^{-1/\alpha}$:

\begin{enumerate}
    \item If $\beta>r/(r-2)$, then 
    $$
  \EE\sup_{f\in \cF}\big|\mathbb{G}_n(f)\big|\le \begin{cases}
        C\max\left\{\sigma^{1 - \frac{\alpha}{2}}, \ n^{\frac{1 - \beta}{2(1 + \beta)}} \sigma^{-\frac{\alpha \beta}{1 + \beta}}\right\}, & \text{ if }\ 0 < \alpha < 2 \,, \\
        C \max\left\{n^{\frac12 - \frac{1}{\alpha}}, n^{\frac{1 - \beta}{2(1 + \beta)}} \sigma^{-\frac{\alpha \beta}{1 + \beta}}\right\}, & \text{ if }\ \alpha > 2\,.
    \end{cases}
    $$
    \item If $\beta< r/(r - 2)$, then: 
    $$
 \EE\sup_{f\in \cF}\big|\mathbb{G}_n(f)\big|\le \begin{cases}
       C\max\left\{n^{\frac{1-\beta\left(1-\frac{2}{r}\right)}{2(1+\beta)}} \sigma^{1-\frac{\alpha(r-1)\beta}{r(1+\beta)}}, n^{\frac{1 - \beta}{2(1 + \beta)}} \sigma^{-\frac{\alpha \beta}{1 + \beta}}\right\} & \text{ if }\ 0 < \alpha < \frac{r(1+\beta)}{\beta(r-1)} \,, \\
        C \max\left\{n^{\frac12 - \frac{1}{\alpha}}, n^{\frac{1 - \beta}{2(1 + \beta)}} \sigma^{-\frac{\alpha \beta}{1 + \beta}}\right\}, & \text{ if }\ \alpha > \frac{r(1+\beta)}{\beta(r-1)}\,.
    \end{cases}
    $$
\end{enumerate}
    In both parts $C>0$ is a constant depending only on $b$, $K$, $r$, $\alpha$ and $\beta$.
\end{cor}

For ease of interpretation, let us consider the global version of \cref{cor:nondonsk_1} where $\sigma$ is assumed to be bounded above and below (away from $0$). For example, when $\lVert f\rVert_{\phi,P}\le 1$ for all $f\in\cF$. Then we have the following analogue of \cref{cor:nondonsk_1}: 
\begin{cor}
\label{cor:nondonsk_global}
Consider the same setting as Corollary \ref{cor:nondonsk_1}. Then we have: 

\begin{enumerate}
    \item If $\beta>r/(r-2)$, then 
    $$
  \EE\sup_{f\in \cF}\big|\mathbb{G}_n(f)\big| \le \begin{cases}
        C\max\left\{1, \ n^{-\frac{1}{2} + \frac{1}{1 + \beta}}\right\} \equiv C, & \text{ if }\ 0 < \alpha < 2 \,, \\
        C n^{\frac12 - \frac{1}{\alpha}}, & \text{ if }\ \alpha > 2\,.
    \end{cases}
    $$
    \item If $\beta< r/(r - 2)$, then: 
    $$
  \EE\sup_{f\in \cF}\big|\mathbb{G}_n(f)\big|\le \begin{cases}
       C\max\left\{n^{\frac{1-\beta\left(1-\frac{2}{r}\right)}{2(1+\beta)}},  n^{-\frac{1}{2} + \frac{1}{1 + \beta}} \right\} \equiv Cn^{\frac{1-\beta\left(1-\frac{2}{r}\right)}{2(1+\beta)}} & \text{ if }\ 0 < \alpha < \frac{r(1+\beta)}{\beta(r-1)} \,, \\
        C n^{\frac12 - \frac{1}{\alpha}}, & \text{ if }\ \alpha > \frac{r(1+\beta)}{\beta(r-1)}\,.
    \end{cases}
    $$
\end{enumerate}
    In both parts $C>0$ is a constant depending only on $C_1$, $C_2$, $r$, $\alpha$ and $\beta$.
\end{cor}
The above Corollary quantifies the trade-off between dependency and complexity of the underlying function class for a fixed $\ornm{\cdot}$ norm. The upshot is simple and intuitive; if $\beta$ is larger than a threshold (i.e. $(r/(r-2)$), then the level of dependence is rather weak, and we recover the same rates as with independent observations. When $\beta$ is smaller than $r/(r-2)$, then the picture is more subtle. In this case, whether or not we observe the same rates as for independent data depends on an interplay between complexity and the dependence governed by the following hyperbola in terms of $\alpha$ and $\beta$:
$$\alpha\equiv \alpha(\beta)=\frac{r}{r-1}\left(1+\frac{1}{\beta}\right).$$
As the right-hand side is decreasing in $\beta$, \cref{cor:nondonsk_global} implies that stronger dependence (smaller values of $\beta$) requires more complex  function classes (larger $\alpha$) to get the same rates as in the i.i.d. setting. On the other hand, when $\beta < r/(r-2)$ and $\alpha < r(1+\beta)/(\beta(r-1))$, i.e. we have both strong dependency and a relatively simple function class, then the effect of dependence dominates the size of the function class and the rates are slower than in the i.i.d. setting. 

\begin{rem}[Boundary cases]
We avoid the cases $\alpha=2$ in part 1, and $\alpha=\frac{r(1+\beta)}{\beta(r-1)}$ in part 2 of the above corollary. In these cases the rates differ from the corresponding $\alpha>2$ and $\alpha>\frac{r(1+\beta)}{\beta(r-1)}$ only up to logarithmic factors. One can also obtain rates for exponentially decaying $\beta$-mixing coefficients, i.e., $\beta_q\le \exp(-c q)$ for some $c>0$. In that case, the bounds are the same as in part 1 above. 
\end{rem}

\begin{rem}[On the $r/(r-2)$ threshold]
  In \cite{bradley1985central}, the author shows that given any $\beta<r/(r-2)$, there exists a strictly stationary sequence with mixing coefficient $\beta_q\le (1+q)^{-\beta}$ and a function $f$ with $\EE|f(X_1)|^r<\infty$ such that $\bbG_n(f)$ does not converge in distribution, and consequently the Donsker Invariance principle cannot hold in general. This suggests that it is only reasonable to expect i.i.d. like behavior (irrespective of the size of the function class), when $\beta>r/(r-2)$. In this sense, the threshold $r/(r-2)$ obtained in \cref{cor:nondonsk_1} is tight.
\end{rem}
\begin{rem}
\label{rem:comparison_doukhan}
In \cite[Thorem 1]{doukhan1995invariance}, the authors prove a functional Donsker's Invariance principle under a  $\beta$-mixing assumption by obtaining a maximal inequality in \cite[Theorem 2]{doukhan1995invariance}. In the setting of \cref{thm:phimainres}, their maximal inequality is meaningful if $\beta>r/(r-2)$ and $\int_0^1 \sqrt{\benp{u}}\,u<\infty$. It can be used to derive the first part of \cref{cor:nondonsk_1}, part 1. In contrast, \cref{cor:nondonsk_1} can deal with much larger function classes that arise in standard nonparametric regression for instance, where $\int_0^1 \sqrt{\benp{u}}\,u=\infty$. Moreover, the more interesting part of \cref{cor:nondonsk_1} is in part 2 with $\beta<r/(r-2)$ (strong dependence), where there is a non-trivial trade-off between dependence and complexity as explained above. This setting is not covered in \cite[Theorem 2]{doukhan1995invariance}. 
To be able to cover the case when $\beta$ is small, there are a number of new  technical challenges that arise compared to the proof of \cite[Theorem 2]{doukhan1995invariance}. On the surface, our proof of \cref{thm:phimainres} relies on a chaining argument via adaptive truncation; a technique potentially pioneered in \cite{Bass1985} which has since been used extensively in empirical process theory; see e.g.~\cite{Ossiander1987,van2000,pollard2002maximal} and the references therein. However, due to the lack of summability of the mixing coefficients for small $\beta$, several steps in the proof of \cite[Theorem 2]{doukhan1995invariance} break down. Overcoming them requires important technical innovations. Two of these include: 

(i) A new maximal inequality for function classes of finite cardinality in \cref{thm:maximal_finite_davy}. In contrast to a related result \cite[Lemma 3]{doukhan1995invariance}, we track the potential inflation for small values of $\beta$ explicitly in \cref{thm:maximal_finite_davy}.

(ii) A new upper bound on the $L_1(P)$ norm of a function in terms of a truncated sequence of weak norms involving the mixing coefficients; see \cref{lem:l1phin}. The level of truncation changes at each step of the chaining argument. This is not required in the setting of \cite[see e.g. Lemma 4]{doukhan1995invariance} which only covers mixing coefficients that are summable in an appropriate sense.
\end{rem}

\section{Maximal inequalities with Uniform $L_{\infty}$ bracketing}
\label{sec:unifbrckt}
In many applications of empirical process theory, both in statistics and machine learning, brackets (see \cref{def:bracketing}) are available in the $L_{\infty}$ norm, in which case some of the computations for bounding the maximal inequality simplify. Interested readers may consider chapter 2.7 of the second edition of \cite{van1996weak} for several examples of function classes and bound on their $L_\infty$ covering/bracketing numbers. 
The main result of this section captures the resulting scenario. Let us first state the modified set of assumptions below. As in the previous section, consider a function class $\cF$ such that $\lVert f\rVert_{\infty}\equiv \lVert f\rVert_{\infty,\mcx}\le \sigma$ for some $0<\sigma\le 1$. Define a function $\mbHi{\cdot}$ such that 
\begin{equation}
\label{eq:asnubin}
\benf{\delta} \le \mbHi{\delta} \qquad \mbox{for}\ \delta>0.
\end{equation}
With a notational abuse, we redefine $\tq{\delta}$ as in \eqref{eq:tqd}, with $\mbH{\cdot}$ replaced with $\mbHi{\cdot}$. The following assumption is analogous to Assumption \ref{assm:bsbd} tailor-made for the $L_\infty$ norm: 

\begin{assumption}\label{assm:bsbdinf}
 Let $\Ldi{\cdot}:[0,\sigma]\to \Rg$ and $\rto:[0,\sigma]\to\Rg$ be two functions such that $\Ldi{\cdot}$ is non-decreasing an $\rto$ is non-increasing. Further, 
 \begin{equation}\label{eq:asnbd23}
     \sum_{k=1}^{\tq{\delta}} \beta(k)\le \Ldi{\delta}
 \end{equation}
 \begin{equation}\label{eq:asnbd3}
 \Ldi{\delta}\left(1 + \mbHi{\delta}\right)\le \rto(\delta).
 \end{equation}
 Here $\mbHi{\cdot}$ is constructed so as to \eqref{eq:asnubin}.
 \end{assumption}
 
 With Assumption \ref{assm:bsbdinf} at our disposal, let us present the modified version of \cref{thm:phimainres}.
\begin{theorem}
\label{thm:infmainres}
Suppose Assumptions \ref{assm:dgp} and  \ref{assm:bsbdinf} hold. Choose $a>0$ satisfying 
\begin{equation}\label{eq:baseinf}
C_0 \int_{\frac{a}{2^6\sqrt{n}}}^{\sigma} \sqrt{\rto(u)}\,du \, \le a \le 8\sqrt{n}\sigma 
\end{equation}
for some $C_0\ge 1$. Then we have:
$$\EE\left[\sup_{f\in\cF}\big|\mathbb{G}_n(f)\big|\right]\leq C \left(a+ \sigma\tq{\sigma}\frac{1 + \mbHi{\sigma}}{\sqrt{n}}\right),$$
where $C>0$ is an absolute constant depending only on $C_0$.
\end{theorem}
Theorem \ref{thm:infmainres} is analogue of Theorem \ref{thm:phimainres}, the only difference is the norm of the localization; in the previous section $\cF_\sigma$ was a  sub-collection of $f$  such that $\|f\|_\phi \le \sigma$ and here $\cF$ is a sub-collection of $f$ such that $\|f\|_\infty \le \sigma$. Although \cref{thm:infmainres} relies on a stronger assumption, one advantage of working with $L_\infty$ entropy is that we do not pay a price for small $\sigma$ in the last term of the above display.

To get a better sense of this implication of \cref{thm:infmainres}, we present the analog of \cref{cor:nondonsk_1} below.

\begin{cor}\label{cor:infmainres}
Consider the same setting as in \cref{thm:infmainres}.  Assume that $\benf{\delta} \le C_1 \delta^{-\alpha}$ for some $C_1>0$ and $\beta_q\le C_2 (1+q)^{-\beta}$ for some $C_2,\beta>0$ and all integers $q\ge 0$. Then the following conclusions hold:

\begin{enumerate}
    \item If $\beta>1$, then:
    $$
 \EE\sup_{f\in \cF}|\mathbb{G}_n(f)|\le \begin{cases}
        C \left(\sigma n^{\frac{1-\beta}{2(1+\beta)}} +  \sigma^{1-\frac{\alpha}{2}}\right), & \text{ if } 0 < \alpha < 2 \,, \\
        C n^{\frac{1}{2} - \frac{1}{\alpha}}, & \text{ if }\alpha > 2\,.
    \end{cases}
    $$
    \item If $\beta< 1$, then: 
    $$
 \EE\sup_{f\in \cF}|\mathbb{G}_n(f)|\le \begin{cases}
        C n^{\frac{1-\beta}{2(1+\beta)}} \sigma^{1-\frac{\alpha\beta}{1+\beta}}, & \text{ if } 0 < \alpha < \frac{1+\beta}{\beta} \,, \\
        C n^{\frac{1}{2} - \frac{1}{\alpha}}, & \text{ if }\alpha > \frac{1+\beta}{\beta}\,.
    \end{cases}
    $$
\end{enumerate}
    In both parts $C>0$ is a constant depending only on $C_1$, $C_2$, $\alpha$ and $\beta$.
\end{cor}

\begin{rem}[Comparison with \cite{yu1994rates}]
\label{rem:yu}
In \cite[Theorem 3.1 and Corollary 3.2]{yu1994rates}, the author considers the case of VC-type function classes, i.e., where $\benf{\delta}\le C \log(1/\delta)$ (which is equivalent to the setup of Corollary \ref{cor:infmainres} with $\alpha = 0$ and an additional log factor) and shows that $$ \sup_{f \in \cF_{r, \sigma}}\left|\bbG_n(f)\right| = O_P(n^{\frac{1-s}{2(1+s)}}),$$ for all $s<\beta$. Observe that \cref{thm:infmainres} provides a more comprehensive insight in multiple ways. Firstly, we can get the exact rate $n^{\frac{1-\beta}{2(1+\beta)}}$ in addition to providing finite sample bounds as opposed to the asymptotic guarantees in \cite{yu1994rates}. Moreover \cite{yu1994rates} heavily relies on the function class being VC whereas \cref{thm:infmainres} can handle much more complex function classes, many of which arise in popular nonparametric problems. Consequently \cite{yu1994rates} does not reflect the trade-off between complexity and dependence that can be seen in the $\beta<1$ regime.
\end{rem}

\section{Localization with $L_r$ bracketing, $1\le r\le 2$}
\label{sec:erm}
One shortcoming of Theorems \ref{thm:phimainres} and \ref{thm:infmainres} is that they do not cover the case where we only have $L_2\equiv L_2(P)$ brackets, as the function class $\Phi$ (defined in \cref{sec:nota}) does not include the function $\phi(x)=x$, (because $\phi^{-1}$ will not be integrable near $0$ and consequently \eqref{eq:piv} will be undefined). However, a control on \eqref{eq:target} based on the $L_2(P)$ bracketing entropy is often needed, especially for quantifying the rate of convergence, say $\delta_n$, of least squares estimators which can be obtained by solving 
$$\int_{\delta_n^2}^{\delta_n} \sqrt{\bent{u}}\,du\sim \sqrt{n} \delta_n^2$$
in the i.i.d. setting, see \cite{Birge1993,GineKoltchinskii2006,van2000empirical}. A careful inspection of our proof reveals that the key difficulty in extending Theorems \ref{thm:phimainres} and \ref{thm:infmainres} to the $L_2(P)$ bracketing setting is that  we cannot bound $\mbox{Cov}(f(X_i), f(X_j))$ directly in terms of the $\beta(|i - j|)$ and $\mbox{Var}(f(X_1))$ (see the proof of Theorem \ref{thm:maximal_finite_davy}).
However, this issue can be taken care of if we resort to another notion of dependency, the $\rho$-mixing coefficient of  \cite{kolmogorov1960strong}, which is defined below. 

\begin{definition}[$\rho$-mixing]\label{def:rho_mixing}
    Given a paired random element $(X,Y)$ with law $P$ and marginals $P_X$, $P_Y$, we define the $\rho$-mixing coefficient between them as: 
    $$
\rho(X,Y) := \sup\left\{\left|\mathrm{cor}(f(X), g(Y))\right|, f \in L_2(P_X), g \in L_2(P_Y)\right\} \,.
$$

Similar to the $\beta$-mixing coefficient in \cref{def:beta_mixing}, we can define the $\rho$-mixing coefficient as: 
$$
\rho(n) := \sup_{k} \rho\left(\cF_{(-\infty, k]}, \cF_{[k+n, \infty)}\right)
$$
where $\cF_{[a, b]}$ is the $\sigma$-field generated by $\{X_a, X_{a+1}, \dots, X_b\}$.
\end{definition}

Note that, now we can bound $\mbox{Cov}(f(X_i), f(X_j))$ by $\rho(|i-j|)$ and $\mbox{Var}(f(X_1))$. Although $\rho$-mixing is advantageous in bounding the correlation directly, one main drawback is that we do not have any analog of Berbee's coupling lemma \cite{Berbee1979,Goldstein1979} for $\rho$-mixing coefficient, a key pillar for establishing maximal inequalities. Therefore, we resort to a new mixing coefficient to utilize the best of both worlds.

\begin{definition}[$\gamma$-mixing]\label{def:gammamix}
    Given a paired random element $(X,Y)$, we define the $\gamma$-mixing coefficient between them as: 
    $$
\gamma(X,Y) := \beta(X,Y)\vee \rho(X,Y)\,.
$$

Similar to the $\beta$-mixing coefficient in \cref{def:beta_mixing}, we can define the $\gamma$-mixing coefficient as: 
$$
\gamma_n := \beta_n \vee \rho_n.
$$
\end{definition}

We are now in position to state the new maximal inequality for bounding \eqref{eq:target} under a stronger $\gamma$-mixing assumption. While the focus is on $L_2(P)$-bracketing, we will nevertheless state the result for $L_r(P)$ bracketing, $r\ge 1$, for the sake of generality. The case $r=1$ is also of interest in some ``set-structured" problems (see \cref{rem:setstruc} and \cref{sec:fastratestruc}). Before stating the theorem, let us redefine a few notations to move from $\beta$-mixing to $\gamma$-mixing. As before, we work with a function class $\cF \equiv \cF_{\sigma, b}$ satisfying   $\|f\|_\infty \le b$ and $\|f\|_r \le \sigma$ for all $f \in \cF$. Recall that we denote by $\benr{\delta}$, the $\delta$-bracketing entropy of $\cF\equiv\cF_{\sigma, b}$ with respect to $L_r(P)$ norm at level $\delta$ (see \cref{def:bracketing}) and $\bbH_r(\delta, \cF)$ be such that $$\benr{\delta}\le \bbH_r(\delta, \cF).$$ As in Section \ref{sec:mainres}, we define $\tilde q_{n, \gamma}(\delta)$ as: 
\begin{equation}
     \label{eq:tq_gamma}
    \tqg{\delta} =\inf\left\{0\le q\le n:\ \gamma_q\le \frac{q}{n}\bigg(1+\sum_{k\ge 0:\ 2^{-k}\sigma\ge \delta} \bbH_r(2^{-k}\sigma,\cF)\bigg)\right\} \,.
\end{equation}
Given these new notations, we are ready to present our main results. 

\begin{theorem}
    \label{thm:gamma_mixing}
Suppose the observed data $\{X_1,X_2,\ldots ,X_n\}$ is strictly stationary $\gamma$-mixing with mixing coefficients $\gamma_q\le c(1+q)^{-\gamma}$ for some $c,\gamma>0$ and all integers $\gamma\ge 0$. Also assume that there exist $r\ge 1$, $B\ge b\vee \sigma \vee e$, $\theta\ge \sigma$, $D\ge 1$, $\theta\ge\sigma$, $K_r>0$, and $V,\alpha\ge 0$, such that 
$$\bbH_r(u,\cF)\le K_r D \left(\frac{\theta}{u}\right)^{\alpha}\left(\log{\left(\frac{B}{u}\right)}\right)^V,$$
for $0<u\le \sigma$. Finally assume that the above parameters satisfy 
\begin{align}
\label{eq:basedef2_gamma}
D\left(\frac{\theta}{\sigma}\right)^{\alpha} \left(\log{\left(\frac{B}{\sigma}\right)}\right)^V\lesssim n.
\end{align}
Define $\tilde{r}:=r\wedge 2$. Then $\EE\sup_{f\in\cF}|\mathbb{G}_n(f)|\lesssim \Pi_n(\cF)$ where $\Pi_n(\cF)$ is defined below.

\begin{enumerate}

    \item If $0\le \alpha<\tilde{r}$, then 
    \begin{align*}
        \Pi_n&(\cF):=\sigma^{\frac{\tilde{r}}{2}}\left[D\left(\frac{\theta}{\sigma}\right)^{\alpha}\right]^{\frac{\gamma}{2(\gamma+1)}}n^{\frac{1}{2(1+\gamma)}}\left(\log{\left(\frac{B}{\sigma}\right)}\right)^{\frac{V\gamma}{2(\gamma+1)}}\\ &+n^{\frac{1}{2}-\frac{\gamma}{1+\gamma}}\left(D\left(\frac{\theta}{\sigma}\right)^{\alpha}\left(\log{\left(\frac{B}{\sigma}\right)}\right)^V\right)^{\frac{\gamma}{1+\gamma}}.
    \end{align*}

    \item If $\sigma\gtrsim n^{-\frac{1}{\alpha+2-\tilde{r}}}$ and $\tilde{r}<\alpha < \tilde{r}(1+\gamma^{-1})$, then 
    \begin{align*}
        \Pi_n&(\cF):= \sigma^{\frac{\tilde{r}}{2}}\left[D\left(\frac{\theta}{\sigma}\right)^{\alpha}\right]^{\frac{\gamma}{2(\gamma+1)}}n^{\frac{1}{2(1+\gamma)}}\left(\log{\left(\frac{B}{\sigma}\right)}\right)^{\frac{V\gamma}{2(\gamma+1)}}\\ &+n^{\frac{1}{2}-\frac{\gamma}{1+\gamma}}\left(D\left(\frac{\theta}{\sigma}\right)^{\alpha}\left(\log{\left(\frac{B}{\sigma}\right)}\right)^V\right)^{\frac{\gamma}{1+\gamma}}+\left(D\theta^{\alpha}n^{\frac{\tilde{r}-\alpha}{2}}\right)^{-\frac{1}{\alpha+2-\tilde{r}}}\left(\log{\left(\frac{B}{\sigma}\right)}\right)^{\frac{V}{2}}.
    \end{align*}

    \item If $\sigma\gtrsim n^{-\frac{1}{\alpha+(2-\tilde{r})(1+\gamma^{-1})}}$ and $\alpha > \tilde{r}(1+\gamma^{-1})$, then 
    \begin{align*}
        \Pi_n&(\cF):=n^{\frac{\gamma(\alpha-\tilde{r})+(2-\tilde{r})}{2(\alpha \gamma+(2-\tilde{r})(\gamma+1))}}\left( D \theta^{\alpha}\right)^{\frac{\gamma}{\alpha \gamma+(2-\tilde{r})(\gamma+1)}}\left(\log{\left(\frac{B}{\sigma}\right)}\right)^{\frac{V\gamma}{2(\gamma+1)}}\\ &+n^{\frac{1}{2}-\frac{\gamma}{1+\gamma}}\left(D\left(\frac{\theta}{\sigma}\right)^{\alpha}\left(\log{\left(\frac{B}{\sigma}\right)}\right)^V\right)^{\frac{\gamma}{\gamma+1}}+n^{\frac{1}{2(1+\gamma)}}\sigma^{\frac{\tilde{r}}{2}}\left(D\left(\frac{\theta}{\sigma}\right)^{\alpha}\left(\log{\left(\frac{B}{\sigma}\right)}\right)^V\right)^{\frac{\gamma}{2(\gamma+1)}}\\ &+n^{\frac{\gamma(\alpha-\tilde{r})}{2(\alpha \gamma + (2-\tilde{r})(\gamma+1))}}\left( D \theta^{\alpha}\right)^{\frac{2\gamma+2-\tilde{r}}{2(\alpha \gamma + (2-\tilde{r})(\gamma+1))}}\left(\log{\left(\frac{B}{\sigma}\right)}\right)^{\frac{V}{2}}.
    \end{align*}
\end{enumerate}

    Here $\lesssim$ hides constants that depend on $b$, $r$, $K_r$, $c$, $\gamma$, $\alpha$, and $V$.
\end{theorem}
Although we have stated \cref{thm:gamma_mixing} for general $r$, note that the relevant exponent is $\tilde{r}=r\wedge 2$. Therefore, the above result yields its most useful implications when $1\le r\le 2$, which as we pointed out earlier, falls beyond the scope of our earlier results, say Theorems \ref{thm:phimainres} and \ref{thm:infmainres}. Let us provide some implications of \cref{thm:gamma_mixing} for ease of exposition. The first corollary below can be viewed as an analogue of \cref{cor:nondonsk_1} but with $r=2$. To avoid repetition, we only focus on the case where $\alpha$ is ``large" enough or $\cF$ is complex enough.

\begin{cor}[Non-Donsker classes and long-range dependence]\label{cor:nonpbd}
    Consider the same assumption on the data generating mechanism as in \cref{thm:gamma_mixing}. Choose $r=2$ and assume that $\bbH_2(u,\cF)\le K u^{-\alpha}$ for $K>0$, $\alpha>0$, $\sigma\sim 1$, and $b,K\lesssim 1$. Then the following conclusion holds for $\alpha>2(1+\gamma^{-1})$:
     \begin{align*}
        \EE\sup_{f\in\cF}|\mathbb{G}_n(f)|\lesssim n^{\frac{1}{2}-\frac{1}{\alpha}}.
    \end{align*}
\end{cor}

Therefore, if $\alpha>4$, then in the \emph{long-range regime}  $(2/(\alpha-2),1)$, it is possible to attain i.i.d. like rates of convergence once we have appropriate control on the size of the $L_2$-brackets of $\cF$. 

\begin{rem}[Comparison with \cref{cor:nondonsk_1}]
    Informally speaking, if we assume $\beta_k\sim \rho_k\lesssim (1+k)^{-\gamma}$ and plug-in $r=2$ in \cref{cor:nondonsk_1}, part 2, then we get the same rates as in \cref{cor:nonpbd}. This is only informal as the constants in \cref{cor:nondonsk_1} diverge as $r\to 2$. Nevertheless, this observation suggests that we did not loose out on statistical accuracy while extending our quantitative results to the $L_2$ bracketing case in this section. Of course there is a trade-off as $\gamma$-mixing is stronger than $\beta$-mixing by definition.
\end{rem}

The final corollary we highlight here is when $\cF$ is a VC (Vapnik-Chervonenkis) class of functions (see \cite{vapnik1999nature}) with VC dimension $D$. These are classes of low complexity but nevertheless very useful in learning theory as they can  approximate various other function classes of interest. Some examples of VC classes include deep neural nets \cite{schmidt2020nonparametric}, reproducing kernel Hilbert spaces \cite{yang2020predicting}, wavelets \cite{schutt1984entropy}, etc. 
\begin{cor}[VC classes and weak dependence]\label{cor:vcbd}
    Consider the same assumption on the data generating mechanism as in \cref{thm:gamma_mixing}. Choose $r=2$ and assume that $\bbH_2(u,\cF)\le D (\log{(B/u)})^V$ for $B\ge \sigma \vee b \vee e$, $V>0$, $D\ge 1$ and $0<u\le \sigma$. Also assume that $D(\log{(B/\sigma)})^V\lesssim n$. Then the following conclusion holds:
     \begin{align}\label{eq:vcbd1}
        \EE\sup_{f\in\cF}|\mathbb{G}_n(f)|\lesssim \sigma \left(n D^{\gamma}(\log{(B/\sigma)})^{V\gamma}\right)^{\frac{1}{2(\gamma+1)}}+n^{\frac{1}{2}-\frac{\gamma}{\gamma+1}}\left(D(\log{(B/\sigma)})^V\right)^{\frac{\gamma}{\gamma+1}}.
    \end{align}
    In particular if $\sigma\gtrsim 1$, then 
    \begin{align}\label{eq:vcbd2}
        \EE\sup_{f\in\cF}|\mathbb{G}_n(f)|\lesssim n^{\frac{1}{2(\gamma+1)}}D^{\frac{\gamma}{2(\gamma+1)}}(\log{(B)})^{\frac{V\gamma}{2(1+\gamma)}}.
    \end{align}
\end{cor}

It is easy to check from \cref{thm:infmainres} that \eqref{eq:vcbd2} holds even under the weaker $\beta$-mixing assumption. However \eqref{eq:vcbd1} is stronger as it tracks the size of the $L_2$ radius of $\cF$. As the $L_r$ norm is larger than the $L_2$ norm for $r\ge 2$, \eqref{eq:vcbd1} cannot be derived from \cref{thm:infmainres}. In the following Section, we will see how tracking the size of this $L_2$ radius can improve rates of convergence in learning theory by a certain ``localization" argument.

\begin{rem}\label{rem:setstruc}
    The case of $r=1$ in \cref{thm:gamma_mixing} is also of interest in some ``set-structured" problems (a term we borrow from \cite{han2021set}, also see \cite{kur2019optimality}). We will see in \cref{sec:fastratestruc} that in such problems, the $r=1$ version of \cref{thm:gamma_mixing} can be used to derive faster rates than a direct $L_2$-bracketing entropy bound.
\end{rem}
\section{Faster rates with $L_r$-bracketing, $1\le r \le 2$}\label{sec:fastratestruc}
While our focus so far has mostly been on non-Donsker classes of functions, \cref{thm:gamma_mixing} can be used to get faster rates in various structured statistical inference problems. We provide three such ways in this section: (a) In \cref{sec:localrate}, we provide a rate theorem with $L_2$ localization for empirical risk minimization (ERM) problems, (b) In \cref{sec:adaptcom}, we provide adaptation bounds which yield faster rates when the underlying true DGP is simple, and (c) In \cref{sec:setstruc}, we provide faster bounds for \eqref{eq:target} when the functions in $\cF$ have ``structured" level sets.


\subsection{Rate theorem and localization in learning theory  under dependence}\label{sec:localrate}

Learning theory typically refers to a genre of problems where we establish a bound on the generalization error of a classifier/regressor trained on the training data. Broadly speaking, suppose we have a closed, convex collection of hypotheses (functions) $\cF$, a loss function $\ell(\cdot, \cdot)$ and $n$ data points $\{(X_1, Y_1), \dots, (X_n, Y_n)\}$ where each $(X_i,Y_i)\sim P$ with the marginal distribution of being $P_X$ and $P_Y$. The aim is to find a predictor of $Y$ given $X$. 
This typically involves minimization of the risk function associated with a loss function $\ell$, i.e.,   $f_\star:=\argmin_{f\in\cF} \EE \ell(Y,f(X))$. 
In the sequel, we will always assume such a $f_\star$ exists. A natural way to estimate $f_\star$ based on the observed sample is via ERM: 
\begin{equation*}
\hat{f}_n \in \argmin_{f \in \cF} \,\, \frac1n \sum_i \ell(Y_i, f(X_i)) \,.
\end{equation*}
The statistical risk of $\hat{f}_n$ is usually quantified via  the \emph{generalization error} or \emph{excess risk} $\cE(\hat{f}_n)$ which is defined as:  
\begin{equation}
\label{eq:excess_risk}
\cE(f) := \bbE_{P}[\ell(Y, f(X))] - \EE_{P}[\ell(Y, f_\star(X))],\quad \quad f\in\cF\,.
\end{equation}
As an example, consider a standard nonparametric regression model with an additive noise: 
\begin{equation}\label{eq:nonpiid}
Y = f_\star(X) + \xi, \ \ \text{ with } \ \ \bbE[\xi \mid X] = 0 \,.
\end{equation}
In this case, the generalization error for squared error loss simplifies to: 
\begin{equation}
\label{eq:quad_curvature}
\mathcal{E}(\hat{f}_n) = \EE_{P}[(Y - \hat{f}_n(X))^2] - \EE_{P}[(Y - f_\star(X))^2] =\EE_{P_X}[(\hat{f}_n(X) - f_\star(X))^2] \,.
\end{equation}

Typically the observations $\{(X_i, Y_i)\}_{i\in [n]}$ are assumed to be independent to derive bounds on $\cE(\hat{f}_n)$. 
However, much attention has been devoted in recent years to weakening the independence assumption and replacing it with \emph{exponentially decaying} mixing conditions; see \cite{kengne2023deep,ma2022theoretical,roy2021empirical,kurisu2022adaptive}. Here we show that our general theory of empirical processes can be used to extend generalization bounds in learning theory to much weaker polynomial dependence. 

We start with the following preliminary bound on the generalization error that follows from \cref{thm:infmainres}. To convey the main message of this section, we will focus only on the short-range regime of dependence. Define a new function class $\cG$  as: 
\begin{equation}
    \label{eq:loss_f_def}
    \cG := \{\ell(y, f(x)) - \ell(y, f_\star(x)): f \in \cF\} \,,
\end{equation}
The following corollary yields a rate of convergence for the generalization error when $\cG$ is relatively simple: 
\begin{cor}
\label{cor:gen_error}
Suppose that \cref{assm:dgp} holds with $\beta_k\le (1+k)^{-\beta}$ for some $\beta>1$. Let $\cG\equiv \cG_{\sigma,b}$ be the class of functions in \eqref{eq:loss_f_def} with $\lVert g\rVert_{L_2(P)}\le \sigma$ and $\lVert g\rVert_{\infty}\le b$. Also assume $\cG$ is a VC-type class of functions with VC dimension $D$, i.e., $L_\infty$ bracketing number of $\cG$ satisfies $\benfg{u}\le D(\log(B/u))^V$ for $u\le b \le B$ and some $V>0$. Then we have $$\EE[\cE(\hat{f}_n)]\lesssim \sqrt{\frac{D}{n}},$$
up to some polylogarithmic factor, where the implied constant depends on $b$, $C$, $\beta$, and $V$. 
\end{cor}

Typically for most interesting loss functions like the squared error loss, logistic loss, hinge loss, etc., the condition on the entropy of $\cG$ can be replaced with the same on the entropy of $\cF$ with respect to $L_2(P_X)$ norm. 
The problem with writing a general result with an entropy bound on $\cF$ is that standard contraction inequalities on Rademacher complexity (see \cite{mendelson2003few}) can no longer be applied due to the lack of exchangeability in a mixing DGP, unless we put strong technical assumptions on $\ell(\cdot,\cdot)$. 

Suppose $\cG$ is a parametric class, i.e., it can be indexed by $\Theta\subset \R^k$ for some $k\le d$, then the above corollary typically gives a rate of $\sqrt{k/n}$ on the generalization error for all $\beta$ (as $D=k$ in such settings). However, at least for i.i.d. data from model \eqref{eq:nonpiid}, it is well known from standard parametric model theory; see \cite[Chapter 3.2]{van1996weak}, that the faster rate $k/n$ is achievable under minimal assumptions. This raises the following natural questions: \emph{Can rates faster than $\sqrt{k/n}$ be achieved under dependence? 
Is it possible to approach the $k/n$ rate as the dependence gets weaker?} 
In the rest of this section, we will answer these questions in the affirmative under the $\gamma$-mixing assumption. The crucial tool will be \cref{thm:gamma_mixing} which will enable us to get faster rates via localization.

Towards that end, we use Theorem 3.2.1 of \cite{van1996weak}, which we state here for the ease of the reader. 
\begin{theorem}[Rate theorem with localization]
\label{thm:roc}
    Suppose that the observed data \\ $(X_1,Y_1),\ldots , (X_n,Y_n)$ is strictly stationary $\gamma$-mixing with mixing coefficients $\gamma_q\le C(1+q)^{-\gamma}$ for $\gamma>0$. Consider $\cG_{\sigma,b}$ as in \cref{cor:gen_error}. Further suppose that $\ell(\cdot,\cdot)$ is Lipschitz and the excess risk $\cE(f)$ (defined in \eqref{eq:excess_risk}) satisfies a quadratic curvature conditions $\cE(f) \ge c \lVert f-f_\star\rVert^2_{L_2(P_X)}$ for some $c\in (0,\infty)$ and all $f\in\cF$. Define the local $L_2(P)$ ball around $f_\star$ as follows:
    \begin{align}\label{eq:modcont}
    \cG_{\delta}:=\{g:\ g\in\cG,\, \lVert g\rVert_{L_2(P)}\le \delta\}.
    \end{align}
    Recall the definition of $\Pi_n(\cdot)$ from \cref{thm:gamma_mixing}. If $\Pi_n(\cG_{\delta})/\delta^{t}$ is non-increasing for some $0<t<2$ and if $\{\delta_n\}$ is a sequence that satisfies: $\Pi_n(\cG_{\delta_n}) \le \sqrt{n}\delta_n^{2}$, then for any $\lambda<2-t$, we have $$\EE\big| \cE(\hat{f}_n)\big|^{\frac{\lambda}{2}}\lesssim \delta_n^{\frac{\lambda}{2}}.$$
\end{theorem}

The assumptions on the boundedness of $\cG$ typically will not cover unbounded responses or covariates. It is possible however  to use truncation arguments to cover the unbounded setting (see \cite{roy2021empirical,mendelson2015learning}). However as our focus is to extract the impact of dependence in the DGP, we avoid digressions about unbounded data here. The quadratic curvature condition on $\cE(f)$ is also a standard assumption prevalent in the literature to get faster rates. By \eqref{eq:quad_curvature}, it holds for quadratic loss. It also holds for the binomial/multinomial logistic loss function (see \cite[Lemmas 8 and 9]{farrell2021deep} for details).

We are now in a position to improve on \cref{cor:gen_error} by leveraging Theorems \ref{thm:gamma_mixing} and \ref{thm:roc}.

\begin{cor}
\label{cor:fast_error}
Suppose that the observed data $(X_1,Y_1),\ldots , (X_n,Y_n)$ is strictly stationary $\gamma$-mixing with mixing coefficients $\gamma_q\le C(1+q)^{-\gamma}$, for some $\gamma>1$. The other conditions in \cref{cor:gen_error} are also assumed to hold. Then we have $$\cE(\hat{f}_n)=O_p\left(\left(\frac{D}{n}(\log{(Bn)})^V\right)^{\frac{\gamma}{\gamma+1}}\right).
$$
\end{cor}

As $\gamma>1$ implies $\gamma/(\gamma+1)>1/2$, \cref{cor:fast_error} already provides a better rate than \cref{cor:gen_error} for all $\gamma>1$, thereby showing that rates faster than $1/\sqrt{n}$ are achievable using localization even under dependence. Finally as $\gamma\to\infty$, $\gamma/(\gamma+1)\to 1$. Therefore informally, as $\gamma\to\infty$, the rate for estimating $f_\star$ is approximately $D/n$ which is the optimal rate in the i.i.d. setting. 

\subsection{Adaptation bounds for complex function classes under mixing assumptions}\label{sec:adaptcom}
In the previous Section, we saw how localization can improve rates of convergence under a $\gamma$-mixing assumption for ``low complexity" function classes $\cF$. In contrast, here we will focus on faster rates for complex function classes. We consider the same learning theory setup as in \cref{sec:localrate}. Recall the definition of the local ball $\cG_{\delta}$ from \eqref{eq:modcont}. We see the benefit of adaptation when the $L_2(P_X)$ entropy of $\cG_{\delta}$ shrinks with the local radius $\delta$. For example, if 
\begin{align}\label{eq:locradent}
\bentgd{u}\le K \left(\frac{\delta}{u}\right)^{\alpha}\left(\log{\left(\frac{B}{u}\right)}\right)^V,
\end{align}
where $0<u\le \delta$, $B\ge \delta \vee b \vee e$, and $K>0$ is a constant free of $\delta$ and $u$. Note the dependence on $\delta$ in \eqref{eq:locradent}. Such control on the entropy is usually achieved when $\cG$ is generated by shape constrained function classes $\cF$ when the center of the ball, i.e., $f_\star$ is ``simple". For example, this happens (up to  log factors) when $\cF$ is the space of multivariate isotonic functions (see \cite{hanisotonic2019}) and $f_\star$ is a constant function, or when $\cF$ is the space of multivariate convex functions(see \cite{kur2019optimality}) and $f_\star$ is a linear function. We provide rigorous descriptions of such a class in \cref{sec:app}. 

The following corollary of \cref{thm:gamma_mixing} and \cref{thm:roc} provides faster rates for function classes satisfying \eqref{eq:locradent}.

\begin{cor}\label{cor:adaptbd}
    Consider the same assumptions on the DGP, $\ell(\cdot,\cdot)$ and $\cE(\cdot)$ as in \cref{thm:roc}. Also assume that \eqref{eq:locradent} holds. Then, if $\alpha>2(1+\gamma^{-1})$, we have:
    $$\lVert \hat{f}_n-f_\star\rVert^2_{L_2(P_X)}=O_p\left(n^{-\frac{2}{\alpha}}(\log{(Bn)})^V\right).$$
\end{cor}

Note that for general function classes, a condition such as \eqref{eq:locradent} may not be satisfied and the best possible bounds for $\bentgd{u}$ are often of the form $u^{-\alpha}$ instead of $(\delta/u)^{\alpha}$. In such cases, localization is not possible and \cref{thm:gamma_mixing} only implies a rate of $n^{-\frac{1}{\alpha}}$ (following \cref{cor:gen_error}). Therefore, by \cref{cor:adaptbd}, property \eqref{eq:locradent} results in the faster rate $n^{-\frac{2}{\alpha}}$ under a $\gamma$-mixing DGP provided $\alpha>2(1+\gamma^{-1})$. This condition can be rewritten as $\gamma>2/(\alpha-2)$. As $2/(\alpha-2)<1$ for $\alpha>4$, we once again see the phenomenon that \emph{even in the long-range dependence regime $\gamma\in (2/(\alpha-2),1)$, we get the same adaptation bound of $n^{-\frac{2}{\alpha}}$ that we expect in the i.i.d. setting}.

\subsection{Set-structured problems with $L_1$-bracketing and mixing assumptions}\label{sec:setstruc}
In this section, we will discuss another class of problems, informally referred to as ``set-structured" problems (see \cite{han2021set,kur2019optimality}), where it is possible in the i.i.d. setting to achieve faster rates in the non-Donsker regime, than by direct entropy based arguments on function classes. We will show here that the same phenomenon persists under dependence via $\gamma$-mixing assumptions. To motivate this class of problems, we begin with a simple observation. Consider the following metric entropy bounds:
\begin{align}\label{eq:coment}
\bent{u}\le K u^{-2\alpha} \quad \mbox{and}\quad \beno{u^2}\le K u^{-2\alpha},
\end{align}
for some function class $\cF$ where $\lVert f\rVert_{L_2}\le\sigma$ and $\lVert f\rVert_{\infty}\le b$ for all $f\in\cF$. As $\lVert f\rVert_{L_2}\le b \lVert f\rVert_{L_1}$, the latter bound is stronger (i.e., any admissible $L_1$ bracket is also an admissible $L_2^2$ bracket). It is easy to check from \cref{thm:gamma_mixing} (also see \cref{cor:nonpbd}) that for $\alpha>1+\gamma^{-1}$, the bound for \eqref{eq:target} under the two entropy bounds above will respectively be $n^{(\alpha-1)/2\alpha}$ and $n^{1/2-1/(\alpha+1+\gamma^{-1})}$. In the $\alpha>1+\gamma^{-1}$ regime, the latter bound, as one would expect, is clearly faster than the former. Now, an important class of functions where the two conditions above are equivalent is when $\cF=\{\mathbf{1}(C):\ C\in\mathcal{C}\}$ where $\mathcal{C}$ is some class of measurable sets. This is because the $L_1$ norm of indicators is the same as the $L_2^2$ norm which implies that any admissible $L_2^2$ bracket is now an admissible $L_1$ bracket. This suggests that whenever \eqref{eq:target} can be further bounded by supremum over indicators of appropriate sets, it may be possible to get faster rates in the non-Donsker regime. Informally ``set-structured" problems are those examples of $\cF$ where  bounding \eqref{eq:target} can be reduced to obtaining uniform convergence bounds for 
$$\sup_{C\in\mathcal{C}} |(\mathbb{P}_n-P)(C)|$$
for an appropriate class of measurable sets $\mathcal{C}$ (depending on $\cF$), having the same $L_1$ bracketing entropy as the $L_2$ bracketing entropy of $\cF$. The following corollary provides bounds on \eqref{eq:target} for such a candidate class $\mathcal{C}$.

\begin{cor}\label{cor:fasterl1}
    Consider the same DGP as in \cref{thm:gamma_mixing}. For $0< t< b$ and $f\in\cF$, define 
    $$A_{t,f}^+:=\{x:f_+(x)>t\}\quad \mbox{and}\quad A_{t,f}^-:=\{x:f_-(x)>t\}.
    $$
    where $f_+(x)=\max\{f(x),0\}$ and $f_-(x)=\max\{-f(x),0\}$. 
    Let $\mathcal{C}$ be a family of subsets of $\R^d$ such that, for all $0<t<b$, $f\in\cF$, we have $A_{t,f}^+\in\mathcal{C}$, $A_{t,f}^-\in \mathcal{C}$ and \begin{align}\label{eq:setent}
    \benoc{u^2}\le K u^{-2\alpha}, 
    \end{align}
    where $\sigma^2=\sup_{C\in\mathcal{C}} \mathbb{P}(C)$. 
    Then the following bound holds on \eqref{eq:target} for $\alpha>1+\gamma^{-1}$:
    $$\EE\sup_{f\in\cF}|\mathbb{G}_n(f)|\lesssim n^{\frac{1}{2}-\frac{1}{\alpha+1+\gamma^{-1}}}.$$
\end{cor}

The benefits of \cref{cor:fasterl1} can be seen in many ERM applications. For example, if $\cF$ denotes the class of multivariate isotonic functions \cite{hanisotonic2019} on $[0,1]^d$ and $P=\mbox{Unif}[0,1]^d$, then $\mathcal{C}$ can be taken as the collection of upper and lower sets on $[0,1]^d$. In this case, the entropy bound in the LHS of \eqref{eq:coment} holds (see \cite[Lemma 3]{hanisotonic2019}) with $\alpha=d-1$ (up to log factors) which yields a rate of $n^{1/2-1/(2(d-1))}$ for $d>2+\gamma^{-1}$ by \cref{cor:nonpbd}. On the other hand, \eqref{eq:setent} is also satisfied with $\alpha=d-1$ (see \cite[Theorem 8.22]{dudley2014uniform}). Therefore, \cref{cor:fasterl1} implies a bound of $n^{1/2-1/(d+\gamma^{-1})}$ which is strictly faster than $n^{1/2-1/(2(d-1))}$ for $d>2+\gamma^{-1}$. Also note that as $\gamma\to\infty$, we approximately recover the $n^{1/2-1/d}$ rate which is the best achievable rate in the i.i.d. setting (see \cite[Proposition 1]{hanisotonic2019}). A similar phenomenon also holds for the class of bounded multivariate convex functions (see \cref{rem:samephencon}).
\section{Applications}
\label{sec:app}
We look at five applications in this section under a strictly stationary mixing DGP: (a) estimation of smooth functions in fixed dimension using deep neural networks (see \cref{sec:nnet}), (b) Estimation of additive model in growing dimension via deep neural networks(see \cref{sec:addreg}) (c) Shape restricted multivariate convex regression and adaptation (see \cref{sec:non_param_reg_application}), (d) Estimation of the $2$-Wasserstein distance with and without entropic regularization (see \cref{sec:OT}), and (e) Classification under the Mammen-Tsybakov margin condition (see \cref{sec:classific}).

\subsection{Smooth function estimation using deep neural networks}
\label{sec:nnet}
In this subsection, we elaborate on the effect of dependence on the estimation of a regression function via deep neural networks (DNNs). The usage of DNNs in modern statistical analysis and machine learning is ubiquitous due to their representation power as well as approximation ability. The statistical properties of DNN based estimators for an underlying regression function in a nonparametric additive noise model, have been investigated in detail in a series of papers, e.g. \cite{kohler2005adaptive, kohler2016nonparametric, bauer2019deep, schmidt2020nonparametric, kohler2021rate, bhattacharya2023deep,farrell2021deep}. Most of previous research is primarily based on i.i.d. setup and there is a paucity of theoretical analysis of DNN estimators in the presence of data dependency. Some exceptions include the following: \cite{truong2022generalization} analyzed the generalization error under Markovian structure and hidden Markov model, \cite{ma2022theoretical} considered nonparametric regression function estimation for exponentially $\alpha$-mixing time series data, \cite{kengne2023excess} used a particular form for $\psi$-dependency (Definition 2.1 and Assumption A.3) as the data generating process and bound the generalization error via DNNs, and \cite{kurisu2022adaptive} provides minimax optimal rates for nonparametric regression under exponential $\beta$-mixing. In contrast, our focus is to  to provide explicit quantitative bounds beyond exponential mixing.
\\\\
\noindent
We define a class of neural networks as: 
$
\cF_{\mathrm{NN}}(d_{\mathrm{in}}, d_{\mathrm{out}}, N, L, W, B)
$
to be the set of all DNNs with input dimension $d_{\text{in}}$, output dimension $d_\text{out}$, width $N$ (i.e. a maximum number of neurons in each hidden layer), depth $L$, the total number of non-zero/active weights $W$ with all weights bounded by $B$. Whenever we use a fully connected neural network, we drop the parameter $W$ as it is a function of $(d_{\mathrm{in}}, d_{\mathrm{out}}, N, L)$. We also use the ReLU activation function $\sigma(\cdot):=\max\{\cdot\,,0\}$. We defer the reader to \cref{sec:nnproof} for a more thorough discussion on the standard neural net architecture.

Consider the nonparametric regression model with additive noise as in \eqref{eq:nonpiid} where $Y_i = f_\star(X_i) + \xi_i$. We assume that the errors are bounded, i.e., $|\xi| \le b_e$ for some $b_e>0$.

\begin{assumption}\label{asn:finitedim}
    Let $\Sigma(s, C)$ denote the space of all Hölder functions $f$ which are $s$ times differentiable with 
$$
\max\left\{\sup_{\|\balpha\|_1 \le \lfloor s \rfloor} \left\|\partial^\balpha f\right\|_\infty, \sup_{\|\balpha\| = \lfloor s \rfloor} \sup_{x,. y}\frac{\left|\partial^\balpha f(x) - \partial^\balpha f(y)\right|}{\|x - y\|^{s - \lfloor s \rfloor}} \right\} \le C \,.
$$
We assume that $f_\star\in\Sigma(s,C)$ for some $,C>0$ and there exists some known $b>0$ such that $\lVert f_\star\rVert_{\infty}\le b$.
\end{assumption}

 It is possible to relax these boundedness assumptions with truncation techniques under sub-Gaussian tails. However we primarily focus on extracting the effect of dependence in this paper and thus leave the task of relaxing boundedness assumptions for future work.   
We approximate $f_\star$ by a collection of neural networks with width $N$ growing with $n$ at an appropriate rate (to be specified later) and depth $c_1$ some fixed positive constant.  As $\|f_\star\|_\infty \le b$, we further truncate our neural networks at level $2b$. Let us denote this class of functions as $\tilde{\cF}_{\mathrm{NN}}(d_{\mathrm{in}},d_{\mathrm{out}},N,L,W,B)$ (as $b$ is assumed known, we suppress the dependence on $b$). More precisely our estimator for $f_\star$ is defined as: 
$$
\hat f_n := \argmin_{f \in \tilde{\cF}_{\mathrm{NN}}(d_{\mathrm{in}},d_{\mathrm{out}},N,L,W,B)} \frac1n \sum_i \left(Y_i - f_\star(X_i)\right)^2
$$
The following theorem establishes the asymptotic property of $\hat f$: 

\begin{theorem}
    \label{thm:dnn_mixing}
    Suppose $\{(X_i, \xi_i)\}_{i=1}^n$  is a strictly stationary $\gamma$-mixing sequence with $\gamma_q \le  c(1+q)^{-\gamma}$ for some $c,\gamma > 0$. Also assume that \cref{asn:finitedim} holds. Suppose we estimate $f_\star$ via a feed-forward fully connected deep neural network with $d_{\mathrm{in}}=d$, $d_{\mathrm{out}}=1$, width $N\sim (n/log{n})^{\frac{1}{2\left(1 + \frac{2s}{d}\left(\frac{1+\gamma}{\gamma}\right)\right)}}$, depth $L$, a fixed constant $>1$, and bound $B=N^{c_2}$ for some constant $c_2>0$. We then have: 
    $$
    \|\hat f_n - f_\star\|_2 = O_p\left(\left(\frac{n}{\log{n}}\right)^{-\frac{s}{d + 2s\left(\frac{1 + \gamma}{\gamma}\right)}}\right) \,.
    $$
\end{theorem}

It is evident from the theorem that the price we pay for the dependency is the factor $(1+\gamma)/\gamma$ in the rate at the exponent of $n$. As $(1+\gamma)/\gamma \to 1$ as $\gamma \uparrow \infty$, we recover the standard nonparametric rate $n^{-s/(2s + d)}$ (modulo log factor) when $\gamma \uparrow \infty$, i.e. for the independent case (modulo logarithmic factors). 

\begin{rem}
    In the above theorem, we assume the estimator has finite depth (although the depth has to be larger than 1, see the discussion preceding Theorem 4 of \cite{fan2023factor}) and the width grows with $n$. One can go the other way round, i.e. grow the depth with $n$ while keeping the width fixed. However, for that, we need to use other approximation results, e.g., \cite[Theorem 4.1]{yarotsky2020phase}. 
    One may also increase both the width and depth with $n$ by utilizing an alternative approximation theorem, such as  \cite[Theorem 1.1]{lu2021deep}.
\end{rem}

\begin{rem}
    Recently, in a series of papers, \cite{kengne2023deep, kengne2023excess, kengne2023sparse}, researchers have explored the generalization error of deep neural networks for regression and classification. Our results are different from them in two aspects: i) the previous works mainly assume either exponentially decaying $\psi$-mixing (closely related to $\rho$-mixing) or at least it decays faster than $x^{-2}$, whereas our results are applicable for all $\gamma > 0$. ii) Those papers presented their results on a broad level, it is not immediate to tease out the exact rate or convergence under Holder smoothness assumption on the true mean function, whereas our results are more explicit in terms of quantifying the rate in terms of $(s, d, \gamma)$ where $s$ is the smoothness index, $d$ is the input dimension and $\gamma$ is the dependence coefficient. 
\end{rem}

\subsection{Additive model regression in growing dimension via deep neural networks}\label{sec:addreg}
The modern era of big data with high dimensionality necessitates the development of statistical theory when $d$ is growing with $n$. Therefore, in contrast to the previous subsection, here we give an example of a structured nonparametric regression model in diverging dimensions, namely the additive model which is one of the more popular structures in the literature. The model stipulates the following DGP: 
\begin{equation}
\label{eq:additive_model}
Y_i = \sum_{j = 1}^d f_{\star,j}(X_{ij}) + \xi_i \,.
\end{equation}
where $X = (X_1, \dots, X_d) \in [0, 1]^d$ and $\xi \indep X$. Here, each $f_j$ is an univariate function. 
The key benefit of the additive model lies in simplifying the complexity of the underlying function class. In other words, the collection of additive functions defined on $\R^d$ is considerably ``smaller" than the set encompassing all functions on $\R^d$. This simplicity of additive models is reflected in their rate of estimation. \cite{stone1985additive} proved that for an additive regression model, if each univariate component $f_j$ is $\beta$-Hölder, then the minimax optimal rate for estimating $f$ is $C_d n^{-\frac{2\beta}{2\beta + 1}}$ (for fixed dimension, which is absorbed in the constant). This minimax rate reveals that the additive model can evade the curse of dimensionality to a large extent as $d$ does not appear in the exponent of $n$. 
Along with its various generalizations, this model inspires a line of research that encompasses but is not limited to \cite{stone1994use, hastie1987generalized, buja1989linear, fan1998direct, horowitz2007rate}. Recently, efforts have been to made to under the estimation error of additive model and its generalization via deep neural networks in several paper (e.g., see \cite{bauer2019deep}, \cite{kohler2016nonparametric}, \cite{kohler2021rate}, \cite{schmidt2020nonparametric} and reference therein). 

The case of $d$ growing with $n$ is much less understood. Now there is a need to track the precise effect of $d$ in the constant $C_d$. When $d/n \to 0$ but $d$ slowly increases with $n$, \cite{sadhanala2019additive} established that the minimax optimal rate of estimation of the additive mean function $f_\star$ is $dn^{-\frac{2\beta}{2\beta + 1}}$ (where each component is $\beta$-smooth) when the coordinates of $X$ are independent, i.e. the optimal factor for dimension we pay is the multiplicative factor $d$. High dimensional additive model (when $d \gg n$) along with some sparsity assumption (namely \emph{sparse additive model}) has also been investigated; initially \cite{raskutti2012minimax} established the minimax optimal rate. Subsequent investigations in the literature have delved deeper into sparse additive models (e.g., see \cite{lin2006component, koltchinskii2010sparsity, ravikumar2009sparse, yuan2016minimax} and references therein), expanding upon recent advancements in penalized linear regression. Lastly, \cite{tan2019doubly} suggests a dual penalty approach, involving the empirical norm and functional semi-norm of each component $f_{\star,j}$ to estimate the regression function. Very recently, \cite{bhattacharya2023deep} extended the analysis for a high dimensional $k$-way sparse interaction model using deep neural networks. However, to the best of our knowledge all the aforementioned results crucially leverage the independence in the DGP and hence cannot be applied to dependent DGPs under appropriate mixing assumptions. To the best of our knowledge, \emph{these are the first results analyzing analyze neural nets in growing dimension under a mixing DGP}. The focus here is on the case when $d$ grows with $n$ but at a slower rate (see \cite{sadhanala2019additive}). The analysis for a high dimensional sparse additive model is left to future research.

We assume the additive noise model \eqref{eq:additive_model} where $\{(X_i, \xi_i)\}_{1 \le i \le n}$ forms a $\gamma$-mixing sequence. Also suppose that the $\xi$s are bounded. We now present two sets of assumptions --- one on the design and the other on the additive function class.

\begin{assumption}\label{asn:design}
    We assume that the covariate $X$ is compactly supported on, say $[0,1]^d$ with continuous and bounded density.
\end{assumption}

\begin{assumption}\label{asn:growingdim}
Define the function class  $\cF_{\add}(d, b, s)$ to be set of all functions $f: [0,1]^d \mapsto \R$ such that: 
\begin{enumerate}
    \item $f$ is additive, i.e. $f(x) = \sum_j f_j(x_j)$ 
    \item Each component $f_j:\R\to\R$ and $f_j\in \Sigma(s, b)$ i.e. $s$-Hölder function class with Hölder norm bounded by $b$.
    \item $\int_{[0,1]} f_j(x)\,dx = 0$ for all $1 \le j \le d$. 
    \item The function $f$ is uniformly bounded by $b$.  
\end{enumerate}   
Hereafter, we assume that $f_\star \in \cF_\add(d, b, s)$. 
\end{assumption}
 The above conditions are quite standard in the additive model literature and we refer the interested reader to \cite[Section 2]{bhattacharya2023deep} and \cite[Section 2]{koltchinskii2010sparsity} for similar assumptions in related work. Crucially, they ensure identifiability of the additive model.
 
We here estimate each component of $f_\star$ by $d$ separate fully connected neural networks each of which has the same  structure. To that end, we define 
$$\cF_{\mathrm{NN},\mathrm{add}}(d,1,N,L,B):=\left\{\sum_{j=1}^d f_j:\ f_j\in\cF_{\mathrm{NN}}(1,1,N,L,B)\right\}.$$
Finally, we write $\tilde{\cF}_{\mathrm{NN},\mathrm{add}}(d,1,N,L,B)$ to be the subset of networks truncated at $2b$. We then define the estimator as follows: 
$$
\hat f_n = \argmin_{f \in \tilde{\cF}_{\mathrm{NN},\mathrm{add}}(d,1,N,L,B)} \ \left\{\frac1n \sum_i \left(Y_i - f(X_i)\right)^2\right\} \,.
$$
Estimation of additive models in high dimension with i.i.d. sample has been recently explored by \cite{bhattacharya2023deep}. The following theorem complements their result by establishing the rate of convergence of $\hat f$ in the presence of dependency among the observations: 
\begin{theorem}
    \label{thm:additive_rate}
    Suppose $(X_i,\xi_i)$ is generated from a strictly stationary $\gamma$-mixing sequence with $\gamma_q\le c(1+q)^{-\gamma}$ for some $c,\gamma>0$. Also assume that Assumptions \ref{asn:design} and \ref{asn:growingdim} hold. Let $d\le N^s$ for some $0<s<1$. Then the least-square estimator $\hat f$ satisfies: 
    $$
    \|\hat f_n - f_\star\|_2^2 = O_p\left(d^{1 - \frac{1}{\gamma}\frac{2s}{2s\left(\frac{1 + \gamma}{\gamma}\right) + 1}}n^{-\frac{2s}{2s\left(\frac{1 + \gamma}{\gamma}\right) + 1}}\left(\log{(n)}\right)^{\frac{\gamma}{(1 + \gamma)}}\right) \,.
    $$
\end{theorem}
Let us now compare the rate obtained here with that for the i.i.d. case which is $dn^{-2s/(2s + 1)}$. 
Taking $\gamma \uparrow \infty$, the rate obtained in Theorem \ref{thm:additive_rate} coincides with the i.i.d. rate and consequently we can recover the minimax optimal rate of estimation for independent observations in the limit. For a fixed $\gamma < \infty$, we unravel an interesting dichotomy; although the power of $n$ is compromised due to dependency (it is smaller than what it would have been for the i.i.d. case), the power of $d$ is boosted as it is strictly $< 1$. At first glance, it may be tempting to think that there is some sort of blessing of dimensionality, i.e. higher the dependence, the lower the effect of dimension on the estimation. However, as we argue now this is not the case. For simplicity, assume $d = n^\alpha$ for some $\alpha < 1$. Then (ignoring the log factor), Theorem \ref{thm:additive_rate} provides the following rate of convergence: 
$$
\|\hat f_n - f_0\|_2^2 \approx O_p\left(n^{\alpha - \frac{2s}{2s\left(\frac{1 + \gamma}{\gamma}\right) + 1}\left(1 + \frac{\alpha}{\gamma}\right)}\right) = O_p\left(n^{- \frac{2s(1 - \alpha) - \alpha}{2s\left(\frac{1 + \gamma}{\gamma}\right) + 1}}\right) \,.
$$
Now, under the same setup, the rate for under i.i.d. observations simplifies to:
$$
\|\hat f_n - f_0\|_2^2 \approx O_p(dn^{-\frac{2s}{2s + 1}}) = O_p\left(n^{-\frac{2s(1 - \alpha) - \alpha}{2s + 1}}\right) \,.
$$
Therefore it is immediate that we overall pay a price for the dependency through a factor of $(1+\gamma)/\gamma$ in the rate of convergence. Whether this factor is optimal or not is an intriguing question and is beyond the scope of this paper. 

\subsection{Shape constrained multivariate convex least squares}
\label{sec:non_param_reg_application}
Consider a strictly stationary sequence $(X_1,\xi_1),\ldots , (X_n,\xi_n)$ satisfying $\beta/\gamma$-mixing property (specific assumptions to be provided in the results), where marginally $X\sim P_X$ has a continuous density bounded away from $0$ and supported on some compact, convex body $\Omega\subset \R^d$, and $\xi_i$s are assumed bounded. We observe data from the regression model $Y_i=f_\star(X_i)+\xi_i$, $\EE[\xi_i|X_i]=0$, see \eqref{eq:nonpiid}, where $f_\star$ is a convex function on $\Omega$ and is bounded by $1$. The goal of this section is to \emph{estimate $f_\star$ using a natural least squares approach (below) and quantify the rate of convergence under mixing assumptions} instead of the well understood i.i.d. setting. 

Nonparametric convex regression has a long history in the statistics \cite{hildreth1954point,hanson1976consistency,Min2016,groeneboom2001estimation} and economics \cite{kuosmanen2008representation,matzkin1991semiparametric,varian1982nonparametric,banker1992maximum} literature. Our estimator of choice is the bounded convex least squares (BCLS) estimator \cite{han2016multivariate,Seijo2011,guntuboyina2015global} defined as follows: 
$$\hat{f}_n\in\argmin_{f\in \cF}\,\, \frac{1}{n}\sum_{i=1}^n (Y_i-f(X_i))^2,$$
where $\cF$ is the class of convex functions on $\Omega$ bounded by $1$. The above definition fixes $\hat{f}_n$ only at the design points. It can be extended to the whole of $\omega$ by an interpolation technique; see \cite[Equations 3.2---3.4]{han2016multivariate} and \cite[Lemmas 2.3 and 2.4]{Seijo2011}; also see \cref{sec:pfshapeconv}. While the computational and statistical properties of $\hat{f}_n$ are well understood in the i.i.d. setting, we will show that interestingly the same bounds can be achieved even under long-range dependence provided $d$ is large enough. To wit, we focus on polytopal domains $\Omega$, i.e., 
$\Omega=\{x\in\R^d:\, a_i\le v_i^{\top}x\le b_i,\, \mbox{for}\, i=1,2,\ldots ,F\},$
for some positive integer $F$ (not changing with $n$), unit vectors $v_1,\ldots ,v_F$, and real numbers $a_1,\ldots , a_F,b_1,\ldots , b_F$. A typical example is $\Omega=[0,1]^d$.

\begin{theorem}\label{thm:shapeconv}
 (a). (Worst case risk) Assume that the DGP is strictly stationary $\beta$-mixing with $\beta_q\le c(1+q)^{-\beta}$ for some $c\ge 0$, $\beta>0$. 
    Then provided $\Omega$ is a polytope as described above, $d>4$ and $\beta\in (2/(d-2),\infty)\setminus\{1\}$, we have:
    \begin{align}
        \EE\lVert \hat{f}_n-f_\star\rVert^2_{L_2(P_X)}\lesssim n^{-\frac{2}{d}}(\log{(n)})^{d+1}.
    \end{align}

    (b). (Adaptation) Assume that the DGP is strictly $\gamma$-mixing with $\gamma_q\le c(1+q)^{-\gamma}$ for some $c\ge 0$ and $\gamma>0$. Suppose $d>8$, $\gamma>(4/(d-4),\infty)\setminus \{1\}$, and $f_\star$, $\Omega$ satisfy the following property: there exists $k$ simplices $\Delta_1,\ldots ,\Delta_k$ such that $\cup_{j=1}^k \Delta_j$ and $f_\star$ is affine on each $\Delta_j$.
    we then have
    \begin{align}
        \lVert \hat{f}_n-f_\star\rVert^2_{L_2(P_X)}\,=\, O_P\left(n^{-\frac{4}{d}}(\log{(n)})^{d+1}\right).
    \end{align}
\end{theorem}

The above result is interesting from two fronts.

(i) Note that the $n^{-2/d}$ rate in part (a) is the optimal achievable rate for the BCLS estimator as matching lower bounds have been established in \cite[Theorem 3.1]{kur2020convex}. This highlights the optimality of our upper bounds in a worst case sense even under $beta$-mixing type dependence in the long-range regime.

(ii) The adaptation rate of $n^{-4/d}$ in part (b) is also known to be the best achievable rate for the BCLS estimator \cite[Theorem 2.5]{kur2020convex}. This highlights that our bounds are tight enough to recover optimal tuning-free adaptive rates under long-range dependence.

We note that for technical reasons it is necessary to work with BCLS instead of the usual least squares (LS) because the usual  convex LS estimator is not well-behaved near the boundary of the domain $\Omega$ (see \cite[Section 4.1.4]{han2016multivariate} for details). 

\begin{rem}[On the domain assumption]\label{rem:samephencon}
The polytopal assumption on the domain is due to the fact that both the rates for BCLS and the minimax risk for estimating $f_\star$ are heavily domain sensitive even in the i.i.d. setting as has been pointed out in numerous influential works \cite{kur2019optimality,kur2020convex,han2016multivariate}. For example, if $\Omega$ is a general compact, smooth domain (see \cite[Section 1.2.2]{han2016multivariate} for definition) which includes, then the BCLS has a slower rate of convergence $n^{-2/(d+1)}$ for i.i.d. data. Strikingly, this rate is tight and it is achieved, for example, when $P_X$ is uniform on the $d$-dimensional Euclidean ball (see \cite[Remark 2]{kur2019optimality}, \cite[Section 2.2]{han2016multivariate}). For such domains in the presence of $\gamma$-mixing (as in \cref{thm:shapeconv}, part (b)), we can leverage \cref{cor:fasterl1} to get the rate $n^{-2/(d+1+2\gamma^{-1})}$ which converges (informally) to the minimax optimal rate $n^{-2/(d+1)}$ as $\gamma\to\infty$.
\end{rem}

\subsection{Estimating Wasserstein distances --- benefits of regularization under dependence}
\label{sec:OT}

Given two compactly supported absolutely continuous probability measures $\mu$ and $\nu$ on $\R^d$, define the $2$-optimal transport/Wasserstein distance between $\mu$ and $\nu$ as follows:
\begin{align}\label{eq:wass}
    \mathcal{W}_2^2(\mu,\nu):=\min_{\pi\in\Gamma(\mu,\nu)} \int  \lVert x-y\rVert^2\,d\pi(x,y),
\end{align}
where $\Gamma(\mu,\nu)$ is the space of all probability measures on $\R^d\times \R^d$ with the first and second marginals fixed at $\mu$ and $\nu$ respectively. The existence of an optimizer in \eqref{eq:wass} follows from \cite[Theorem 4.1]{Villani2009}. The quantity $\mathcal{W}_2(\mu,\nu)$ has attracted a significant amount of attention in recent years with applications in structured prediction \cite{frogner2015learning,luise2018differential}, image analysis \cite{bonneel2011displacement,glaunes2004diffeomorphic}, nonparametric testing \cite{boeckel2018multivariate,deb2023multivariate}, generative modeling \cite{bernton2017inference,mohamed2016learning}, etc. Accordingly, the goal of this section is to estimate $\mathcal{W}_2(\mu,\nu)$. 

In practice, an important hurdle in estimating $\mathcal{W}_2(\mu,\nu)$ is that explicit functional or parametric forms of $\mu$ and $\nu$ are not usually available, and instead, the practitioner only has access to samples from $\mu$ and $\nu$. In other words, suppose that $\{X_1,\ldots ,X_n\}$ and $\{Y_1,\ldots ,Y_n\}$ are strictly stationary and $\beta$-mixing with stationary distributions $\mu$ and $\nu$ respectively. The goal is to estimate $\mathcal{W}_2(\mu,\nu)$ based on the observed samples. The case where $X_i$'s and $Y_j$'s are independent samples is well understood thanks to many quantitative results in the recent years, see e.g. \cite{chizat2020faster,deb2021rates,groppe2023lower,hundrieser2022empirical,manole2021sharp} and the references therein. The case of strictly stationary $\beta$-mixing sequences with summable mixing coefficients has also been considered, see e.g.~\cite{bernton2017inference,goldfeld2022statistical}, although with different goals in mind. However, to the best of our knowledge, the more interesting case where the $\beta$-mixing coefficients are not summable remains unexplored. In the sequel, we will describe and analyze the two most popular estimators of $\mathcal{W}_2(\mu,\nu)$ which will help us compare the two estimators across different regimes of dependence and dimension $d$.

We now describe the two estimators of $\mathcal{W}_2(\mu,\nu)$. First, we define 
$$\hum:=\frac{1}{n}\sum_{i=1}^n \delta_{X_i}, \quad \mbox{and}\quad \hun:=\frac{1}{n}\sum_{j=1}^n \delta_{Y_j},$$
which are the natural empirical measures based on the samples. Both the estimators will be based on the plug-in principle which simply suggests replacing $\mu$ and $\nu$ with $\hum$ and $\hun$ respectively, in an appropriate objective function.

\begin{itemize}
    \item \textbf{\emph{Standard plug-in estimator}}: The first estimator is $\mathcal{W}_2(\hum,\hun)$ defined as in \eqref{eq:wass}. In the sequel, we will provide quantitative finite sample bounds on the quantity $$\EE|\mathcal{W}_2^2(\hum,\hun)-\mathcal{W}_2^2(\mu,\nu)|.$$ Although $\mathcal{W}_2(\hum,\hun)$ is a very natural estimator, it is computationally challenging with a best known worst-case complexity of $\tilde{O}(n^3)$ when $m=n$ (up to logarithmic factors); see e.g. the Hungarian algorithm \cite{jonker1988shortest}. This motivates the need for computationally faster estimators with comparable statistical accuracy. 
    \item \textbf{\emph{Regularized plug-in estimator with Sinkhorn algorithm}}: Motivated by the computational issues mentioned above, we follow \cite{chizat2020faster} and consider the entropic regularized optimal transport problem as a way to estimate $\mathcal{W}_2(\mu,\nu)$. To wit, consider the entropic optimal transport cost:
    \begin{equation}\label{eq:enwass}
    \mathcal{T}_{\vep,2}(\mu,\nu):=\min_{\gamma\in\Gamma(\mu,\nu)} \left(\int\int \lVert x-y\rVert^2\,d\gamma(x,y)+2\vep \mathrm{KL}(\gamma|\mu\otimes\nu)\right),
    \end{equation}
    where $\mathrm{KL}(\cdot|\cdot)$ denotes the standard Kullback-Leibler divergence, and $\vep>0$. Also define the associated Sinkhorn divergence:
    \begin{equation}\label{eq:sinkwass}
    \mathcal{S}_{\vep,2}(\mu,\nu)=\mathcal{T}_{\vep,2}(\mu,\nu)-\frac{1}{2}\left(\mathcal{T}_{\vep,2}(\mu,\mu)+\mathcal{T}_{\vep,2}(\nu,\nu)\right).
    \end{equation}
    Note that $\mathcal{S}_{0,2}(\mu,\nu)=\mathcal{W}_2^2(\mu,\nu)$. It is well-known that the Sinkhorn divergence defined in \eqref{eq:sinkwass} is always non-negative~\cite{genevay2018learning}. The plug-in principle would then suggest using $\mathcal{S}_{\vep,2}(\hum,\hun)$ with a ``small" $\vep$ for estimating $\mathcal{W}_2^2(\mu,\nu)$. This has inspired a line of work studying \eqref{eq:enwass} for $\vep\to 0$ (see \cite{deb2023wasserstein,carlier2017convergence,pooladian2021entropic}). A major breakthrough for solving the optimization problem in \eqref{eq:enwass} can be attributed to \cite{cuturi2013sinkhorn} where the author proposed the Sinkhorn algorithm or IPFP (iterative proportional fitting procedure) for solving \eqref{eq:enwass}; also see \cite{altschuler2017near,dvurechensky2018computational} for related work. As \eqref{eq:enwass} cannot be solved explicitly, we will instead construct our estimator based on an appropriate number of Sinkhorn iterates as described below.

    Set $v_j^{(0)}=0$ for $j\in [n]$ and for $k\ge 1$, $i\in [m]$, let
    $$u_i^{(k)}:=-\vep \log\left(n^{-1}\sum_{j=1}^n \exp\left(\vep^{-1}(v_j^{(k-1)}-\lVert X_i-Y_j\rVert^2)\right)\right), \quad \mbox{and}$$ 
    $$v_j^{(k)}:=-\vep \log\left(m^{-1}\sum_{i=1}^m \exp\left(\vep^{-1}(u_i^{(k)}-\lVert X_i-Y_j\rVert^2)\right)\right).$$
    Then the approximation for the estimator $\mathcal{T}_{\vep,2}(\hum,\hun)$ after $k$ iterations of the Sinkhorn algorithm is given by $\mathcal{T}_{\vep,2}^{(k)}(\hum,\hun):=2m^{-1}\sum_{i=1}^m u_i^{(k)}+n^{-1}\sum_{j=1}^n v_j^{(k)}$. One can similarly define $\mathcal{S}_{\vep,2}^{(k)}(\hum,\hun)$ as in \eqref{eq:sinkwass} with $\mathcal{T}_{\vep,2}(\hum,\hun)$ replaced by  $\mathcal{T}_{\vep,2}^{(k)}(\hum,\hun)$. We shall use $\mathcal{S}_{\vep,2}^{(k)}(\hum,\hun)$ as our estimator for $\mathcal{W}_2^2(\mu,\nu)$ for an appropriately chosen $k$ and $\vep$, and provide finite sample bounds on 
    $$\EE\big|\mathcal{S}_{\vep,2}^{(k)}(\hum,\hun)-\mathcal{W}_2^2(\hum,\hun)\big|.$$
\end{itemize}

The goal of this section is to address the following question:

\begin{center}
\emph{How do the above estimators $\mathcal{W}_2^2(\hum,\hun)$ and $\mathcal{S}_{\vep,2}^{(k)}(\hum,\hun)$ compare based on their statistical accuracy and computational complexity in the presence of $\beta$-mixing?}
\end{center}

\begin{theorem}\label{th:otbalance}
    Assume that $\mu$ and $\nu$ are compactly supported on a convex set, absolutely continuous distributions on $\R^d$.  Fix some $d\ge 4$ and say $\beta_q\leq C(1+q)^{-\beta}$ for some $\beta,C>0$,  $q\in \mathbb{N}\cup \{0\}$. Then the following conclusions hold: 

    \begin{itemize}
        \item[(I)] The unregularized estimator $\mathcal{W}_2^2(\hum,\hun)$ satisfies 
        \begin{equation}\label{eq:wassproof}
        \EE\Bigg|\mathcal{W}_2^2(\hum,\hun)-\mathcal{W}_2^2(\mu,\nu)\bigg|\lesssim  \begin{cases} \hfill n^{-\frac{2}{d}} & \mbox{if}\ \beta> \frac{2}{d-2} \\ n^{-\frac{\beta}{\beta+1}} & \mbox{if}\ \beta<\frac{2}{d-2} \end{cases}.
        \end{equation}
        \item[(II)] Next suppose that both $\mu$ and $\nu$ have a finite Fisher information. Also assume that the Fisher information of the Wasserstein geodesic between $\mu$ and $\nu$ is finite (for all relevant technical definitions, see \cref{sec:pfotbalance}). The regularized estimator $\mathcal{S}_{\vep,2}^{(k)}(\hum,\hun)$ then satisfies 
        \begin{equation}\label{eq:sinkproof}
        \EE\Bigg|\mathcal{S}_{\vep_n,2}^{(k_n)}(\hum,\hun)-\mathcal{W}_2^2(\mu,\nu)\bigg|\lesssim  \begin{cases} \hfill n^{-\frac{2}{d}} & \mbox{if}\ \beta> \frac{2}{d-2} \\ n^{-\frac{\beta}{\beta+1}} & \mbox{if}\ \beta<\frac{2}{d-2} \end{cases},
        \end{equation}
        where 
        $$(k_n,\vep_n)=\begin{cases}\hfill \left(n^{\frac{3}{d}}, n^{-\frac{1}{d}}\right) & \mbox{if}\ \beta> \frac{2}{d-2} \\ \left(n^{\frac{3\beta}{2(\beta+1)}}, n^{-\frac{\beta}{2(\beta+1)}}\right) & \mbox{if}\ \beta<\frac{2}{d-2} \end{cases}.$$
    \end{itemize}
    With the above choices of $(k_n,\vep_n)$, it follows that $\mathcal{S}^{(k_n)}_{\vep_n,2}(\hum,\hun)$ can be computed in time $\tilde{O}(n^{2+\frac{3}{d}})$ for $\beta>2/(d-2)$, and in time $\tilde{O}(n^{2+\frac{3\beta}{2(\beta+1)}})$ for $\beta<2/(d-2)$.
\end{theorem}

\begin{rem}[Benefits of entropic regularization under long-range dependence]\label{rem:benenreg}
    Based on \cref{th:otbalance}, we note that both $\mathcal{W}_2^2(\hum,\hun)$ and $\mathcal{S}_{\vep,2}^{(k)}(\hum,\hun)$ attain the same statistical upper bounds for estimating $\mathcal{W}_2^2(\mu,\nu)$. On the other hand, there is a marked difference in their computational complexities. As mentioned before, the worst case complexity for computing $\mathcal{W}_2^2(\hum,\hun)$ is $\tilde{O}(n^3)$. Note that the computational complexity of $\mathcal{S}_{\vep,2}^{(k)}(\hum,\hun)$ is always faster than $\tilde{O}(n^3)$.  In particular, informally speaking, as $\beta$ gets smaller with fixed $d$, the complexity of the regularized estimator $\mathcal{S}_{\vep,2}^{(k)}(\hum,\hun)$ also decreases and approaches $\tilde{O}(n^2)$. Similarly, if $d$ increases for fixed $\beta$, then again the computational complexity of  $\mathcal{S}_{\vep,2}^{(k)}(\hum,\hun)$ decreases and approaches $\tilde{O}(n^2)$. This suggests that $\mathcal{S}_{\vep,2}^{(k)}(\hum,\hun)$ can be a more favorable estimator than $\mathcal{W}_2^2(\hum,\hun)$, particularly under strong dependence or large values of $d$. We also note that the benefits of regularization for large $d$ was already pointed out in \cite{chizat2020faster}. Our main contribution in this section is to extract the benefits of dependence.
\end{rem}

\begin{rem}[Assumptions on $\mu$, $\nu$]\label{rem:assnext}
    There are a number of natural extensions of \cref{th:otbalance} possible with different assumptions on $\mu$, $\nu$, and other cost functions beyond the squared distance. A particularly interesting one is when say, one of $\mu$ or $\nu$, is supported on a discrete set of bounded size (in $n$). In this case, using the same proof technique as employed for proving \cref{th:otbalance}, coupled with \cite[Lemma 2.1]{hundrieser2022empirical}, it is possible to show 
    $$\EE\Bigg|\mathcal{W}_2^2(\hum,\hun)-\mathcal{W}_2^2(\mu,\nu)\bigg|\lesssim  \begin{cases} \hfill n^{-\frac{1}{2}} & \mbox{if}\ \beta> 1 \\ n^{-\frac{\beta}{\beta+1}} & \mbox{if}\ \beta<1 \end{cases}.$$ 
\end{rem}

\subsection{Classification under Mammen-Tsybakov margin condition}\label{sec:classific}
In classification problems, one observes $(X_1,Y_1),\ldots , (X_n,Y_n)$ where each $(X_i,Y_i)\sim P$ (unknown), a probability measure supported on $\R^d\times \{0,1\}$. We will write the marginal of $X$ as $P_X$ and we refer to \emph{functions $g:\R^d\to\{0,1\}$ as classifiers}. The goal is to find $g$ that minimizes the \emph{generalization error} defined as 
\begin{align}\label{eq:classgener}
\cE(g):=\mathbb{P}(Y\neq g(X))-\inf_g \mathbb{P}(Y\neq g(X)).
\end{align}
It is well known that the infimum above is attained by the Bayes optimal classifier $g_\star(x):=\mathbf{1}(\EE[Y|X=x]\ge 1/2)$ (see \cite{devroye2013probabilistic}). As $P$ is unknown, neither $\EE[Y|X=x]$ nor $\cE(\cdot)$ can be computed explicitly. 

There are two predominant approaches to tackling the classification problem --- (a) Plug-in principle where the strategy is to estimate $\EE[Y|X=x]$ under smoothness assumptions (see), and (b) ERM where $\cE_P(\cdot)$ is replaced by its empirical analog $\cE_{\mathbb{P}_n}(\cdot)$; see \cite{devroye2013probabilistic,vapnik1999nature,tsybakov2004optimal,Audibert2007} and the references therein. In keeping with the rest of the paper, we will adopt the latter approach here and relax the standard i.i.d. assumption by allowing for dependence in the DGP via mixing assumptions.

As the second term in \eqref{eq:classgener} is free of $g$, a natural estimator can be constructed by minimizing the empirical training error, i.e., 
$$\hat{g}_n\in \argmin_{g\in\mathcal{G}}\,\,\frac{1}{n}\sum_{i=1}^n \mathbf{1}(Y_i\neq g(X_i)),$$
where $\mathcal{G}$ is some prespecified hypothesis class. We assume $g_\star\in\mathcal{G}$ to avoid digressions. The global ERM defined above has been studied previously in the i.i.d. setting; see e.g., \cite{GineKoltchinskii2006,Massart2006,Kol06,tsybakov2004optimal} but to the best of our knowledge, no quantitative bounds on the generalization error $\cE_P(\hat{g}_n)$ exists when the independence assumption in the DGP is removed. The result below aims to fill this gap in the literature. 

Before stating the main result, we require an important assumption, namely the ``Mammen-Tsybakov margin (low noise) condition'' (see \cite{mammen1999smooth,tsybakov2004optimal}): there exists $c>0$ such that 
\begin{equation}\label{eq:margincon}
    \cE(g)\ge c\lVert g-g_\star\rVert_{L_2(P_X)}^2.
\end{equation}
Following the notation of \cite{tsybakov2004optimal}, \eqref{eq:margincon} assumes the margin condition with $\kappa=2$. We refer the reader to the discussion around \cite[Proposition 1]{tsybakov2004optimal} for motivation on the margin condition and its relation with the behavior of $\EE[Y|X=x]$ near $1/2$. For larger values of $\kappa$, it is possible to derive faster rates (even under dependence) when $\mathcal{G}$ is of low complexity (see \cite{tsybakov2004optimal,GineKoltchinskii2006}). However to keep the discussion streamlined, we only focus on $\mathcal{G}$ of high complexity and $\kappa=1$ in the following result.

\begin{theorem}\label{prop:classerr}
Suppose that the DGP is stationary $\gamma$-mixing with $\gamma_q\le C(1+q)^{-\gamma}$ for some $C>0$. Assume that $\mathcal{G}:=\{\mathbf{1}(C):\ C\in\mathcal{C}\}$ and $\mathcal{C}$ satisfies the following entropy condition: there exists $\alpha>1+\gamma^{-1}$ such that for all $u>0$, $\benoc{u}\le Lu^{-\alpha}$. Finally assume that $\mathcal{G}$ satisfies the margin condition \eqref{eq:margincon}. Then given any $\{r_n\}){n\ge 1}$ such that $r_n^2 n^{\frac{1}{\alpha+1+\gamma^{-1}}}\to\infty$, we have  $$\cE(\hat{g}_n)=o_P(r_n^2).$$
\end{theorem}

Note that as $\gamma\to\infty$, the generalization error informally converges to $n^{-1/(\alpha+1)}$ which is the minimax optimal rate in the i.i.d. setting (see \cite{han2021set}). This shows that our empirical process bounds can be used to derive meaning rates even with non-smooth loss functions. The presentation and proof of \cref{prop:classerr} follows that of \cite[Theorem 3.5]{han2021set} and \cite[Theorem 7.1]{GineKoltchinskii2006} with appropriate modifications to incorporate dependence. We also refer the reader to the discussion around \cite[Theorem 1]{mammen1999smooth} for examples of $\mathcal{G}$ that satisfy the entropy condition in \cref{prop:classerr}.



\section{Conclusion and future research}
In this paper, we provide the first quantifiable trade-offs between dependence among the observed data and the complexity of function/hypothesis classes.
A major finding is that i.i.d. like rates hold under long-range dependence for complex enough function classes. 
Our approach is broadly applicable as we provide dependence-sensitive bounds on empirical processes across all levels of dependence (both long and short-range), which can be applied directly to various learning problems of interest (See Section \ref{sec:app} for some applications). 
Several interesting future research directions originate from our work.
\begin{itemize}
    \item {\bf Minimax lower bounds} --- A natural question is whether our trade-offs are optimal in a minimax sense. A first step towards that should be to construct sequences for which the covariance inequality in \cref{lem:var_bound_improved_davy} is tight. It seems that an extension of the lower bound in \cite[Theorem 1.1 (b)]{rio1993covariance} to infinite sequences would be helpful.
    \item {\bf Concentration inequality} --- While we obtain upper bounds for the expected supremum of empirical processes, the question of obtaining Talagrand type concentration inequalities (see \cite{talagrand1996new,bousquet2002bennett}) under polynomial dependence remains open. In this case, we do not expect sub-exponential decay and it seems challenging to extract the correct polynomial tail without enforcing strong model assumptions or without restricting to VC classes of functions.
    \item {\bf Extension to other notions of dependency} --- While our results apply to $\beta/\gamma$ mixing sequences, it remains challenging to establish same bounds for weaker notions of dependence, such as $\alpha$-mixing \cite{rosenblatt1956central}, $\tau$-mixing \cite{dedecker2004coupling}, nonlinear system theory based dependence \cite{wu2005nonlinear}, etc.
\end{itemize}

\bibliographystyle{ims}
\bibliography{sample}

\newpage

\appendix 

\section{Proof of main results}
\label{sec:pfmain}

This section is organized as follows: In \cref{sec:lemaux}, we present some preparatory lemmas which will be used in the proofs of our main results. The proofs of Theorems \ref{thm:phimainres}, \ref{thm:infmainres}, \ref{thm:gamma_mixing}, and \ref{thm:roc} will be presented in Sections \ref{sec:pfphithm}, \ref{sec:pfinfthm}, \ref{sec:pfgamthm}, and \ref{sec:proof_roc}.
\subsection{Auxiliary Lemmas}\label{sec:lemaux}
\label{sec:auxlem}
In this section, we collect all the auxiliary results needed for our main theorems. The proofs of these results are deferred to \cref{sec:auxlem}. We would also refer the reader to \cref{sec:nota} for all the notation used in the sequel.

The first result relates two important norms $\lVert\cdot\rVert_{2,P}$ and $\lVert \cdot\rVert_{\phi,P}$ for $\phi\in \Phi$.
\begin{lemma}\label{lem:qdbd}
Given a measurable $h:\mcx\to\R$ and probability measure $P$ supported on $\mcx$ such that $h\in \mL_{\phi}^P$, for $\phi\in\Phi$, there exists a constant $\cph>1$ such that
$$\ltp{h}\le \cph \ornm{h}.$$
\end{lemma}

The next lemma provides a generic variance bound under a $\beta$-mixing assumption.

\begin{lemma}
\label{lem:var_bound_improved_davy}
Suppose $\{X_i\}_{i \in \bbN}$ be a stationary $\beta$-mixing sequence with marginal distribution $P$, taking values in some Polish space $\mcx$. Recall the definition of $\Lbp$ from \eqref{eq:piv}. Let $h:\mcx\to\R$,  $\phi\in\Phi$, and $\cph$ be taken from \cref{lem:qdbd}. Then 
$$
\var\left(\sum_{i=1}^q h(X_i)\right) \le  q \ornm{h}^2\left(\cph^2+2\Lbp(q)\right) \,.
$$
\end{lemma}

We will now state an important maximal inequality for a finite set of functions which plays a very important role in the proof of \cref{thm:phimainres}.

\begin{prop}
\label{thm:maximal_finite_davy}
Consider the same setting as in \cref{lem:var_bound_improved_davy}. Let $\cF\subseteq \mL_{\phi}(P)$ be a finite collection of functions $f$ with $\|f\|_\infty \le b$ and $\ornm{f}\le \sigma$ for some $b,\sigma>0$ and all $f \in \cF$. Then we have: 
\begin{align*}
&\;\;\;\;\;\;\EE\left[\max_{f \in \cF} \left|\frac{1}{\sqrt{n}}\sum_{i=1}^n\left(f(X_i) - Pf\right)\right|\right] \\ & \le K\inf_{1 \le q \le n}\left(\sigma \ppb(q) \sqrt{1+\log{|\cF|}} + b q \frac{1+\log{|\cF|}}{\sqrt{n}} + b \beta_q\sqrt{n}\right),
\end{align*}
for some absolute constant $K>0$, where 
\begin{equation}\label{eq:piq}
\ppb(q) := \sqrt{\cph^2+2\Lbp(q)}.
\end{equation}
\end{prop}

The next result is a technical bound for the upper tail $L_1(P)$ integral of a function and is a generalization of 
 \cite[Lemma 4]{doukhan1995invariance}. We need to define two relevant quantities. Fix any $\vep\in (0,1)$, $q\in \mathbb{N}$, and $h\in L_1(P)$. Recall that $Q_h$ is the inverse of $t \mapsto P(|f(X)|>t\cdot)$ and define,
\begin{equation}\label{eq:bq}
B_q(t)=\int_0^t \bigg(\sum_{k=0}^{q}\mathds{1}(u\le \beta(k))\bigg)\,du=\sum_{k=0}^{q} (\beta(k)\wedge t)\,,
\end{equation}
and
\begin{equation}\label{eq:omqh}
\omega_{q,h}(\vep):=\sup_{t\le \vep}\  Q_h(t)\sqrt{B_q(t)}.
\end{equation}
We are now ready to state the required result.
\begin{lemma}
\label{lem:l1phin}
Fix any $\vep\in (0,1)$, $q\in \mathbb{N}$, and let $h\in L_1(P)$.  Then 
\begin{align}\label{eq:phin1}
    \lVert h\mathds{1}(|h|> Q_h(\vep))\rVert_{L_1(P)}\le \frac{2\vep\omega_{q,h}(\vep)}{\sqrt{B_q(\vep)}}.
\end{align}
Suppose now $h\in\mL_{\phi}(P)$ and $\upsilon>0$ satisfies 
\begin{align}\label{eq:phin2}
\upsilon\sqrt{\vep\sum_{k=0}^q \mathds{1}(\vep\le \beta_k)}\ge \lVert h\rVert_{\phi,P}\sqrt{\Lbp(q)}.
\end{align}
Then, we have:
\begin{align}\label{eq:phin3}
    \lVert h\mathds{1}(|h|> \upsilon)\rVert_{L_1(P)}\le 2\lVert h\rVert_{\phi,P}\sqrt{\Lbp(q)}\sqrt{\frac{\vep}{\sum_{k=0}^{q} \mathds{1}(\vep\le \beta_k)}}.
\end{align}
\end{lemma}

\subsection{Proof of \cref{thm:phimainres}}\label{sec:pfphithm}
\begin{proof}
Throughout this proof, we will use $M$ to denote universal constants that might change from line to line. In order to extend \cref{thm:maximal_finite_davy} to the case where $|\cF|$ is infinite, we need further notation.

First, we define 
$$
S:=\min\left\{s\ge 0: 2^{-s}\le \frac{a}{2^4\sqrt{n}\sigma}\right\}.$$
By \eqref{eq:basedef1}, we have
$$\frac{a}{2^4\sqrt{n}\sigma}\le \frac{1}{2},
$$
and so $S\ge 1$, and $\sigma>2^{-(S+1)}\sigma$. By definition of $S$, $2^{-(S-1)}>a/(2^4\sqrt{n}\sigma)$ and so $\sigma>2^{-(S+1)}\sigma>a/(2^6\sqrt{n})$. This implies: 
\begin{align*}
    \int_{\frac{a}{2^6\sqrt{n}}}^{\sigma} \sqrt{1+\mbH{u}}\,du&\ge \int_{2^{-(S+1)}\sigma}^{\sigma}\sqrt{1+\mbH{u}}\,du\\ &=\sum_{s=1}^{S+1} \int_{2^{-s}\sigma}^{2^{-s+1}\sigma} \sqrt{1+\mbH{u}}\,du\\ &\ge \sum_{s=1}^{S+1} 2^{-s}\sigma \sqrt{1+\mbH{2^{-s+1}\sigma}}\\ &=\frac{1}{2}\sum_{s=0}^S 2^{-s}\sigma \sqrt{1+H_s},
\end{align*}
where $H_s:=\mbH{2^{-s}\sigma}$, $s=0,1,\ldots ,S$. Above, we have used the non-increasing nature of $\mbH{\cdot}$ in the first argument. By \eqref{eq:basedef1} and noting that 
$$\rno(\delta)\ge 1+ \mbH{\delta}$$ 
for $\delta>0$, we then have:
\begin{equation}\label{eq:basedef3}
\sum_{s=0}^S 2^{-s}\sigma \sqrt{1+H_s}\le \frac{2a}{C_0}.
\end{equation}

An implication of \eqref{eq:basedef3} is also that:

\begin{align}\label{eq:basedef4}
    \sum_{s=1}^S 2^{-s}\sigma \sqrt{\sum_{k=0}^s (1+H_k)}&\le \sum_{k=0}^S \sqrt{1+H_k}\sigma \sum_{s=k}^S 2^{-s}\nonumber \\ &\le 2\sum_{k=0}^S 2^{-k}\sigma\sqrt{1+H_k}\le \frac{4a}{C_0}.
\end{align}
Here the first inequality uses the fact that $\sqrt{\sum_j a_j} \le \sum_j \sqrt{a_j}$. 
Similarly, by using the non-increasing nature of $\rno$, we also have:

\begin{align}\label{eq:basedef600}
    \sum_{s=0}^S 2^{-s}\sigma \sqrt{\rno(2^{-s}\sigma)}\le \frac{2a}{C_0}.
\end{align}

\begin{align}\label{eq:basedef600st}
    \sum_{s=1}^S 2^{-s}\sigma \sqrt{\Ld{2^{-s}\sigma}}\sqrt{\sum_{k=0}^s (1+H_k)}&\le \sum_{k=0}^S \sqrt{1+H_k}\sigma \sum_{s=k}^S 2^{-s}\sqrt{\Ld{2^{-s}\sigma}}\nonumber \\&\le 2\sum_{k=0}^S 2^{-k}\sigma \sqrt{1+H_k}\sqrt{\Ld{2^{-k}\sigma}}\le \frac{4a}{C_0}.
\end{align}

In the last display above, we used the non-decreasing nature of $\Ld{\cdot}$ from \cref{assm:bsbd}, followed by~\eqref{eq:asnbd1} and \eqref{eq:basedef600}.\par

Keeping the above observations in mind, for $s=0,1,\ldots ,S$, let $[\gjl{s},\gju{s}]_{j=1}^{N_s}$ be a $2^{-s}\sigma$-bracket of $\cF$, where $N_s=\bnum{2^{-s}\sigma}$ is the $2^{-s}\sigma$-bracketing number of $\cF$. Given any $f\in\cF$, and any $s\in \{0,1,\ldots ,S\}$, there exists $[\gl{s},\gu{s}]$ such that 
$
\gl{s}(x)\le f(x)\le \gu{s}(x),\, \mbox{for all} \ x\in\mcx$, and $\lVert\gu{s}-\gl{s}\rVert_{L_1(P)}\le 2^{-s}\sigma$.

Further define the following finer (and nested) brackets: 
$$\tgl{s}(x):=\max_{0\le k\le s}\gl{k}(x),\quad \tgu{s}(x):=\min_{0\le k\le s}\gu{k}(x),$$
and 
$$\df{s}(x):=\tgu{s}(x)-\tgl{s}(x).$$
It is immediate from the definition that: 
\begin{equation}\label{eq:obs1}
    \tgl{s}(x)\le f(x)\le \tgu{s}(x).
\end{equation}
Define 
\begin{align}\label{eq:defks}
K_s:= \frac{2^{-s+1}\sqrt{n}\sigma \sqrt{\Ld{2^{-s-1}\sigma}}}{\tq{2^{-s-1}\sigma}\sqrt{1+\sum_{k=0}^{s+1} H_k}}, \quad \mbox{for}\ 0\le s\le S-1.
\end{align}
Finally, define 
\begin{align}\label{eq:adaptrun}
\nu_f(x):=\begin{cases} \min\{0\le s\le S-1:\ \df{s}(x)\ge K_s\} & \mbox{if}\ \df{s}(x)\ge K_s\ \mbox{for some}\ 0\le s\le S-1,\\ S & \mbox{otherwise}.\end{cases}
\end{align}

With this in view, given any $f\in\cF$, we begin with the following decomposition:

\begin{align}\label{eq:chaindec}
    f(x)=\tgl{0}(x)+\sum_{s=0}^S (f(x)-\tgl{s}(x))\mathds{1}(\nu_f(x)=s)+\sum_{s=1}^S (\tgl{s}(x)-\tgl{s-1}(x))\mathds{1}(\nu_f(x)\ge s),
\end{align}
which is easy to verify. For notational convenience, define:  

\begin{equation}\label{eq:THETAS}
\begin{split}
\got(x) & :=(f(x)-\tgl{s}(x))\mathds{1}(\nu_f(x)=s), \\
\gtt(x) & :=(\tgl{s}(x)-\tgl{s-1}(x))\mathds{1}(\nu_f(x)\ge s).
\end{split}
\end{equation}
Consequently the following bound holds:
\begin{align*}
    &\;\;\;\;\;\EE\left[\frac{1}{\sqrt{n}}\sup_{f\in\cF}\bigg|\sum_{i=1}^n (f(X_i)-\EE f(X_1))\bigg|\right] \\ 
    &\le \EE\left[\frac{1}{\sqrt{n}}\sup_{f\in\cF}\bigg|\sum_{i=1}^n (\tgl{0}(X_i)-\EE \tgl{0}(X_1))\bigg|\right]+\EE\left[\frac{1}{\sqrt{n}}\sup_{f\in\cF}\bigg|\sum_{s=0}^S \sum_{i=1}^n \big(\got(X_i)-\EE\got(X_1)\big)\bigg|\right]\\ 
    & +\EE\left[\frac{1}{\sqrt{n}}\sup_{f\in\cF}\bigg|\sum_{s=1}^S \sum_{i=1}^n \big(\gtt(X_i)-\EE\gtt(X_1)\big)\bigg|\right]\\ 
    & =: T_1+T_2+T_3.
\end{align*}
We now bound the three terms $T_1$, $T_2$, and $T_3$ individually.

{\bf Bound for $T_1$: } As $f\in\cF$ varies, note that $\tgl{0}$ varies over a finite set of cardinality $\exp(H_0)$.  With this observation, we bound $T_1$ directly using \cref{lem:var_bound_improved_davy} to get:
\begin{equation}\label{eq:t11}
T_1\le M\inf_{1 \le q \le n}\left(\sigma \ppb(q) \sqrt{1+H_0} + b q \frac{1+H_0}{\sqrt{n}} + b \beta_q\sqrt{n}\right)
\end{equation}
for the same constant as in \cref{lem:var_bound_improved_davy}. We now plug in $q\equiv \tq{\sigma}$ in \eqref{eq:t11} and observe that:
$$
b\tq{\sigma} \frac{1+H_0}{\sqrt{n}} + b \beta(\tq{\sigma})\sqrt{n}\le 2b\tq{\sigma}\frac{1+H_0}{n}.$$

Combining the above observations with \eqref{eq:basedef3}, \eqref{eq:basedef600} and \eqref{eq:t11}, we get:
\begin{align}\label{eq:t12}
    T_1&\le M\left(\sigma\cph\sqrt{1+H_0}+2\sigma\sqrt{\Ld{\sigma}(1+H_0)}+ 2b \tq{\sigma}\frac{1+H_0}{\sqrt{n}}\right)\nonumber  \\&\le \boxed{M\left(\frac{2a\cph}{C_0}+\frac{4a}{C_0}+ 2b \tq{\sigma}\frac{1+H_0}{\sqrt{n}}\right) \le \frac{2aM}{C_0}(\cph+4)+2b M \tq{\sigma}\frac{1+H_0}{\sqrt{n}}}.
\end{align}

{\bf Bound on $T_2$: }
Recall the definition of $T_2$: 
$$
\EE\left[\frac{1}{\sqrt{n}}\sup_{f\in\cF}\bigg|\sum_{s=0}^S \sum_{i=1}^n \big(\got(X_i)-\EE\got(X_1)\big)\bigg|\right] 
$$
where $\got(x) = (f(x)-\tgl{s}(x))\mathds{1}(\nu_f(x)=s)$. It is immediate that: 
\begin{align}\label{eq:T2-1}
    T_2 & \le \EE\left[\frac{1}{\sqrt{n}}\sup_{f\in\cF}\bigg| \sum_{i=1}^n \big(\gotz(X_i)-\EE\gotz(X_1)\big)\bigg|\right] \nonumber \\ 
    & \qquad \qquad + \sum_{s = 1}^S \EE\left[\frac{1}{\sqrt{n}}\sup_{f\in\cF}\bigg|\sum_{i=1}^n \big(\got(X_i)-\EE\got(X_1)\big)\bigg|\right] 
\end{align}
Note that, by \eqref{eq:obs1}, we have the following for any $0\le s\le S$: 
\begin{align*}
    & \got(X) - \bbE[\got(X)] \\
    & = (f(X)-\tgl{s}(X))\mathds{1}(\nu_f(X)=s) - \bbE\left[(f(X)-\tgl{s}(X))\mathds{1}(\nu_f(X)=s)\right] \\
    & \le \Delta_f^s(X)\mathds{1}(\nu_f(X)=s) - \bbE\left[(f(X)-\tgl{s}(X))\mathds{1}(\nu_f(X)=s)\right] \\
    & = \Delta_f^s(X)\mathds{1}(\nu_f(X)=s) - \bbE[\Delta_f^s(X)\mathds{1}(\nu_f(X)=s)] \\
    & \qquad \qquad + \bbE\left[\left\{\Delta_f^s(X) - (f(X)-\tgl{s}(X))\right\}\mathds{1}(\nu_f(X)=s)\right] \\
    & \le \Delta_f^s(X)\mathds{1}(\nu_f(X)=s) - \bbE[\Delta_f^s(X)\mathds{1}(\nu_f(X)=s)] + \bbE[\Delta_f^s(X)\mathds{1}(\nu_f(X)=s)]
\end{align*}
Therefore from \eqref{eq:T2-1}, we have: 
\begin{align}\label{eq:T2-2}
    T_2 & \le \sum_{s = 1}^S \EE\left[\frac{1}{\sqrt{n}}\sup_{f\in\cF}\bigg|\sum_{i=1}^n \big(\Delta_f^s(X_i)\mathds{1}(\nu_f(X_i)=s) -\EE\Delta_f^s(X_1)\mathds{1}(\nu_f(X_1)=s)\big)\bigg|\right]\nonumber \\ &+  \EE\left[\frac{1}{\sqrt{n}}\sup_{f\in\cF}\bigg|\sum_{i=1}^n \big(\Delta_f^0(X_i)\mathds{1}(\nu_f(X_i)=0) -\EE\Delta_f^0(X_1)\mathds{1}(\nu_f(X_1)=0)\big)\bigg|\right]\nonumber \\
    & +\sqrt{n}\sum_{s = 0}^{S-1} \sup_{f \in \cF}\EE\big|\Delta_f^s(X_1)\mathds{1}(\nu_f(X_1)=s)\big)\big|+\sqrt{n}\sup_{f \in \cF}\EE\big|\Delta_f^S(X_1)\mathds{1}(\nu_f(X_1)=S)\big)\big|\,\nonumber \\ & =: T_{2,1} + T_{2,2} + T_{2,3} + T_{2,4}.
\end{align}
We now bound each of the four terms above. Towards this direction, define a function class $$\cF_s := \left\{\Delta_f^s\mathds{1}(\nu_f=s): f \in \cF\right\}.$$ We then have the following observations: 
\begin{equation}
\label{eq:obs2}
   |\cF_s| \le \left(\Pi_{k = 0}^s N_k\right)^2 \implies \log{|\cF_s|} \le 2\sum_{k = 0}^s H_k.
\end{equation}
  
\begin{equation}\label{eq:obs3}
\lVert \Delta_f^0 \rVert_{\infty}\le 2b.
\end{equation}

  By the definition of $\nu_f$,

\begin{equation}\label{eq:obs4}
\left\|\Delta_f^s \mathds{1}(\nu_f=s)\right\|_\infty \le \left\|\Delta_f^{s-1} \mathds{1}(\nu_f=s)\right\|_\infty \le K_{s-1} \quad \mbox{for}\, s=1,\ldots ,S.
\end{equation}

\begin{equation}\label{eq:obs5}
   \left\|\Delta_f^s \mathds{1}(\nu_f=s)\right\|_{\phi, P}\le \left\|\Delta_f^s \right\|_{\phi, P} \le 2^{-s}\sigma.
\end{equation}

To bound $T_2$, we begin with the first term on the right hand side of \eqref{eq:T2-2}. Again, as application of maximal inequality for finite $\cF$ from \cref{thm:maximal_finite_davy} coupled with \eqref{eq:obs2}, \eqref{eq:obs4}, and \eqref{eq:obs5}, yields: 
\begin{align}\label{eq:T21_init}
    T_{2,1} & =\sum_{s = 1}^S \EE\left[\frac{1}{\sqrt{n}}\sup_{f\in\cF}\bigg|\sum_{i=1}^n \big(\Delta_f^s(X_i)\mathds{1}(\nu_f(X_i)=s) -\EE\Delta_f^s(X_1)\mathds{1}(\nu_f(X_1)=s)\big)\bigg|\right]\nonumber \\
    & \le M\sum_{s = 1}^S\inf_{1 \le q \le n}\left(2^{-s}\sigma \ppb(q) \sqrt{1+\sum_{k = 0}^s H_k} + K_{s-1} q \frac{1+\sum_{k = 0}^s H_k}{\sqrt{n}} + K_{s-1} \beta_q\sqrt{n}\right)
\end{align}
For $1\le s\le S$, by choosing $q\equiv \tq{2^{-s}\sigma}$, we observe that
$$q \frac{1+\sum_{k = 0}^s H_k}{\sqrt{n}} +  \beta_q\sqrt{n}\le 2\tq{2^{-s}\sigma} \frac{1+\sum_{k = 0}^s H_k}{\sqrt{n}},$$
and 
$$K_{s-1}\tq{2^{-s}\sigma} \frac{1+\sum_{k = 0}^s H_k}{\sqrt{n}}=2^{-s+2}\sigma\sqrt{\Ld{2^{-s}\sigma}}\sqrt{1+\sum_{k=0}^s H_k}.$$
Combining the two displays above with \eqref{eq:basedef4}, \eqref{eq:basedef600st}, and \eqref{eq:T21_init}, we get the following:
\begin{align}\label{eq:T21_final}
    T_{2,1} &\le M \sum_{s=1}^S \bigg(2^{-s}\sigma \cph \sqrt{1+\sum_{k = 0}^s H_k} + 2^{-s+4}\sigma \sqrt{\Ld{2^{-s}\sigma}} \sqrt{1+\sum_{k = 0}^s H_k}\bigg)\nonumber \\ &\le \boxed{M\left(\frac{4a\cph}{C_0}+\frac{64a}{C_0}\right)=\frac{4aM}{C_0}(\cph+16)}.
\end{align}

We move on to the second term on the right hand side of \eqref{eq:T2-2}. By invoking \cref{thm:maximal_finite_davy} combined with \eqref{eq:obs3}, we have:

\begin{align*}
    T_{2,2}&=\EE\left[\frac{1}{\sqrt{n}}\sup_{f\in\cF}\bigg|\sum_{i=1}^n \big(\Delta_f^0(X_i)\mathds{1}(\nu_f(X_i)=0) -\EE\Delta_f^0(X_1)\mathds{1}(\nu_f(X_1)=0)\big)\bigg|\right]\\ &\le M\inf_{1 \le q \le n}\left(\sigma \ppb(q) \sqrt{1+H_0} + 2b q \frac{1+H_0}{\sqrt{n}} + 2b \beta_q\sqrt{n}\right)\\ &\le \boxed{\frac{4aM}{C_0}(\cph+4)+4 b M q \frac{1+H_0}{\sqrt{n}}} ,
\end{align*}
where the last bound follows from \eqref{eq:t12} by choosing $q\equiv \tq{\sigma}$. \par 

We now move on to the third term on the right hand side of \eqref{eq:T2-2}. The main technical tool for bounding this  term, i.e. $T_{2,3}$, is \cref{lem:l1phin}. We invoke it with the following choices of parameters:
$$\upsilon \equiv \upsilon_s := K_{s} \qquad \mbox{and} \qquad  \vep \equiv \vep_s := ((\tq{2^{-s-1}\sigma}-1)\vee 1)\frac{1+\sum_{k=0}^{s+1} H_k}{n}.$$
By the definition of $\tq{2^{-s}\sigma}$, we have 
$$
\beta(\tq{2^{-s-1}\sigma}-1)\ge \vep_s.
$$

In other words, we have:
$$\sum_{k=0}^{\tq{2^{-s-1}\sigma}} \mathds{1}(\vep_s\le \beta(k)) \ge (\tq{2^{-s-1}\sigma}-1)\vee 1\ge \frac{1}{2}\tq{2^{-s-1}\sigma}.$$
Using the above displays, we conclude that:
\begin{align*}
    &\;\;\;\;\upsilon \sqrt{\vep_s\sum_{k=0}^{\tq{2^{-s-1}\sigma}} \mathds{1}(\vep_s\le \beta(k))}\\ 
    &=\frac{2^{-s+1}\sqrt{n}\sigma \sqrt{\Ld{2^{-s-1}\sigma}}}{\tq{2^{-s-1}\sigma}\sqrt{1+\sum_{k=0}^{s+1} H_k}}\sqrt{\vep_s}\sqrt{\sum_{k=0}^{\tq{2^{-s-1}\sigma}} \mathds{1}(\vep_s\le \beta(k))}\\ &\ge \frac{2^{-s+1}\sqrt{n}\sigma \sqrt{\Ld{2^{-s-1}\sigma}}}{\tq{2^{-s-1}\sigma}\sqrt{1+\sum_{k=0}^{s+1} H_k}} \sqrt{\frac{\tq{2^{-s-1}\sigma}\big(1+\sum_{k=0}^{s+1} H_k\big)}{2n}}\sqrt{\frac{1}{2}\tq{2^{-s-1}\sigma}}\\ &=\cdot 2^{-s}\sigma \sqrt{\Ld{2^{-s-1}\sigma}}\\ &\ge \lVert \Delta_f^s\rVert_{\phi,P}\sqrt{\Ld{2^{-s-1}\sigma}}\\ &\ge \lVert \Delta_f^s\rVert_{\phi,P}\sqrt{\Lbp(\tq{2^{-s-1}\sigma})}.
\end{align*}
Here the last two inequalities follow from \eqref{eq:obs5} and \eqref{eq:asncall} respectively. The above inequality verifies \eqref{eq:phin2} with 
$$q\equiv \tq{2^{-s-1}\sigma},\qquad \mbox{and}\qquad h\equiv \Delta_f^s.$$
Therefore by applying \eqref{eq:phin3}, combined with \eqref{eq:obs5}, we get:
\begin{align*}
    \EE\big|\Delta_f^s(X_1)\mathds{1}(\nu_f(X_1)=s)\big)\big|& \le 2^{-s+1}\sigma \sqrt{\Ld{2^{-s-1}\sigma}}\sqrt{\frac{\vep_s}{\sum_{k=0}^{\tq{2^{-s-1}\sigma}} \mathds{1}(\vep_s\le \beta(k))}}\\ &\le \sqrt{2} 2^{-s+1}\sigma \sqrt{\Ld{2^{-s-1}\sigma}}\sqrt{\frac{1+\sum_{k=0}^{s+1} H_k}{n}}.
\end{align*}
With the above observation, we can now bound $T_{2,3}$ as follows:
\begin{align*}
    T_{2,3}&=\sqrt{n}\sum_{s = 0}^{S-1} \sup_{f \in \cF}\EE\big|\Delta_f^s(X_1)\mathds{1}(\nu_f(X_1)=s)\big)\big|\\ &\le \boxed{8\sqrt{n}\sum_{s=1}^S 2^{-s}\sigma \sqrt{\Ld{2^{-s}\sigma}}\sqrt{\frac{1+\sum_{k=0}^{s} H_k}{n}}\le \frac{64a}{C_0}}, 
\end{align*}
where the final two inequalities follow from \eqref{eq:asnbd1} and \eqref{eq:basedef600st} respectively.
\par
Finally, we arrive at the fourth term on the right hand side of \eqref{eq:T2-2}. Note that, by using \cref{lem:qdbd} and \eqref{eq:obs5}, we have: 
\begin{align*}
T_{2,4}=\sqrt{n}\sup_{f \in \cF}\EE\big|\Delta_f^S(X_1)\mathds{1}(\nu_f(X_1)=S)\big)\big| & \le \cph\sqrt{n}\sup_{f \in \cF}\|\Delta_f^S(X_1)\mathds{1}(\nu_f(X_1)=S)\big)\|_{\phi, P} \\
& \le \boxed{\sqrt{n} 2^{-S}\sigma \le \frac{a\cph}{2^4}} \,. 
\end{align*}
Here the last bound follows from the definition of $S$. \par 
Combining all the bounds, we get: 
$$
T_2 \le \boxed{\frac{aM}{16}\left(\cph+\frac{\cph+1}{C_0}\right)+b M q \frac{1+H_0}{\sqrt{n}}} \,.
$$
{\bf Bound on $T_3$: } Recall the definition of $T_3$: 
\begin{align*}
T_3 & = \EE\left[\frac{1}{\sqrt{n}}\sup_{f\in\cF}\bigg|\sum_{s=1}^S \sum_{i=1}^n \big(\gtt(X_i)-\EE\gtt(X_1)\big)\bigg|\right] \\
& \le \sum_{s = 1}^S  \EE\left[\sup_{f\in\cF}\frac{1}{\sqrt{n}}\bigg|\sum_{i=1}^n \big(\gtt(X_i)-\EE\gtt(X_1)\big)\bigg|\right] 
\end{align*}
where $\gtt(x)=(\tgl{s}(x)-\tgl{s-1}(x))\mathds{1}(\nu_f(x)\ge s)$. Following similar argument as used in $T_2$, we have: 
\begin{enumerate}
    \item $|\tilde \cF_s| \le \left(\Pi_{k = 0}^s N_k\right)^2 \implies \log{|\tilde \cF_s|} \le 2\sum_{k = 0}^s H_k$. 
    \item $|\gtt(x)| \le K_{s-1}$. 
    \item $\|\gtt(X)\|_{\phi, P} \le 2^{-(s-1)}\sigma$. 
\end{enumerate}
Therefore, again applying the maximal inequality for the finite collection of functions, we have: 
\begin{align*}
T_3  &\; \le M\sum_{s = 1}^S\inf_{1 \le q \le n}\left(2^{-(s-1)}\sigma \ppb(q) \sqrt{1+2\sum_{k = 0}^s H_k} + K_{s-1} q \frac{1+2\sum_{k = 0}^s H_k}{\sqrt{n}} + K_{s-1} \beta_q\sqrt{n}\right)\\ &\le \boxed{\frac{8aM}{C_0}(\cph+16) }\,, 
\end{align*}
where the last bound follows from~\eqref{eq:T21_final}.  Combining the bounds for $T_1$, $T_2$, and $T_3$, completes the proof.
\end{proof}

\subsection{Proof of Theorem \ref{thm:infmainres}}\label{sec:pfinfthm}
The proof of this Theorem is very similar to that of \cref{thm:phimainres}. We only highlight the key differences here. 
In particular, we need slightly different forms of two results used earlier, namely \cref{thm:maximal_finite_davy} and \cref{lem:l1phin}. 

\begin{prop}\label{prop:maxfindavy2}
    Let $\{X_i\}_{i\in\mathbb{N}}$ be a stationary $\beta$-mixing sequence with marginal distribution $P$, taking values in some Polish space $\mcx$. Also let $\cF$ be a finite collection of functions $f$ with $\|f\|_\infty \le \sigma$ for some $\sigma>0$ and all $f \in \cF$. Then we have: 
\begin{align*}
&\;\;\;\;\;\;\EE\left[\max_{f \in \cF} \left|\frac{1}{\sqrt{n}}\sum_{i=1}^n\left(f(X_i) - Pf\right)\right|\right] \\ & \le K\sigma\inf_{1 \le q \le n}\left( \ppbi(q) \sqrt{1+\log{|\cF|}} +  q \frac{1+\log{|\cF|}}{\sqrt{n}} +  \beta_q\sqrt{n}\right),
\end{align*}
for some absolute constant $K>0$, where 
\begin{equation}\label{eq:piqi}
\ppbi(q) := \sqrt{1+2\sum_{k=0}^q \beta_k}.
\end{equation}

\end{prop}

\begin{lemma}\label{lem:liphin2}
    Fix any $\vep\in (0,1)$, $q\in \mathbb{N}$, and let $h\in L_{\infty}(\mcx)$. Recall the definition of $Q_h$, $B_q$, and $\omega_{q,h}$ from \cref{sec:nota}, \eqref{eq:bq}, and \eqref{eq:omqh} respectively. Then 
\begin{align}\label{eq:phin1new}
    \lVert h\mathds{1}(|h|> Q_h(\vep))\rVert_{L_1(P)}\le \frac{2\vep\omega_{q,h}(\vep)}{\sqrt{B_q(\vep)}}.
\end{align}
Suppose now $\upsilon>0$ satisfies 
\begin{align}\label{eq:phin2new}
\upsilon\sqrt{\vep\sum_{k=0}^q \mathds{1}(\vep\le \beta_k)}\ge \lVert h\rVert_{\infty,\mcx}\sqrt{\sum_{k=0}^q \beta_k}.
\end{align}
Then, we have:
\begin{align}\label{eq:phin3new}
    \lVert h\mathds{1}(|h|> \upsilon)\rVert_{L_1(P)}\le 2\lVert h\rVert_{\infty,\mcx}\sqrt{\sum_{k=0}^q \beta_k}\sqrt{\frac{\vep}{\sum_{k=0}^{q} \mathds{1}(\vep\le \beta_k)}}.
\end{align}
\end{lemma}

Let us now complete the proof of \cref{thm:infmainres} using the above results. We proceed exactly as in the proof of \cref{thm:phimainres} by decomposing an upper bound on the LHS of \eqref{eq:target} into terms $T_1$, $T_2$, $T_3$, and further decomposing $T_2$ into $\{T_{2,i}\}_{i\in [4]}$. Instead of invoking \cref{thm:maximal_finite_davy} for bounding $T_1$, $T_{2,1}$, $T_{2,2}$ and $T_3$, we will now use \cref{prop:maxfindavy2}. Similarly, instead of invoking \cref{lem:l1phin} to bound $T_{2,3}$, we will now use \cref{lem:liphin2}. The bound on $T_{2,4}$ stays exactly the same. The rest of the proof is verbatim that of \cref{thm:phimainres}. We leave the rest of the bookkeeping to the reader.

\subsection{Proof of Theorem \ref{thm:gamma_mixing}}\label{sec:pfgamthm}
In this section, we provide a proof of Theorem \ref{thm:gamma_mixing}. While we will still use the blocking and adaptive truncation strategy that was used for the proof of Theorem \ref{thm:phimainres}, a more careful chaining is required to balance all the required contributions. We begin with a modified version of \cref{prop:maxfindavy2}.

\begin{prop}
\label{lem:finite_maximal_gamma}
Suppose $\{X_i\}_{i\in\mathbb{N}}$ be a stationary $\gamma$-mixing sequence with marginal distribution $P$ supported on some Polish space $\mcx$. Let $\cF$ 
be a finite collection of functions on $\mcx$ with $\|f\|_\infty \le b$ and $\|f\|_r \le \sigma$ for all $f \in \cF$, for some $\sigma, b > 0$, $r\ge 1$. Define $\tilde{r}:=r\wedge 2$. Then we have: 
\begin{align*}
&\;\;\;\;\;\;\EE\left[\max_{f \in \cF} \left|\frac{1}{\sqrt{n}}\sum_{i=1}^n\left(f(X_i) - Pf\right)\right|\right] \\ 
& \le K\inf_{1 \le q \le n}\left(b^{1-\frac{\tilde{r}}{2}}\sqrt{\sigma^{\tilde{r}}\sum_{k=1}^{q-1} \gamma_k} \sqrt{1+\log{|\cF|}} + b q \frac{1+\log{|\cF|}}{\sqrt{n}} + b \gamma_q\sqrt{n}\right),
\end{align*}
for some absolute constant $K>0$. 
\end{prop}

We will begin with the same construction of a nested sequence of brackets as in the proof of Theorem \ref{thm:phimainres}. Fix $S\ge 0$ (to be chosen later). Throughout this proof, the generic constant $C$ and the sign $\lesssim$ will be used to denote changing constants depending only on $\alpha, \gamma$, $r$, and $V$. We begin with some preliminary observations. Define $H_k:=\bbH_2(2^{-k}\sigma,\cF)$ where 
$$\bbH_2(u,\cF):=D\left(\log\left(\frac{B}{u}\right)\right)^V\left(\frac{\theta}{u}\right)^{\alpha},$$
for some $B\ge b\vee e\vee \sigma$, $\theta\ge \sigma$, $V,\alpha\ge 0$, and $D\ge 1$. First note that for $s\in [S]$, we have:
\begin{align}\label{eq:prelim1}
    \sum_{k=0}^s H_k\sim D\left(\frac{\theta}{\sigma}\right)^{\alpha}(2^{s\alpha} \vee s) \left[\left(\log{\left(\frac{B}{\sigma}\right)}\right)^V + s^V\right].
\end{align}
Next, note that whenever $1+\sum_{k=0}^s H_k\le n$, then by \eqref{eq:tq_gamma} and \eqref{eq:prelim1}, we have:
\begin{align}
    \tqg{2^{-s}\sigma}&\le 2\left(\frac{n}{1+\sum_{k=0}^s H_k}\right)^{\frac{1}{\gamma+1}} \label{eq:qbound} \\ &\lesssim 2\left(\frac{n \sigma^{\alpha}}{D\theta^{\alpha} (2^{s\alpha}\vee s)}\right)^{\frac{1}{1+\gamma}}\left[\left(\log{\left(\frac{B}{\sigma}\right)}\right)^{-\frac{V}{1+\gamma}}+s^{-\frac{V}{1+\gamma}}\right]\notag.
\end{align}

As a result, we get:

\begin{align}\label{eq:prelim2}
    \tqg{2^{-s}\sigma} \begin{cases} \lesssim 2\left(\frac{n \sigma^{\alpha}}{D\theta^{\alpha} (2^{s\alpha}\vee s)}\right)^{\frac{1}{1+\gamma}}\left[\left(\log{\left(\frac{B}{\sigma}\right)}\right)^{-\frac{V}{1+\gamma}}+s^{-\frac{V}{1+\gamma}}\right] & \mbox{if}\ 1+\sum_{k=0}^s H_k\le n  \\ = 1 & \mbox{otherwise}.\end{cases}
\end{align}

Note that for any $q\ge 1$ and $\gamma\ne 1$, we have:
\begin{align}\label{eq:prelim7}
    \sum_{k=0}^q \gamma_k\lesssim \sum_{k=0}^q (1+k)^{-\gamma}\lesssim q^{1-\gamma} \mathbf{1}(\gamma<1)+\mathbf{1}(\gamma>1).
\end{align}

Next, we come to one of the main differences between the current proof and the proof of \cref{thm:phimainres}, namely the definition of $K_s$ (see \eqref{eq:defks}) which was used for thresholding the increments in the chaining argument in the proof of \cref{thm:phimainres}. With $\tilde{r}=r\wedge 2$ as before, for $0\le s\le S-1$, we will redefine  
\begin{align}\label{eq:choosek}
K_s := \begin{cases} (2^{-s+1}\sigma)^{\frac{\tilde{r}}{2}} \left(\frac{n}{1 + \sum_{k = 0}^{s+1} H_k}\right)^{\frac{\gamma}{2(1+\gamma)}} & \mbox{if} \ 1+\sum_{k=0}^s H_k\le n \\ (2^{-s}\sigma)^{\frac{\tilde{r}}{2}}\sqrt{\frac{n}{1+\sum_{k=0}^s H_k}} & \mbox{otherwise}.\end{cases}
\end{align}
We redefine $\nu_f$ from \eqref{eq:adaptrun} with $K_s$ chosen as above.

\emph{Main proof}. By \eqref{eq:chaindec}, we have:
\begin{align*}
    f(x)=\tgl{0}(x)+\sum_{s=0}^S (f(x)-\tgl{s}(x))\mathds{1}(\nu_f(x)=s)+\sum_{s=1}^S (\tgl{s}(x)-\tgl{s-1}(x))\mathds{1}(\nu_f(x)\ge s),
\end{align*}
where all the notation bears the same meaning. We will use the definition of $\tqg{\cdot}$ from \eqref{eq:tq_gamma}. The definitions of $\Theta_{f,s}^{(1)}$ and $\Theta_{f,s}^{(2)}$ stay the same as in \eqref{eq:THETAS}. 
Consequently the following bound holds:
\begin{align*}
    &\;\;\;\;\;\EE\left[\frac{1}{\sqrt{n}}\sup_{f\in\cF}\bigg|\sum_{i=1}^n (f(X_i)-\EE f(X_1))\bigg|\right] \\ 
    &\le \EE\left[\frac{1}{\sqrt{n}}\sup_{f\in\cF}\bigg|\sum_{i=1}^n (\tgl{0}(X_i)-\EE \tgl{0}(X_1))\bigg|\right]+\EE\left[\frac{1}{\sqrt{n}}\sup_{f\in\cF}\bigg|\sum_{s=0}^S \sum_{i=1}^n \big(\got(X_i)-\EE\got(X_1)\big)\bigg|\right]\\ 
    & +\EE\left[\frac{1}{\sqrt{n}}\sup_{f\in\cF}\bigg|\sum_{s=1}^S \sum_{i=1}^n \big(\gtt(X_i)-\EE\gtt(X_1)\big)\bigg|\right]\\ 
    & =: T_1+T_2+T_3.
\end{align*}

We need to bound the three terms $T_1$, $T_2$, and $T_3$ individually. Let us begin with some preliminary estimates. 
Note that \eqref{eq:prelim2} implies 
\begin{align}\label{eq:prelim4}
    \frac{\tqg{\sigma}(1+H_0)}{\sqrt{n}}&\lesssim n^{\frac{1}{2}-\frac{\gamma}{1+\gamma}}\left[D\left(\frac{\theta}{\sigma}\right)^{\alpha}\left(\log{\left(\frac{B}{\sigma}\right)}\right)^V\right]^{\frac{\gamma}{1+\gamma}}+n^{-1/2}D\left(\frac{\theta}{\sigma}\right)^{\alpha}\left(\log{\left(\frac{B}{\sigma}\right)}\right)^V.
\end{align}
Now, from \eqref{eq:prelim2} and \eqref{eq:choosek}, we note that whenever $\tqg{2^{-s}\sigma}=1$, we have $K_s=(2^{-s}\sigma)^{\frac{\tilde{r}}{2}}\sqrt{\frac{n}{1+\sum_{k=0}^s H_k}}$. Therefore, by combining \eqref{eq:prelim2} with the choice of $K_s$ in \eqref{eq:choosek}, we have the following for $s\in [S]$,
\begin{align}\label{eq:prelim5}
    &\;\;\;\;K_{s-1}\frac{\tqg{2^{-s}\sigma}\left(1+\sum_{k=0}^s H_k\right)}{\sqrt{n}}\notag \\ &= (2^{-s}\sigma)^{\frac{\tilde{r}}{2}}\left(\frac{n}{1+\sum_{k=0}^s H_k}\right)^{\frac{\gamma}{2(\gamma+1)}}\frac{(n r_n)^{\frac{1}{\gamma+1}}}{\sqrt{n}}\left(1+\sum_{k=0}^s H_k\right)^{\frac{\gamma}{1+\gamma}}+(2^{-s}\sigma)^{\frac{\tilde{r}}{2}}\sqrt{1+\sum_{k=0}^s H_k}\notag \\ & \lesssim (2^{-s}\sigma)^{\frac{\tilde{r}}{2}}n^{\frac{1}{2(1+\gamma)}}\left[D\left(\frac{\theta}{\sigma}\right)^{\alpha}(2^{s\alpha}\vee s)\right]^{\frac{\gamma}{2(\gamma+1)}}\left\{\left(\log{\left(\frac{B}{\sigma}\right)}\right)^{\frac{V\gamma}{2(\gamma+1)}} + s^{\frac{V\gamma}{2(\gamma+1)}} \right\}\notag \\ &\;\;\;\;+(2^{-s}\sigma)^{\frac{\tilde{r}}{2}}\sqrt{D\left(\frac{\theta}{\sigma}\right)^{\alpha}(2^{s\alpha}\vee s)}\left\{\left(\log{\left(\frac{B}{\sigma}\right)}\right)^{\frac{V}{2}} + s^{\frac{V}{2}} \right\}.
\end{align}
In the last inequality above, we have used \eqref{eq:prelim1}.
Finally, we will also use the following: 
\begin{align}
\label{eq:prelim6}
&\;\;\;\sqrt{\sum_{k=0}^{\tqg{2^{-s}\sigma}} \gamma_k}\,\,\,\,\,\sqrt{1+\sum_{k=0}^s H_k}\notag \\ 
&\lesssim \left(\tqg{2^{-s}\sigma}^{\frac{1-\gamma}{2}}\sqrt{1+\sum_{k=0}^s H_k}\right) \ \ \mathds{1}_{\gamma < 1}+ \sqrt{1+\sum_{k=0}^s H_k} \ \ \mathds{1}_{\gamma > 1}\notag \\ 
&\lesssim \left(\left(\frac{n}{1+\sum_{k=0}^s H_k}\right)^{\frac{1- \gamma}{2(\gamma+1)}}\sqrt{1+\sum_{k=0}^s H_k}\right) \ \ \mathds{1}_{\gamma < 1, 1+\sum_{k=0}^s H_k \le n}+ \sqrt{1+\sum_{k=0}^s H_k} \notag \\
& \lesssim n^{\frac{1-\gamma}{2(1 +\gamma)}} \left(1+\sum_{k=0}^s H_k\right)^{\frac{\gamma}{1 + \gamma}} \mathds{1}_{\gamma< 1, 1+\sum_{k=0}^s H_k \le n} + \sqrt{1+\sum_{k=0}^s H_k} \notag \\ 
& \lesssim n^{\frac{1}{2(1 +\gamma)}} \left(1+\sum_{k=0}^s H_k\right)^{\frac{\gamma}{2(1 + \gamma)}} \mathds{1}_{\gamma < 1, 1+\sum_{k=0}^s H_k \le n} + \sqrt{1+\sum_{k=0}^s H_k} \notag \\ 
&\lesssim n^{\frac{1}{2(1+\gamma)}}\left[D\left(\frac{\theta}{\sigma}\right)^{\alpha}(2^{s\alpha}\vee s)\right]^{\frac{\gamma}{2(\gamma+1)}}\left\{\left(\log{\left(\frac{B}{\sigma}\right)}\right)^{\frac{V\gamma}{2(\gamma+1)}} + s^{\frac{V\gamma}{2(\gamma+1)}} \right\}\mathbf{1}(\gamma<1)\notag \\ &\;\;\;\;+\sqrt{D\left(\frac{\theta}{\sigma}\right)^{\alpha}(2^{s\alpha}\vee s)}\left[\left(\log{\left(\frac{B}{\sigma}\right)}\right)^{\frac{V}{2}}+s^{\frac{V}{2}}\right].
\end{align}
The first inequality above follows from \eqref{eq:prelim7}, the second inequality uses \eqref{eq:qbound} and the fact that $\tqg \lesssim 1$ when $n < 1+\sum_{k=0}^s H_k$. The fourth inequality uses the fact that when $n \ge 1+\sum_{k=0}^s H_k$, 
$$
n^{\frac{1-\gamma}{2(1 +\gamma)}} \left(1+\sum_{k=0}^s H_k\right)^{\frac{\gamma}{1 + \gamma}} \le n^{\frac{1}{2(1 +\gamma)}} \left(1+\sum_{k=0}^s H_k\right)^{\frac{\gamma}{2(1 + \gamma)}} \,.
$$
The last inequality used the bound on $\sum_k H_k$, namely \eqref{eq:prelim1}. 
\\\\
\noindent
\textbf{Bound on $T_1$}. By \cref{lem:finite_maximal_gamma}, we then have:
\begin{align*}
   T_1 & = \bbE\left[\sup_{f \in \cF} \frac{1}{\sqrt{n}}\sum_i\left|\tgl{0}(X_i) - \EE\tgl{0}(X)\right|\right] \notag \\
   & \lesssim \inf_{1 \le q \le n}\left(\sqrt{\sigma^{\tilde{r}}\sum_{k=0}^q \gamma_k}\sqrt{1+H_0} + b q \frac{1+H_0}{\sqrt{n}} + b \gamma_q\sqrt{n}\right) \notag \\
   & \lesssim \left(\sqrt{\sigma^{\tilde{r}}\left(\tqg{\sigma}^{1-\gamma} \mathbf{1}(\gamma<1)+\mathbf{1}(\gamma>1)\right)}\sqrt{1+H_0} + 2b \tqg{\sigma} \frac{1+H_0}{\sqrt{n}}\right).
\end{align*}

The first inequality above uses \cref{lem:finite_maximal_gamma}. The second inequality uses \eqref{eq:prelim7} and the choice $q\equiv \tqg{\sigma}$. Therefore, by using \eqref{eq:prelim4} and \eqref{eq:prelim6}, we get: 
\begin{equation*}
\boxed{
\begin{aligned}
T_1&\lesssim n^{\frac{1}{2(1+\gamma)}}\sigma^{\frac{\tilde{r}}{2}}\left[D\left(\frac{\theta}{\sigma}\right)^{\alpha}\right]^{\frac{\gamma}{2(\gamma+1)}}\left(\log{\left(\frac{B}{\sigma}\right)}\right)^{\frac{V\gamma}{2(\gamma+1)}}+\sigma^{\frac{\tilde{r}}{2}}\sqrt{D\left(\frac{\theta}{\sigma}\right)^{\alpha}\left(\log{\left(\frac{B}{\sigma}\right)}\right)^V}\\ &\;\;\;+n^{\frac{1}{2}-\frac{\gamma}{1+\gamma}}\left(D\left(\frac{\theta}{\sigma}\right)^{\alpha}\left(\log{\left(\frac{B}{\sigma}\right)}\right)^V\right)^{\frac{\gamma}{1+\gamma}}+n^{-1/2}D\left(\frac{\theta}{\sigma}\right)^{\alpha}\left(\log{\left(\frac{B}{\sigma}\right)}\right)^V.
\end{aligned}
}
\end{equation*}

\textbf{Bound on $T_2$}. Just as in \eqref{eq:T2-2}, it suffices to bound $\{T_{2,j}\}_{j=1,2,3,4}$. We begin with bounding $T_{2,1}$.

With $K_s$ defined as in \eqref{eq:choosek}, the same calculation as in the proof of \cref{thm:phimainres} yields:
\begin{align*}
    T_{2,1} & =\sum_{s = 1}^S \EE\left[\frac{1}{\sqrt{n}}\sup_{f\in\cF}\bigg|\sum_{i=1}^n \big(\Delta_f^s(X_i)\mathds{1}(\nu_f(X_i)=s) -\EE\Delta_f^s(X_1)\mathds{1}(\nu_f(X_1)=s)\big)\bigg|\right]\notag \\ 
    & \lesssim \sum_{s = 1}^S\inf_{1 \le q \le n}\left(\sqrt{(2^{-s}\sigma)^{\tilde{r}} \sum_{k=0}^q \gamma_k}\sqrt{1+\sum_{k = 0}^s H_k} + K_{s-1} q \frac{1+\sum_{k = 0}^s H_k}{\sqrt{n}} + K_{s-1} \gamma_q\sqrt{n}\right) \\
    & \lesssim \sum_{s = 1}^S \bigg(\sqrt{(2^{-s}\sigma)^{\tilde{r}} \left(\tqg{2^{-s}\sigma}\mathbf{1}(\gamma<1)+\mathbf{1}(\gamma>1)\right)}\sqrt{1+\sum_{k = 0}^s H_k} \notag \\ &+ 2K_{s-1} \tqg{2^{-s}\sigma} \frac{1 + \sum_{k = 0}^s H_k}{\sqrt{n}}\bigg)
\end{align*}
The first inequality above uses   \cref{lem:finite_maximal_gamma}. The second inequality uses \eqref{eq:prelim7} and the choice $q\equiv \tqg{2^{-s}\sigma}$. Therefore, by using \eqref{eq:prelim5} and \eqref{eq:prelim6}, we get: 
\begin{equation*}
\boxed{
\begin{aligned}
    T_{2,1}&\lesssim \sum_{s=1}^S (2^{-s}\sigma)^{\frac{\tilde{r}}{2}}n^{\frac{1}{2(1+\gamma)}}\left[D\left(\frac{\theta}{\sigma}\right)^{\alpha}(2^{s\alpha}\vee s)\right]^{\frac{\gamma}{2(\gamma+1)}}\left\{\left(\log{\left(\frac{B}{\sigma}\right)}\right)^{\frac{V\gamma}{2(\gamma+1)}} + s^{\frac{V\gamma}{2(\gamma+1)}} \right\}\notag \\ &\;\;\;\;+\sum_{s=1}^S (2^{-s}\sigma)^{\frac{\tilde{r}}{2}}\sqrt{D\left(\frac{\theta}{\sigma}\right)^{\alpha}(2^{s\alpha}\vee s)}\left\{\left(\log{\left(\frac{B}{\sigma}\right)}\right)^{\frac{V}{2}} + s^{\frac{V}{2}} \right\}.
\end{aligned}
}
\end{equation*}

We now move on to the bound of $T_{2,2}$. As it only involves the $s=0$-th level of chaining, i.e., a supremum over $\exp(H_0)$ functions, the bound for $T_{2,2}$ will be the exact same as $T_1$. In other words, we have:
\begin{equation*}
\boxed{
\begin{aligned}
T_{2,2}&\lesssim n^{\frac{1}{2(1+\gamma)}}\sigma^{\frac{\tilde{r}}{2}}\left[D\left(\frac{\theta}{\sigma}\right)^{\alpha}\right]^{\frac{\gamma}{2(\gamma+1)}}\left(\log{\left(\frac{B}{\sigma}\right)}\right)^{\frac{V\gamma}{2(\gamma+1)}}+\sigma^{\frac{\tilde{r}}{2}}\sqrt{D\left(\frac{\theta}{\sigma}\right)^{\alpha}\left(\log{\left(\frac{B}{\sigma}\right)}\right)^V}\\ &\;\;\;+n^{\frac{1}{2}-\frac{\gamma}{1+\gamma}}\left(D\left(\frac{\theta}{\sigma}\right)^{\alpha}\left(\log{\left(\frac{B}{\sigma}\right)}\right)^V\right)^{\frac{\gamma}{1+\gamma}}+n^{-1/2}D\left(\frac{\theta}{\sigma}\right)^{\alpha}\left(\log{\left(\frac{B}{\sigma}\right)}\right)^V.
\end{aligned}
}   
\end{equation*}

Next, we go to the bound on $T_{2, 3}$. Our key tool is the following elementary inequality: $\EE[|X|\mathbf{1}(|X|\ge a)]\le a^{-1}\EE X^2$ for $a>0$. Applying this inequality yields:
\begin{align*}
     T_{2,3} &=\sqrt{n}\sum_{s = 0}^{S-1} \sup_{f \in \cF}\EE\big|\Delta_f^s(X_1)\mathds{1}(\nu_f(X_1)=s)\big)\big| \\ 
     &=\sqrt{n}\sum_{s = 0}^{S-1} \sup_{f \in \cF}\EE\big|\Delta_f^s(X_1)\mathds{1}(\Delta_f^s(X_1) > K_s)\big)\big| \\
     & \le \sqrt{n}\sum_{s = 0}^{S-1} \frac{\bbE[(\Delta_f^s(X_1))^2]}{K_s} \le \sqrt{n}\sum_{s=0}^{S-1} \frac{(2^{-s}\sigma)^{\tilde{r}}}{K_s}\\
     & \lesssim \sum_{s=0}^{S-1} (2^{-s}\sigma)^{\frac{\tilde{r}}{2}}\sqrt{1+\sum_{k=0}^{s+1} H_k}+n^{\frac{1}{2(1+\gamma)}}\sum_{s=0}^{S-1} (2^{-s}\sigma)^{\frac{\tilde{r}}{2}}\left(1+\sum_{k=0}^{s+1} H_k\right)^{\frac{\gamma}{2(\gamma+1)}}\\
\end{align*}

By combining the above inequality with \eqref{eq:prelim1}, we get:
\begin{equation*}
\boxed{
\begin{aligned}
    T_{2,3}&\lesssim \sum_{s=1}^S (2^{-s}\sigma)^{\frac{\tilde{r}}{2}}n^{\frac{1}{2(1+\gamma)}}\left[D\left(\frac{\theta}{\sigma}\right)^{\alpha}(2^{s\alpha}\vee s)\right]^{\frac{\gamma}{2(\gamma+1)}}\left\{\left(\log{\left(\frac{B}{\sigma}\right)}\right)^{\frac{V\gamma}{2(\gamma+1)}} + s^{\frac{V\gamma}{2(\gamma+1)}} \right\}\notag \\ &\;\;\;\;+\sum_{s=1}^S (2^{-s}\sigma)^{\frac{\tilde{r}}{2}}\sqrt{D\left(\frac{\theta}{\sigma}\right)^{\alpha}(2^{s\alpha}\vee s)}\left\{\left(\log{\left(\frac{B}{\sigma}\right)}\right)^{\frac{V}{2}} + s^{\frac{V}{2}} \right\}.
\end{aligned}
}
\end{equation*}

Finally we provide an upper bound on $T_{2, 4}$ to complete the task of bounding $T_2$: 
\begin{align*}
     T_{2,4} &=\sqrt{n}\sup_{f \in \cF}\EE\big|\Delta_f^S(X_1)\mathds{1}(\nu_f(X_1)=S)\big)\big| \\
     & \lesssim \sqrt{n}\sup_{f \in \cF}\|\Delta_f^S(X_1)\mathds{1}(\nu_f(X_1)=S)\big)\|_{\tilde{r}, P} \\
     & \le \boxed{\sqrt{n} 2^{-S}\sigma} \,.
\end{align*}
Now that we have bounded $\{T_{2,j}\}_{j=1,2,3,4}$, we can combine all the terms to bound $T_2$. This will yield:

\begin{equation*}
    \boxed{
    \begin{aligned}
        T_2&\lesssim \sum_{s=0}^S (2^{-s}\sigma)^{\frac{\tilde{r}}{2}}n^{\frac{1}{2(1+\gamma)}}\left[D\left(\frac{\theta}{\sigma}\right)^{\alpha}(2^{s\alpha}\vee s)\right]^{\frac{\gamma}{2(\gamma+1)}}\left\{\left(\log{\left(\frac{B}{\sigma}\right)}\right)^{\frac{V\gamma}{2(\gamma+1)}} + s^{\frac{V\gamma}{2(\gamma+1)}} \right\}\notag \\ &\;\;\;\;+\sum_{s=0}^S (2^{-s}\sigma)^{\frac{\tilde{r}}{2}}\sqrt{D\left(\frac{\theta}{\sigma}\right)^{\alpha}(2^{s\alpha}\vee s)}\left\{\left(\log{\left(\frac{B}{\sigma}\right)}\right)^{\frac{V}{2}} + s^{\frac{V}{2}} \right\}+\sqrt{n}2^{-S}\sigma \\ 
        &\;\;\;\;+n^{\frac{1}{2}-\frac{\gamma}{1+\gamma}}\left(D\left(\frac{\theta}{\sigma}\right)^{\alpha}\left(\log{\left(\frac{B}{\sigma}\right)}\right)^V\right)^{\frac{\gamma}{1+\gamma}}+n^{-1/2}D\left(\frac{\theta}{\sigma}\right)^{\alpha}\left(\log{\left(\frac{B}{\sigma}\right)}\right)^V
    \end{aligned}
    }
\end{equation*}
We now move on to bounding $T_3$. 

\textbf{Bound for $T_3$}. 
The initial steps are exactly the same as in the proof of \cref{thm:phimainres}. In particular, 
\begin{align*}
     T_3  & = \EE\left[\frac{1}{\sqrt{n}}\sup_{f\in\cF}\bigg|\sum_{s=1}^S \sum_{i=1}^n \big(\gtt(X_i)-\EE\gtt(X_1)\big)\bigg|\right]\\ &\lesssim  \sum_{s = 1}^S\inf_{1 \le q \le n}\left(\sqrt{(2^{-s}\sigma)^{\tilde{r}} \sum_{k=0}^q \gamma_k}\sqrt{1+\sum_{k = 0}^s H_k} + K_{s-1} q \frac{1+\sum_{k = 0}^s H_k}{\sqrt{n}} + K_{s-1} \gamma_q\sqrt{n}\right)\,.
\end{align*}
It is evident from the above upper bound on $T_3$ that it is similar to that of $T_{2, 1}$. Therefore we obtain: 
\begin{equation*}
\boxed{
\begin{aligned}
    T_{3}&\lesssim \sum_{s=1}^S (2^{-s}\sigma)^{\frac{\tilde{r}}{2}}n^{\frac{1}{2(1+\gamma)}}\left[D\left(\frac{\theta}{\sigma}\right)^{\alpha}(2^{s\alpha}\vee s)\right]^{\frac{\gamma}{2(\gamma+1)}}\left\{\left(\log{\left(\frac{B}{\sigma}\right)}\right)^{\frac{V\gamma}{2(\gamma+1)}} + s^{\frac{V\gamma}{2(\gamma+1)}} \right\}\notag \\ &\;\;\;\;+\sum_{s=1}^S (2^{-s}\sigma)^{\frac{\tilde{r}}{2}}\sqrt{D\left(\frac{\theta}{\sigma}\right)^{\alpha}(2^{s\alpha}\vee s)}\left\{\left(\log{\left(\frac{B}{\sigma}\right)}\right)^{\frac{V}{2}} + s^{\frac{V}{2}} \right\}.
\end{aligned}
}
\end{equation*}
Combining all the bounds obtained so far, we have: 
\begin{equation*}
    \boxed{
    \begin{aligned}
        T_1 + T_2 + T_3 &\lesssim \sum_{s=0}^S (2^{-s}\sigma)^{\frac{\tilde{r}}{2}}n^{\frac{1}{2(1+\gamma)}}\left[D\left(\frac{\theta}{\sigma}\right)^{\alpha}(2^{s\alpha}\vee s)\right]^{\frac{\gamma}{2(\gamma+1)}}\left\{\left(\log{\left(\frac{B}{\sigma}\right)}\right)^{\frac{V\gamma}{2(\gamma+1)}} + s^{\frac{V\gamma}{2(\gamma+1)}} \right\}\notag \\ &\;\;\;\;+\sum_{s=0}^S (2^{-s}\sigma)^{\frac{\tilde{r}}{2}}\sqrt{D\left(\frac{\theta}{\sigma}\right)^{\alpha}(2^{s\alpha}\vee s)}\left\{\left(\log{\left(\frac{B}{\sigma}\right)}\right)^{\frac{V}{2}} + s^{\frac{V}{2}} \right\}+\sqrt{n}2^{-S}\sigma \\ 
        &\;\;\;\;+n^{\frac{1}{2}-\frac{\gamma}{1+\gamma}}\left(D\left(\frac{\theta}{\sigma}\right)^{\alpha}\left(\log{\left(\frac{B}{\sigma}\right)}\right)^V\right)^{\frac{\gamma}{1+\gamma}}+n^{-1/2}D\left(\frac{\theta}{\sigma}\right)^{\alpha}\left(\log{\left(\frac{B}{\sigma}\right)}\right)^V
    \end{aligned}
    }
\end{equation*}
Note that we still have a parameter $S\ge 0$ to choose. Now that we have obtained bounds from $T_1, T_2, T_3$, we now highlight the choice of $S$ for different regimes of $\alpha$, $r$, and $\gamma$. 

\subsubsection{Rates for Donsker class when $0 \le \alpha < \tilde{r}$}
In this subsection, we consider the case when $0 \le \alpha < 2$, which implies $\alpha < \tilde{r}(1+\gamma)/\gamma$. By choosing $S\ge 0$ satisfying 
$$2^S > n^{\frac{\gamma}{2(\gamma+1)}}\sigma^{1-\frac{\tilde{r}}{2}+\frac{\alpha \gamma}{2(\gamma+1)}}\theta^{-\frac{\alpha\gamma}{2(\gamma+1)}},$$
combined with the bounds for $T_1, T_2$, and $T_3$, we get:
    \begin{align*}
        &\;\;\;\;T_1+T_2+T_3\\ &\lesssim \sigma^{\frac{\tilde{r}}{2}}\left[D\left(\frac{\theta}{\sigma}\right)^{\alpha}\right]^{\frac{\gamma}{2(\gamma+1)}}n^{\frac{1}{2(1+\gamma)}}\left(\log{\left(\frac{B}{\sigma}\right)}\right)^{\frac{V\gamma}{2(\gamma+1)}}+\sigma^{\frac{\tilde{r}}{2}}\sqrt{D\left(\frac{\theta}{\sigma}\right)^{\alpha}}\left(\log{\left(\frac{B}{\sigma}\right)}\right)^{\frac{V}{2}}\\ &\;\;\;\;+n^{\frac{1}{2}-\frac{\gamma}{1+\gamma}}\left(D\left(\frac{\theta}{\sigma}\right)^{\alpha}\left(\log{\left(\frac{B}{\sigma}\right)}\right)^V\right)^{\frac{\gamma}{1+\gamma}}+n^{-1/2}D\left(\frac{\theta}{\sigma}\right)^{\alpha}\left(\log{\left(\frac{B}{\sigma}\right)}\right)^V 
    \end{align*}

   As $D(\theta/\sigma)^{\alpha}(\log{(B/\sigma)})^V\lesssim n$, we observe the following inequalities: 
   \begin{align}
   \label{eq:adjust1}
       &\;\;\;\;\left[D\left(\frac{\theta}{\sigma}\right)^{\alpha}\right]^{\frac{\gamma}{2(\gamma+1)}}n^{\frac{1}{2(1+\gamma)}}\left(\log{\left(\frac{B}{\sigma}\right)}\right)^{\frac{V\gamma}{2(\gamma+1)}} \notag \\ 
       &\gtrsim \left[D\left(\frac{\theta}{\sigma}\right)^{\alpha}\right]^{\frac{\gamma}{2(\gamma+1)}+\frac{1}{2(\gamma+1)}}\left(\log{\left(\frac{B}{\sigma}\right)}\right)^{\frac{V\gamma}{2(\gamma+1)}+\frac{V}{2(\gamma+1)}}=\sqrt{D\left(\frac{\theta}{\sigma}\right)^{\alpha}}\left(\log{\left(\frac{B}{\sigma}\right)}\right)^{\frac{V}{2}},
   \end{align}
   and 
   \begin{align}
   \label{eq:adjust2}
       &\;\;\;\;n^{\frac{1}{2}-\frac{\gamma}{1+\gamma}}\left(D\left(\frac{\theta}{\sigma}\right)^{\alpha}\left(\log{\left(\frac{B}{\sigma}\right)}\right)^V\right)^{\frac{\gamma}{1+\gamma}}\notag \\ 
       &\gtrsim n^{-\frac{1}{2}}\left[D\left(\frac{\theta}{\sigma}\right)^{\alpha}\right]^{\frac{\gamma}{1+\gamma}+\frac{1}{\gamma}}\left(\log{\left(\frac{B}{\sigma}\right)}\right)^{\frac{V\gamma}{1+\gamma}+\frac{V}{1+\gamma}}=n^{-1/2}D\left(\frac{\theta}{\sigma}\right)^{\alpha}\left(\log{\left(\frac{B}{\sigma}\right)}\right)^V.
   \end{align}
   Combining the above observations yields a conclusion.
\subsubsection{Non-Donsker class: when $\tilde{r} < \alpha < \tilde{r}(1+\gamma)/\gamma$}
In this subsection, we consider the non-Donsker class when $\tilde{r} < \alpha < \tilde{r}(1+\gamma)/\gamma$. By choosing $S\ge 0$ satisfying 
$$2^S\sim n^{\frac{1}{\alpha+2-\tilde{r}}}\sigma \left(D\theta^{\alpha}\right)^{-\frac{1}{\alpha+2-\tilde{r}}},$$
combined with the bounds for $T_1$, $T_2$, and $T_3$, we get:
\begin{align*}
        &\;\;\;\;T_1+T_2+T_3\\ &\lesssim \sigma^{\frac{\tilde{r}}{2}}\left[D\left(\frac{\theta}{\sigma}\right)^{\alpha}\right]^{\frac{\gamma}{2(\gamma+1)}}n^{\frac{1}{2(1+\gamma)}}\left(\log{\left(\frac{B}{\sigma}\right)}\right)^{\frac{V\gamma}{2(\gamma+1)}}+\left(D\theta^{\alpha}\right)^{-\frac{1}{\alpha+2-\tilde{r}}}\left(\log{\left(\frac{B}{\sigma}\right)}\right)^{\frac{V}{2}}n^{\frac{\alpha-\tilde{r}}{2(\alpha+2-\tilde{r})}}\\ &\;\;\;\;+n^{\frac{1}{2}-\frac{\gamma}{1+\gamma}}\left(D\left(\frac{\theta}{\sigma}\right)^{\alpha}\left(\log{\left(\frac{B}{\sigma}\right)}\right)^V\right)^{\frac{\gamma}{1+\gamma}}+n^{-1/2}D\left(\frac{\theta}{\sigma}\right)^{\alpha}\left(\log{\left(\frac{B}{\sigma}\right)}\right)^V \\ &\;\;\;\;+\left(D\theta^{\alpha}\right)^{-\frac{1}{\alpha+2-\tilde{r}}}n^{\frac{\alpha-\tilde{r}}{2(\alpha+2-\tilde{r})}}\left(\log{\big(n^{\frac{1}{\alpha+2-\tilde{r}}}\sigma \theta^{-\frac{\alpha}{\alpha+2-\tilde{r}}}\big)}\right)^{\frac{V}{2}}+\sigma^{\frac{\tilde{r}}{2}}\sqrt{D\left(\frac{\theta}{\sigma}\right)^{\alpha}\left(\log{\left(\frac{B}{\sigma}\right)}\right)^V}.
    \end{align*}
As $\sigma\gtrsim n^{-\frac{1}{\alpha+2-\tilde{r}}}$, the last term above is non-negative. The conclusion now follows by combining \eqref{eq:adjust1}, \eqref{eq:adjust2}, and noting that 
$$\log{\big(n^{\frac{1}{\alpha+2-\tilde{r}}}\sigma \theta^{-\frac{\alpha}{\alpha+2-\tilde{r}}}\big)}\le \log{\big(n^{\frac{1}{\alpha+2-\tilde{r}}}\sigma \big)} \le  \log{\big(n^{\frac{1}{\alpha+2-\tilde{r}}} B\big)}\le \log{\left(\frac{B}{\sigma}\right)}.$$

\subsubsection{Non-Donsker class: when $\alpha > \tilde{r}(1 + \gamma)/\gamma$}
We end this proof with the final case when $\alpha > \tilde{r}(1+\gamma)/\gamma$. By choosing $S\ge 0$ satisfying 
$$2^S\sim n^{\frac{1}{\alpha+\frac{(2-\tilde{r})(\gamma+1)}{\gamma}}}\sigma \left(D\theta^{\alpha}\right)^{-\frac{1}{\alpha+\frac{(2-\tilde{r})(\gamma+1)}{\gamma}}},$$
combined with the bounds for $T_1$, $T_2$, and $T_3$, we get, as above:
\begin{align*}
        &\;\;\;\;T_1+T_2+T_3\\ &\lesssim n^{\frac{\gamma(\alpha-\tilde{r})+(2-\tilde{r})}{2(\alpha \gamma+(2-\tilde{r})(\gamma+1))}}\left( D \theta^{\alpha}\right)^{\frac{\gamma}{\alpha \gamma+(2-\tilde{r})(\gamma+1)}}\left(\log{\left(\frac{B}{\sigma}\right)}\right)^{\frac{V\gamma}{2(\gamma+1)}}+\sigma^{\frac{\tilde{r}}{2}}\sqrt{D\left(\frac{\theta}{\alpha}\right)^{\alpha}\left(\log{\left(\frac{B}{\sigma}\right)}\right)^V}\\ &+n^{\frac{1}{2}-\frac{\gamma}{1+\gamma}}\left(D\left(\frac{\theta}{\sigma}\right)^{\alpha}\left(\log{\left(\frac{B}{\sigma}\right)}\right)^V\right)^{\frac{\gamma}{\gamma+1}}+n^{\frac{1}{2(1+\gamma)}}\sigma^{\frac{\tilde{r}}{2}}\left(D\left(\frac{\theta}{\sigma}\right)^{\alpha}\left(\log{\left(\frac{B}{\sigma}\right)}\right)^V\right)^{\frac{\gamma}{2(\gamma+1)}}\\ &+n^{\frac{\gamma(\alpha-\tilde{r})}{2(\alpha \gamma + (2-\tilde{r})(\gamma+1))}}\left( D \theta^{\alpha}\right)^{\frac{2\gamma+2-\tilde{r}}{2(\alpha \gamma + (2-\tilde{r})(\gamma+1))}}\left(\log{\left(\frac{B}{\sigma}\right)}\right)^{\frac{V}{2}}+n^{-1/2}D\left(\frac{\theta}{\sigma}\right)^{\alpha}\left(\log{\left(\frac{B}{\sigma}\right)}\right)^V
    \end{align*}
    In the last inequality, we have used the bound $\sigma\gtrsim n^{-\frac{1}{\alpha+(2-\tilde{r})(1+\gamma^{-1})}}$, \eqref{eq:adjust1}, and \eqref{eq:adjust2}. This completes the proof.

    \subsection{Proof of Theorem \ref{thm:roc}}
\label{sec:proof_roc}
The proof of Theorem \ref{thm:roc} follows from standard peeling argument. We provide the proof here for completeness. 

Note that by definition of $\hat f_n$, we have: 
$$
\hat f_n \in \argmin_{f \in \cF} \bbP_n(\ell(Y, f(X))) \,.
$$
Fix $s>1$ large enough to be fixed later. For the rest of the proof, define $A_j := \{f: 2^{2j}s^2\delta_n^2 \le \cE(\hat{f}_n) \le 2^{2j+2}s^2\delta_n^2: f \in \cF\}$. We then have: 
\begin{align*}
    & \bbP\left(\cE(\hat{f}_n) > s\delta_n\right) \\
    & \le \bbP\left(\sup_{f: \cE(f) > s\delta_n} \bbP_n\left(\ell(Y, f_*(X)) - \ell(Y, f(X))\right) \ge 0\right) \\
    & \le \sum_{j = 0}^\infty \bbP\left(\sup_{f \in A_j} \bbP_n\left(\ell(Y, f_*(X)) - \ell(Y, f(X))\right) \ge 0\right) \\
    & \le  \sum_{j = 0}^\infty \bbP\left(\sup_{f \in A_j} (\bbP_n - \bbP)\left(\ell(Y, f_*(X)) - \ell(Y, f(X))\right) \ge \inf_{f \in A_j} \bbP\left(\ell(Y, f(X)) - \ell(Y, f_*(X))\right)\right) \\
    & \le  \sum_{j = 0}^\infty \frac{\bbE\left[\sup_{\cE(f) \le 2^{2j+2}s^2\delta_n^2} \left|(\bbP_n - \bbP)\left(\ell(Y, f_*(X)) - \ell(Y, f(X))\right)\right|\right]}{\inf_{f \in A_j}\cE(f)} \\
    & \le \sum_{j = 0}^\infty \frac{\bbE\sup_{\lVert f-f_\star\rVert_{L_2(P_X)\le c^{-1/2}2^{j+1}s\delta_n}}\big|\mathbb{G}_n(g)\big|}{\sqrt{n} 2^{2j+1}s^2\delta_n^2}
\end{align*}
In the last display, we have used the fact that $\cE(f)\ge c \lVert f-f_\star\rVert^2_{L_2(P_X)}$. Now by the $L$-Lipschitzness of $\ell(\cdot,\cdot)$, we have:
$$\lVert \ell(Y,f(X))-\ell(Y,f_\star(X))\rVert^2_{L_2(P)}\le L^2 \lVert f-f_\star\rVert^2_{L_2(P_X)}\le L^2 c^{-1}2^{2j+2}s^2\delta_n^2.$$
Now by the definition of $\mathcal{G}_{\delta}$ in \cref{thm:roc}, the above displays imply that 
$$\bbP\left(\cE(\hat{f}_n) > s\delta_n\right)\le \sum_{j = 0}^\infty \frac{\bbE\sup_{g\in\mathcal{G}_{\delta}}\big|\mathbb{G}_n(g)\big|}{\sqrt{n} 2^{2j+1}s^2\delta_n^2}.$$
By applying \cref{thm:gamma_mixing} and abbreviating $\Pi_n(\mathcal{G}_{\delta})$ as $\Pi_n(\delta)$ (defined in \cref{thm:gamma_mixing}), we get:
\begin{align*}
\bbP\left(\cE(\hat{f}_n) > s\delta_n\right)\le \sum_{j = 0}^\infty \frac{1}{\sqrt{n} 2^{2j+1}s^2\delta_n^2} \bbE\sup_{g\in\mathcal{G}_{L c^{-1/2}2^{j+1}s\delta_n}}\big|\mathbb{G}_n(g)\big|\le \sum_{j = 0}^\infty \frac{\Pi_n\left(L c^{-\frac{1}{2}}2^{j+1}s\delta_n\right)}{\sqrt{n} 2^{2j+2}s^2\delta_n^2}.
\end{align*}
Choose $s>2\sqrt{c}/L$, then by using the fact that $\Pi_n(\mathcal{G}_{\delta})/\delta^t$ is decreasing for some $t\in (0,2)$ and $\Pi_n(\delta_n)\le \sqrt{n}\delta_n^2$, we have:
\begin{align*}
    \bbP\left(\cE(\hat{f}_n) > s\delta_n\right)\le\sum_{j=0}^{\infty}\frac{(L C^{-\frac12}2^{j+1}s)^{t}\sqrt{n}\delta_n^2}{\sqrt{n}2^{2j+2}s^2\delta_n^2}\lesssim s^{-(2-t)},
\end{align*}
where the $\lesssim$ hides constants free of $s$. The moment bound follows by integrating the above  tail bound.

\section{Proofs of Corollaries from main paper} 

\begin{proof}[Proof of \cref{cor:nondonsk_1}, Part (1)]
With $r>2$, if $\beta>\frac{r}{r-2}$, then note that for $q\ge 1$, by \eqref{eq:piv}, we get:
\begin{align*}
    \Lbp(q) \lesssim \int_0^1 u^{-\frac{2}{r}}\,du + \sum_{i=2}^{q+1} \int_0^{i^{-\beta}} u^{-\frac{2}{r}}\,du \lesssim 1+\sum_{i=1}^{\infty} i^{-\beta\left(1-\frac{2}{r}\right)}\lesssim 1, 
\end{align*}
where all the implicit constants are free of $q$. The above display implies that $\Ld{\cdot}$ can be chosen as a constant function and consequently \cref{assm:bsbd} holds. Further note that from \eqref{eq:tqd}, we have:
$$\tq{\sigma}\lesssim 1\vee \left(\frac{n}{1+\mbH{\sigma}}\right)^{\frac{1}{1+\beta}}.$$
Consequently,
\begin{align}\label{eq:therrterm}
\frac{b \tq{\sigma}(1+\mbH{\sigma})}{\sqrt{n}}&\lesssim \frac{b}{\sqrt{n}}(1+\mbH{\sigma})+ b n^{\frac{1}{2}-\frac{\beta}{1+\beta}}(1+\mbH{\sigma})^{\frac{\beta}{1+\beta}}\nonumber \\ &\lesssim \frac{b \sigma^{-\alpha}}{\sqrt{n}}+b n^{\frac{1}{2}-\frac{\beta}{1+\beta}}\sigma^{-\frac{\alpha \beta}{1+\beta}}\lesssim b n^{\frac{1}{2}-\frac{\beta}{1+\beta}}\sigma^{-\frac{\alpha \beta}{1+\beta}},
\end{align}
where we have used $\sigma\gtrsim n^{-\frac{1}{\alpha}}$ in the last inequality. With these observations, we now split the proof into two cases.

\emph{When $\alpha<2$}. We observe that 
$$\int_{\frac{a}{2^6\sqrt{n}}}^{\sigma} \sqrt{\mbH{u}}\,du\le \int_{\frac{a}{2^6\sqrt{n}}}^{\sigma} u^{-\frac{\alpha}{2}}\,du\lesssim \sigma^{1-\frac{\alpha}{2}}\lesssim \sqrt{n}\sigma.$$
Here we have used $\sigma\gtrsim n^{-\frac{1}{\alpha}}$ in the above display. 
Combining the above observation with \cref{thm:phimainres} and  \eqref{eq:therrterm} completes the proof in this case.

\emph{When $\alpha>2$}. We observe that 
$$\int_{\frac{a}{2^6\sqrt{n}}}^{\sigma} \sqrt{\mbH{u}}\,du\le \int_{\frac{a}{2^6\sqrt{n}}}^{\sigma} u^{-\frac{\alpha}{2}}\,du\lesssim \left(\frac{a}{\sqrt{n}}\right)^{1-\frac{\alpha}{2}}.$$
Therefore, by choosing $a\gtrsim n^{\frac{1}{2}-\frac{1}{\alpha}}$, we get:
$$\int_{\frac{a}{2^6\sqrt{n}}}^{\sigma} \sqrt{\mbH{u}}\,du\lesssim \left(\frac{a}{\sqrt{n}}\right)^{1-\frac{\alpha}{2}}\lesssim a \frac{n^{\left(\frac{1}{2}-\frac{1}{\alpha}\right)\left(-\frac{\alpha}{2}\right)}}{n^{\frac{1}{2}-\frac{\alpha}{4}}}\lesssim a.$$
Note that $n^{\frac{1}{2}-\frac{1}{\alpha}}\lesssim \sqrt{n}\sigma$ whenever $\sigma\gtrsim n^{-\frac{1}{\alpha}}$. 
By combining the above observation with \cref{thm:phimainres} and \eqref{eq:therrterm} completes the proof in this case.

\emph{Proof of Part (2)}. With $r>2$, if $\beta<\frac{r}{r-2}$, then note that for $q\ge 1$, by \eqref{eq:piv}, we get:
\begin{align*}
    \Lbp(q) \lesssim \int_0^1 u^{-2/r} \ dr + \sum_{i = 2}^q \int_0^{i^{-\beta}} u^{-2/r}  du & \lesssim \left(1 + \sum_{i = 1}^q i^{-\beta\left(1 - \frac{2}{r}\right)}\right) \lesssim  q^{1 - \beta\left(1 - \frac{2}{r}\right)} \,.
\end{align*}
Regarding the complexity of $\cF$, we have:  
\begin{align*}
    \sum_{k\ge 0:\ 2^{-k}\sigma \ge u} \mbH{2^{-k}\sigma} \gtrsim \sum_{k=0}^{\lfloor \log_2{(\sigma/u)}\rfloor} (2^{-k}\sigma)^{-\alpha}\gtrsim u^{-\alpha}.
\end{align*}
By \eqref{eq:tqd}, it is now easy to check that 
$$
\tq{u}\lesssim 1+\left(\frac{n}{1+u^{-\alpha}}\right)^{\frac{1}{1+\beta}} \lesssim 1+(n u^\alpha)^{\frac{1}{1 + \beta}},$$ and  $$\Lbp(\tq{u})\lesssim 1+\left(u^{\alpha} n\right)^{\frac{1}{1+\beta}\left(1-\frac{\beta(r-2)}{r}\right)}=:\Ld{u} \,.
$$
As $\beta <r/(r-2)$, the function $\Ld{\cdot}$ is  non-decreasing in $u$ as required in \cref{assm:bsbd}. We also have 
\begin{align*}
\Ld{u}(1+\mbH{u}) & \lesssim n^{\frac{1 - \beta\left(1 - \frac{2}{r}\right)}{1+ \beta}}u^{-\alpha\left(1-\frac{1 - \beta\left(1 - \frac{2}{r}\right)}{1+\beta}\right)} + u^{-\alpha} \\
& = n^{\frac{1 - \beta\left(1 - \frac{2}{r}\right)}{1+ \beta}} \left(u^{\alpha}\right)^{-\frac{2\beta}{1+\beta}\left(1-\frac{1}{r}\right)} + u^{-\alpha}:= \rno{(u)},
\end{align*}
which is non-increasing in $\delta$ as $r>2$. This verifies the other requirement for \cref{assm:bsbd}. We split the rest of the proof into three parts as follows.

\emph{When $\alpha<2$}. Observe that 
$$
\int_{\frac{a}{2^6\sqrt{n}}}^{\sigma} \sqrt{\rno(u)}\,du \lesssim n^{\frac{1-\beta\left(1-\frac{2}{r}\right)}{2(1+\beta)}} \sigma^{1-\frac{\alpha(r-1)\beta}{r(1+\beta)}}+\sigma^{1-\frac{\alpha}{2}}\lesssim \sqrt{n}\sigma \,.
$$
For the last inequality, we have used $\sigma\gtrsim n^{-\frac{1}{\alpha}}$. As \eqref{eq:therrterm} continues to hold, therefore the above display combined with \cref{thm:phimainres} completes the proof in this case.

\emph{When $2<\alpha<\frac{r(1+\beta)}{(r-1)\beta}$}. In this case, we note that 
$$
\int_{\frac{a}{2^6\sqrt{n}}}^{\sigma} \sqrt{\rno(u)}\,du \lesssim n^{\frac{1-\beta\left(1-\frac{2}{r}\right)}{2(1+\beta)}} \sigma^{1-\frac{\alpha(r-1)\beta}{r(1+\beta)}}+\left(\frac{a}{\sqrt{n}}\right)^{1-\frac{\alpha}{2}}\,.
$$
By choosing $a\gtrsim n^{\frac{1-\beta\left(1-\frac{2}{r}\right)}{2(1+\beta)}} \sigma^{1-\frac{\alpha(r-1)\beta}{r(1+\beta)}}$, we have 
\begin{align*}
\int_{\frac{a}{2^6\sqrt{n}}}^{\sigma} \sqrt{\rno(u)}\,du &\lesssim n^{\frac{1-\beta\left(1-\frac{2}{r}\right)}{2(1+\beta)}} \sigma^{1-\frac{\alpha(r-1)\beta}{r(1+\beta)}}+\left(n^{-\frac{\beta(r-1)}{r(1+\beta)}}\sigma^{1-\frac{\alpha(r-1)\beta}{r(1+\beta)}}\right)^{1-\frac{\alpha}{2}}\\ &\lesssim n^{\frac{1-\beta\left(1-\frac{2}{r}\right)}{2(1+\beta)}} \sigma^{1-\frac{\alpha(r-1)\beta}{r(1+\beta)}}\lesssim \sqrt{n}\sigma\,.
\end{align*}
For the last two inequalities, we have used $\sigma\gtrsim n^{-\frac{1}{\alpha}}$. As \eqref{eq:therrterm} continues to hold, therefore the above display combined with \cref{thm:phimainres} completes the proof in this case.

\emph{When $\alpha>\frac{r(1+\beta)}{(r-1)\beta}$}. In this case, we note that 
$$
\int_{\frac{a}{2^6\sqrt{n}}}^{\sigma} \sqrt{\rno(u)}\,du \lesssim n^{\frac{1-\beta\left(1-\frac{2}{r}\right)}{2(1+\beta)}} \left(\frac{a}{\sqrt{n}}\right)^{1-\frac{\alpha(r-1)\beta}{r(1+\beta)}} \,.
$$
Consequently, by choosing $a\gtrsim n^{\frac{1}{2}-\frac{1}{\alpha}}\lesssim \sqrt{n}\sigma$ (as $\sigma\gtrsim n^{-\frac{1}{\alpha}}$), we observe that 
$$\int_{\frac{a}{2^6\sqrt{n}}}^{\sigma} \sqrt{\rno(u)}\,du\lesssim n^{\frac{1-\beta\left(1-\frac{2}{r}\right)}{2(1+\beta)}} \left(\frac{a}{\sqrt{n}}\right)^{1-\frac{\alpha(r-1)\beta}{r(1+\beta)}}\lesssim a n^{\frac{1-\beta\left(1-\frac{2}{r}\right)}{2(1+\beta)}}\big(n^{-\frac{1}{\alpha}}\big)^{-\frac{\alpha \beta (r-1)}{r(1+\beta)}}n^{-\frac{1}{2}}=a.$$
Once again, \eqref{eq:therrterm} continues to hold. Therefore the above display combined with \cref{thm:phimainres} completes the proof in this case.
\end{proof}

\begin{proof}[Proof of \cref{cor:infmainres}]
Given a function $f\in\cF$, let $f_+=f\ \vee\ 0$ and similarly $f_-=(-f) \ \vee \ 0$. Note that $f=f_+\ -\ f_-$. Also define 
$$\cF_+:=\{f_+:\ f\in \cF\},\quad \cF_-=\{f_-:\ f\in \cF\}.$$
Note that 
$$\EE\sup_{f\in\cF}|\mathbb{G}_n(f)|\le \EE\sup_{f\in\cF_+} |\mathbb{G}_n(f)|\ + \ \EE\sup_{f\in\cF_-} |\mathbb{G}_n(f)|.$$
We will only bound the first term in the above inequality as the other can be bounded in exactly the same way. Recall that $\lVert f\rVert_{\infty}\le \sigma$ for all $f\in\cF$. Two key observations will be used in the rest of the proof, namely 
\begin{equation}\label{eq:obs11}
\mathscr{H}_{[ \  ]}(\sigma, \cF_+, \|\cdot \|_\infty)=0=:\mbHi{\sigma},
\end{equation}
and 
\begin{equation}\label{eq:obs12}
\mathscr{H}_{[ \  ]}(\delta, \cF_+, \|\cdot \|_\infty)\le \mathscr{H}_{[ \  ]}(\delta, \cF, \|\cdot \|_\infty)\le C_1\delta^{-\alpha}=:\mbHi{\delta}\quad \forall \ \delta<\sigma.
\end{equation}
To see \eqref{eq:obs11}, note that $\sup_{f\in \cf_+} \sup_{x\in\mcx} f(x)\le \sigma$ and clearly all $f\in\cF_+$ are non-negative functions. Therefore, the bracket formed by the constant functions $g^L(x)\equiv 0$ and $g^U(x)\equiv \sigma$ contains all of $\cF_+$, thereby establishing \eqref{eq:obs1}. 

On the other hand, \eqref{eq:obs12} follows directly by invoking \cref{lem:bracketing_property}, part 4.

Next, we note that, from \eqref{eq:tqd}, we have:
$$\tq{\sigma}\lesssim 1\vee \left(\frac{n}{1+\mbHi{\sigma}}\right)^{\frac{1}{1+\beta}}.$$
Consequently,
\begin{align}\label{eq:therrterm1}
\frac{\sigma \tq{\sigma}(1+\mbHi{\sigma})}{\sqrt{n}}&\lesssim \frac{b}{\sqrt{n}}(1+\mbHi{\sigma})+ \sigma n^{\frac{1}{2}-\frac{\beta}{1+\beta}}(1+\mbHi{\sigma})^{\frac{\beta}{1+\beta}}\lesssim \sigma n^{\frac{1}{2}-\frac{\beta}{1+\beta}}.
\end{align}
The last inequality above uses \eqref{eq:obs1}. 

\emph{Proof of part 1}. When $\beta>1$, note that 
$$\sum_{k=1}^q \beta_k\lesssim \sum_{k=1}^q (1+k)^{-\beta}\lesssim 1,$$
for $q\ge 1$, where all the implicit constants are free of $q$. The above display implies that $\Lambda_{n,\infty}^*(\cdot)$ can be chosen as a constant function and consequently \cref{assm:bsbdinf} holds. 
Therefore, by combining \eqref{eq:therrterm1} with \eqref{eq:obs12} and \cref{thm:infmainres}, we get:

\begin{equation}\label{eq:obs13}
    \EE\sup_{f\in\cF_+} |\mathbb{G}_n(f)|\le C \left( \sigma n^{\frac{1}{2}-\frac{\beta}{1+\beta}}+\inf\left\{0\le a\le 8\sqrt{n}\sigma:\ a\ge C_0\int_{\frac{a}{2^6\sqrt{n}}}^{\sigma} \sqrt{\mbHi{u}}\,du\right\}\right).
\end{equation}
We split the next part of the proof into two further cases.

\emph{When $\alpha<2$}. We observe that 
$$\int_{\frac{a}{2^6\sqrt{n}}}^{\sigma} \sqrt{\mbHi{u}}\,du\le \int_{\frac{a}{2^6\sqrt{n}}}^{\sigma} u^{-\frac{\alpha}{2}}\,du\lesssim \sigma^{1-\frac{\alpha}{2}}\lesssim \sqrt{n}\sigma.$$
Here we have used $\sigma\gtrsim n^{-\frac{1}{\alpha}}$ in the above display. 
Combining the above observation with \eqref{eq:obs13} completes the proof in this case.

\emph{When $\alpha>2$}. We observe that 
$$\int_{\frac{a}{2^6\sqrt{n}}}^{\sigma} \sqrt{\mbHi{u}}\,du\le \int_{\frac{a}{2^6\sqrt{n}}}^{\sigma} u^{-\frac{\alpha}{2}}\,du\lesssim \left(\frac{a}{\sqrt{n}}\right)^{1-\frac{\alpha}{2}}.$$
Therefore, by choosing $a\gtrsim n^{\frac{1}{2}-\frac{1}{\alpha}}$, we get:
$$\int_{\frac{a}{2^6\sqrt{n}}}^{\sigma} \sqrt{\mbHi{u}}\,du\lesssim \left(\frac{a}{\sqrt{n}}\right)^{1-\frac{\alpha}{2}}\lesssim a \frac{n^{\left(\frac{1}{2}-\frac{1}{\alpha}\right)\left(-\frac{\alpha}{2}\right)}}{n^{\frac{1}{2}-\frac{\alpha}{4}}}\lesssim a.$$
Note that $n^{\frac{1}{2}-\frac{1}{\alpha}}\lesssim \sqrt{n}\sigma$ whenever $\sigma\gtrsim n^{-\frac{1}{\alpha}}$. 
By combining the above observation with \eqref{eq:obs13} completes the proof in this case.

\emph{Proof of Part (2)}. For $\beta<1$, we get
\begin{align*}
    \sum_{k=1}^q \beta_k \lesssim \sum_{k=1}^q (1+k)^{-\beta} \lesssim  q^{1 - \beta} \,.
\end{align*}
for $q\ge 1$, where all the implicit constants are free of $q$. By \eqref{eq:obs12}, we can work with $\mbHi{u}=C u^{-\alpha}$ for some constant $C>0$. 
Regarding the complexity of $\cF$, we have:  
\begin{align*}
    \sum_{k\ge 0:\ 2^{-k}\sigma \ge u} \mbHi{2^{-k}\sigma} \gtrsim \sum_{k=0}^{\lfloor \log_2{(\sigma/u)}\rfloor} (2^{-k}\sigma)^{-\alpha}\gtrsim u^{-\alpha}.
\end{align*}
By \eqref{eq:tqd}, it is now easy to check that 
$$
\tq{u}\lesssim 1+\left(\frac{n}{1+u^{-\alpha}}\right)^{\frac{1}{1+\beta}} \lesssim 1+(n u^\alpha)^{\frac{1}{1 + \beta}},$$ and  $$\sum_{k=1}^{\tq{u}}\beta_k\lesssim 1+\left(u^{\alpha} n\right)^{\frac{1-\beta}{1+\beta}}=:\Lambda_{n,\infty}^*(u) \,.
$$
As $\beta < 1$, the function $\Lambda_{n,\infty}^*(\cdot)$ is  non-decreasing in $u$ as required in \cref{assm:bsbdinf}. We also have 
\begin{align*}
\Lambda_{,\infty}^*(u)(1+\mbHi{u}) & \lesssim n^{\frac{1 - \beta}{1+ \beta}}\delta^{-\alpha\left(1-\frac{1 - \beta}{1+\beta}\right)} + u^{-\alpha} \\
& = n^{\frac{1 - \beta}{1+ \beta}} \left(\delta^{\alpha}\right)^{-\frac{2\beta}{1+\beta}} + u^{-\alpha}:= \rto{(u)},
\end{align*}
which is non-increasing in $u$. This verifies the other requirement for \cref{assm:bsbdinf}. We split the rest of the proof into three parts as follows.

\emph{When $\alpha<2$}. Observe that 
$$
\int_{\frac{a}{2^6\sqrt{n}}}^{\sigma} \sqrt{\rto(u)}\,du \lesssim n^{\frac{1-\beta}{2(1+\beta)}} \sigma^{1-\frac{\alpha\beta}{1+\beta}}+\sigma^{1-\frac{\alpha}{2}}\lesssim \sqrt{n}\sigma \,.
$$
For the last inequality, we have used $\sigma\gtrsim n^{-\frac{1}{\alpha}}$. As \eqref{eq:therrterm1} continues to hold, therefore the above display combined with \cref{thm:infmainres} completes the proof in this case.

\emph{When $2<\alpha<\frac{1+\beta}{\beta}$}. In this case, we note that 
$$
\int_{\frac{a}{2^6\sqrt{n}}}^{\sigma} \sqrt{\rto(u)}\,du \lesssim n^{\frac{1-\beta}{2(1+\beta)}} \sigma^{1-\frac{\alpha\beta}{1+\beta}}+\left(\frac{a}{\sqrt{n}}\right)^{1-\frac{\alpha}{2}}\,.
$$
By choosing $a\gtrsim n^{\frac{1-\beta}{2(1+\beta)}} \sigma^{1-\frac{\alpha\beta}{1+\beta}}$, we have 
\begin{align*}
\int_{\frac{a}{2^6\sqrt{n}}}^{\sigma} \sqrt{\rto(u)}\,du &\lesssim n^{\frac{1-\beta}{2(1+\beta)}} \sigma^{1-\frac{\alpha\beta}{1+\beta}}+\left(n^{-\frac{\beta}{1+\beta}}\sigma^{1-\frac{\alpha\beta}{1+\beta}}\right)^{1-\frac{\alpha}{2}}\\ &\lesssim n^{\frac{1-\beta}{2(1+\beta)}} \sigma^{1-\frac{\alpha\beta}{1+\beta}}\lesssim \sqrt{n}\sigma\,.
\end{align*}
For the last two inequalities, we have used $\sigma\gtrsim n^{-\frac{1}{\alpha}}$. As \eqref{eq:therrterm1} continues to hold, therefore the above display combined with \cref{thm:infmainres} completes the proof in this case.

\emph{When $\alpha>\frac{1+\beta}{\beta}$}. In this case, we note that 
$$
\int_{\frac{a}{2^6\sqrt{n}}}^{\sigma} \sqrt{\rto(u)}\,du \lesssim n^{\frac{1-\beta\left(1-\frac{2}{r}\right)}{2(1+\beta)}} \left(\frac{a}{\sqrt{n}}\right)^{1-\frac{\alpha(r-1)\beta}{r(1+\beta)}} \,.
$$
Consequently, by choosing $a\gtrsim n^{\frac{1}{2}-\frac{1}{\alpha}}\lesssim \sqrt{n}\sigma$ (as $\sigma\gtrsim n^{-\frac{1}{\alpha}}$), we observe that 
$$\int_{\frac{a}{2^6\sqrt{n}}}^{\sigma} \sqrt{\rto(u)}\,du\lesssim n^{\frac{1-\beta}{2(1+\beta)}} \left(\frac{a}{\sqrt{n}}\right)^{1-\frac{\alpha\beta}{1+\beta}}\lesssim a n^{\frac{1-\beta}{2(1+\beta)}}\big(n^{-\frac{1}{\alpha}}\big)^{-\frac{\alpha \beta}{1+\beta}}n^{-\frac{1}{2}}=a.$$
Once again, \eqref{eq:therrterm1} continues to hold. Therefore the above display combined with \cref{thm:infmainres} completes the proof in this case.

\end{proof}

\begin{proof}[Proof of \cref{cor:nonpbd}]
We invoke \cref{thm:gamma_mixing} with $r=2$, $V=0$, $K_r=K\sim 1$, $D=\theta=1$, $\sigma\sim 1$. We will show that each of the four terms in the definition of $\Pi_n(\cF)$ in \cref{thm:gamma_mixing}, part (3), is bounded (up to constants) by $n^{\frac12-\frac1\alpha}$. To wit, observe that the first and the fourth term add up to 
$$\lesssim n^{\frac{\gamma(\alpha-2)}{2\alpha\gamma}}=n^{\frac12-\frac1\alpha}.$$
The second and third terms add up to 
$$\lesssim n^{\frac{1}{2}-\frac{\gamma}{1+\gamma}}+n^{\frac{1}{2}-\frac{\gamma}{2(\gamma+1)}}\lesssim n^{\frac{1}{2}-\frac{\gamma}{2(\gamma+1)}}\lesssim n^{\frac{1}{2}-\frac{1}{\alpha}}.$$
Here the last inequality uses the fact that $\alpha\ge 2(1+\gamma)/\gamma$. This completes the proof.
\end{proof}

\begin{proof}[Proof of \cref{cor:vcbd}]
The conclusion follows by invoking \cref{thm:gamma_mixing} with $r=2$, $K_r=1$, $D$, $\alpha=0$, and $V,\sigma$ as is.
\end{proof}

\begin{proof}[Proof of \cref{cor:gen_error}]
Note that by the definition of $\hat{f}_n$, we have $\bbP_n(\ell(Y,\hat{f}_n(X))-\ell(Y,f_\star(X)))\le 0$. Therefore, 
\begin{align*}
\cE(\hat f_n) & = \bbP_n\left(\ell(Y, \hat f_n(X)) - \ell(Y, f^\star(X))\right) + (P - \bbP_n)\left(\ell(Y, \hat f_n(X)) - \ell(Y, f^\star(X))\right) \\
    & \le (P - \bbP_n)\left(\ell(Y, \hat f_n(X)) - \ell(Y, f^\star(X))\right)  \\
    & \le \sup_{f \in \cF}\left|(\bbP_n - P)\left(\ell(Y, \hat f_n(X)) - \ell(Y, f^\star(X))\right) \right|\\
    & =\sup_{g \in \cG}\left|(\bbP_n - P)g(X, Y) \right|
\end{align*}
where $\cG$ is defined as in \eqref{eq:loss_f_def},  
and $P$ is the expectation with respect to the distribution of the new sample $(X,Y)$. Note that all the terms above are functions of $(X_1,Y_1),\ldots , (X_n,Y_n)$. Therefore, to bound the \emph{expected generalization error} we need to bound the expected supremum of the empirical process. An application of \cref{cor:infmainres} with $\sigma\sim 1$ completes the proof.

\end{proof}

\begin{proof}[Proof of \cref{cor:fast_error}]
By invoking \cref{cor:vcbd}, we have 
$$\Pi_n(\mathcal{G}_{\delta})=\delta \big(n D^{\gamma} (\log{(B/\sigma)})^V\big)^{\frac{1}{2(\gamma+1)}}+n^{\frac12-\frac{\gamma}{\gamma+1}}(D(\log{(B/\sigma)})^V)^{\frac{\gamma}{\gamma+1}}.$$
Note that $\Pi_n(\delta)/\delta^t$ is decreasing in $\delta$ for any $t\in (1,2)$. next by choosing 
$$\delta_n=\left(\frac{D}{n}\right)^{\frac{\gamma}{2(\gamma+1)}}(\log{(Bn)})^{V\gamma}{2(\gamma+1)},$$
it is east to check that $\Pi_n(\mathcal{G}_{\delta_n})\le \sqrt{n}\delta_n^2$. This completes the proof.
\end{proof}

\begin{proof}[Proof of \cref{cor:adaptbd}]
As $\alpha>2(1+\gamma^{-1})$, under \eqref{eq:locradent}, by invoking the definition of $\Pi_n(\cdot)$ from \cref{thm:gamma_mixing}, part (3), we get:
\begin{align*}
    \Pi_n(\cG_{\delta})&= n^{\frac12-\frac1\alpha}\delta\left[\left(\log{\left(\frac{B}{\delta}\right)}\right)^{\frac{V\gamma}{2(\gamma+1)}}+\left(\log{\left(\frac{B}{\delta}\right)}\right)^{\frac{V}{2}}\right]+n^{\frac12-\frac{\gamma}{1+\gamma}}\left(\log{\left(\frac{B}{\delta}\right)}\right)^{\frac{V\gamma}{\gamma+1}}\\ &+n^{\frac{1}{2(1+\gamma)}}\delta \left(\log{\left(\frac{B}{\delta}\right)}\right)^{\frac{V\gamma}{2(\gamma+1)}}.
\end{align*}
Once again, we observe that $\Pi_n(\cG_{\delta})/\delta^t$ is decreasing for all $t\in (1,2)$. Therefore, by choosing 
$$\delta_n=n^{-\frac{1}{\alpha}}\left(\log{(Bn)}\right)^{\frac{V}{2}},$$
we observe that $\Pi_n(\cG_{\delta_n})\le \sqrt{n}\delta_n^2$. By invoking \cref{thm:roc}, we complete the proof.
\end{proof}

\begin{proof}[Proof of \cref{cor:fasterl1}]
The proof of this proceeds using a technique from \cite[Theorem 3.6]{han2021set}. Observe that by writing expectations of non-negative functions as tail integrals, given any $f\in\cF$, we have:
\begin{align*}
    |(\mathbb{P}_n-P)f|&\le |(\mathbb{P}_n-P)f_+|+|(\mathbb{P}_n-P)f_-|\\ &\le \bigg|\int_0^{\infty}\left(\bbP_{\bbP_n}(f_+(X)>t)-\bbP_{P}(f_+(X)>t)\right)\,dt\bigg|\\ &\qquad \qquad +\bigg|\int_0^{\infty}\left(\bbP_{\bbP_n}(f_-(X)>t)-\bbP_{P}(f_-(X)>t)\right)\,dt\bigg|\\ &\le 2b \sup_{C\in\mathcal{C}} |(\bbP_n-P)(C)|,
\end{align*}
where we have used the fact that all the level sets of $f_+$ and $f_-$ are contained in $\mathcal{C}$. Note that 
$$\sup_{C\in\mathcal{C}} |(\bbP_n-P)(C)|=\sup_{f\in \{\mathbf{1}(C):\ C\in\mathcal{C}\}} |(\bbP_n-P)(f)|.$$
The condition \eqref{eq:setent} gives the $L_1$ entropy of the class of functions $\{\mathbf{1}(C):\ C\in\mathcal{C}\}$. Invoking \cref{thm:gamma_mixing} with $r=1$ completes the proof.
\end{proof}

\section{Proofs from Applications}\label{sec:apprespf}

This section presents the proofs of all our results from \cref{sec:app} sequentially.

\subsection{Proof of Theorem \ref{thm:dnn_mixing}}\label{sec:nnproof}

We begin with a brief discussion on neural net architecture for the sake of completion. For ease of exposition, we will stick to the ReLU activation function and a feedforward fully connected neural network (also called multi-layer perceptron). Such neural networks have attracted a lot of attention due to their empirical success. We also recall the notation $\cF_{\mathrm{NN}}(d_{\mathrm{in}}, d_{\mathrm{out}}, N, L, B)$ where $d_{\mathrm{in}}, d_{\mathrm{out}}$ are the input and the output dimensions, $N$ and $L$ denote the width and the depth, while $B$ is the bound on the weights. Functions belonging to a  ReLU based neural network take the following form: 
$$f(x)=\mathcal{L}_{L}\ \circ\ \sigma\ \circ\ \mathcal{L}_{L-1}\ \circ\ \sigma\ \circ\ \cdots \mathcal{L}_1\ \circ\ \sigma\ \circ\ \mathcal{L}_0(x),$$
where $\mathcal{L}_{\ell}(x):=W_{\ell}x+b_{\ell}$ for $W_{\ell}\in \R^{d_{\ell}\times d_{\ell-1}}$, $b_{\ell}\in d_{\ell}$, $d_0=d$, $d_{L}=1$ and $\sigma$ is applied coordinatewise. Here $L$ denotes the number of layers and $(d_0,d_1,\ldots ,d_{L})$ is typically called a width vector with the width $W=\max_{\ell=0}^L d_{\ell}$. The bound on the weights implies $\lVert W_{\ell}\rVert_{\infty}\le b$ and $\lVert b_{\ell}\rVert_{\infty}\le b$ for $\ell=0,1,\ldots ,L$. Following \cite[Definition 2]{fan2023factor}, the truncated class $\tilde{\cF}_{\mathrm{NN}}$ is defined as follows:
$$\tilde{\cF}_{\mathrm{NN}}(d_{\mathrm{in}}, d_{\mathrm{out}}, N, L, B):=\{\tilde{f}(x)=\mathrm{sign}(f(x))(|f(x)|\wedge (2b)):\, f\in \cF_{\mathrm{NN}}(d_{\mathrm{in}}, d_{\mathrm{out}}, N, L, B)\}.$$
Note that as we are working with a fully connected network, we have suppressed $W$, the number of active weights, throughout the notation used.

As in general sieve-based estimation procedure for nonparametric regression function, the proof of Theorem \ref{thm:dnn_mixing} consists of two key ingredients: bounding bias/approximation error and bounding variance/statistical error. We will need a preparatory result each for the above two components. 

To quantify the bias/approximation error of a Hölder smooth function by neural nets, we use the following result. It is a simple modification of \cite[Theorem 4]{fan2023factor}.

\begin{prop}
  There exists universal constants $c_1$, $c_2$, $c_3$, $N_0$ depending only on $d$, $s$, $C$, $b$ such that given any $f:\R^d\to\R$, $f\in\Sigma(s,C)$ with $\lVert f\rVert_{\infty}\le b$ and any arbitrary $N\ge N_0$, it holds that 
  \begin{equation}\label{eq:dnn_approx}
\inf_{\phi \in \tilde{\cF}_{\mathrm{NN}}(d, 1, N, c_1, N^{c_2})}\|f - \phi\|_\infty \le c_3 N^{-\frac{2s}{d}} \,,
\end{equation}
\end{prop}

Concerning the sampling variance, we need a bound on the complexity of the collection of neural networks. Towards that end, a bound on the $L_\infty$ bracketing number of neural networks can also be found in the literature; which we state here in the following proposition, which is taken from \cite[Lemma 5]{schmidt2020nonparametric}. 

\begin{prop}
    With $\tilde{\cF}_{\mathrm{NN}}(d, 1, N, c_1, N^{c_2})$ defined above, there exists some constants $c_4$, $c_5$, $c_6$ depending only on $c_1$, $c_2$ such that 
    \begin{equation}
\label{eq:bracketing_number_DNN}
\benf{\delta} \le c_4 N^2 \log{\left(\frac{2c_5 N^{c_6}}{\delta}\right)} \,.
\end{equation}
\end{prop}
Note that \cite[Lemma 5]{schmidt2020nonparametric} provides bounds on the covering number, instead of the bracketing number. However, with respect to the uniform norm, the bracketing number and covering number are equivalent. Hence the conclusion follows.

\begin{proof}[Completing the proof of \cref{thm:dnn_mixing}] 
Henceforth, denote by $\cF_n := \tilde{\cF}_{\mathrm{NN}}(d, 1, N, c_1, N^{c_2})$, where we will later choose $N$ depending on $n$. Given $\cF_n$, define the \emph{best approximator} $f_n$ as: 
$$
f_n = \argmin_{f \in \cF_n} \bbE[(f_*(X) - f_n(X))^2] \,.
$$
By \eqref{eq:dnn_approx}, we have $\|f_n - f_*\|_2^2 \le c_3^2 N^{-\frac{4s}{d}}$ \,. Recall that $\hat f_n$ is the least-square estimator over $\cF_n$, i.e.,
$$
\hat f_n = \argmin_{f \in \cF_n} \frac1n \sum_{i = 1}^n (Y_i - f(X_i))^2 \,.
$$
We will prove the result based on the following proposition, which is a slightly adapted version of \cref{thm:roc}. 

\begin{prop}\label{prop:modifyroc}
    Suppose the functions $\phi_{n,1}(\cdot)$ and $\phi_{n,2}(\cdot)$ satisfy the following inequalities: 
    \begin{equation*}
\label{eq:bound_1}
    \phi_{n, 1}(\delta) \ge \sqrt{n}\bbE\left[\sup_{\|f - f_n\|_2 \le \delta} \left|(\bbP_n - \bbP)(f - f_n)^2\right|\right] 
\end{equation*}
\begin{equation*}
    \label{eq:bound_2}
    \phi_{n, 2}(\delta) \ge \sqrt{n}\bbE\left[\sup_{\|f - f_n\|_2 \le \delta} \left|(\bbP_n - \bbP)\xi (f - f_n)\right|\right] 
\end{equation*}
Also assume there exists $\alpha\in (0,2)$ such that both $\phi_{n,1}(\delta)/\delta^{\alpha}$ and $\phi_{n,2}(\delta)/\delta^{\alpha}$ are decreasing functions. Finally let $r_n$ be a sequence satisfying 
\begin{align}
\label{eq:roc_dnn_1}
    & r_n^2 \left(\phi_{n, 1}(r_n^{-1}) + \phi_{n, 2}(r_n^{-1})\right) \le \sqrt{n} \,, \\
\label{eq:roc_dnn_2}
    & r_n \|f_n - f_0\|_2 = r_n \inf_{f \in \cF_n}\|f - f_0\|_2 \le \frac{1}{\sqrt{2}} \,.
\end{align}
Then there exists a constant $C>0$, depending on $\alpha$, such that for any $t>1$, the following holds: 
$$\mathbb{P}(r_n\lVert \hat f_n-f_\star\rVert_2\ge t)\le C t^{2-\alpha}.$$
\end{prop}

The proof of this proposition requires a slight modification of the peeling argument used in the proof of \cref{thm:roc}. We will first complete the proof of \cref{thm:dnn_mixing} assuming the above claim. We will then prove the proposition itself.

To obtain bounds on $\phi_{n, 1}, \phi_{n, 2}$ we use Theorem \ref{thm:gamma_mixing}. For $\phi_{n, 1}$ we consider the function class $\cF_n^2 = \{(f - f_n)^2: f \in \cF_n\}$. From~\cref{lem:bracketing_property}, part 1, we have: 
$$
\mathscr{H}_{[\,]}(\delta, \cF_n^2, \lVert\cdot\rVert_{\infty}) \le 2 \mathscr{H}_{[\,]}(\delta/8b, \cF_n, \lVert\cdot\rVert_{\infty}) \le c_4 N^2 \log{\left(\frac{16bc_5 N^{c_6}}{\delta}\right)} \,.
$$
Applying part 1. of Theorem \ref{thm:gamma_mixing} (with $\alpha = 0$ and $\tilde r = 2$) and \eqref{eq:bracketing_number_DNN}, we obtain (ignoring the constants): 
\begin{equation}
    \label{eq:moc_1_dnn}
    \phi_{n, 1}(\delta) \lesssim \ \delta n^{\frac{1}{2(1 + \gamma)}} \left(N^2\log{\left(\frac{N}{\delta}\right)}\right)^{\frac{\gamma}{2(\gamma + 1)}} + n^{\frac12 - \frac{\gamma}{1 + \gamma}}\left(N^2\log{\left(\frac{N}{\delta}\right)}\right)^{\frac{\gamma}{1 + \gamma}} \,.
\end{equation}
The analysis for $\phi_{n, 2}$ is similar; as $f \in \cF_n$ is bounded by $2b$ and $\xi$ is bounded by $b_e$, the functions $\xi(f - f_n)$ is bounded by $4bb_e$. Call this collection to be $\cF_{n, \xi}$. By Lemma \ref{lem:bracketing_property} we have: 
\begin{equation}
\label{eq:bracketing_error}
\mathscr{H}_{[\,]}(\delta, \cF_{n,\xi}, \lVert\cdot\rVert_{\infty}) \le \mathscr{H}_{[\,]}(\delta/b_e, \cF_n, \lVert\cdot\rVert_{\infty}) \,.
\end{equation}
Therefore, another application of part 1. of Theorem \ref{thm:gamma_mixing} (with $\alpha = 0, \tilde r = 2$), coupled with \eqref{eq:bracketing_number_DNN}, we get: 
\begin{equation}
    \label{eq:moc_2_dnn}
    \phi_{n, 2}(\delta) \lesssim \ \delta n^{\frac{1}{2(1 + \gamma)}} \left(N^2\log{\left(\frac{N}{\delta}\right)}\right)^{\frac{\gamma}{2(\gamma + 1)}} + n^{\frac12 - \frac{\gamma}{1 + \gamma}}\left(N^2\log{\left(\frac{N}{\delta}\right)}\right)^{\frac{\gamma}{1 + \gamma}} \,.
\end{equation}
To satisfy \eqref{eq:roc_dnn_1} and \eqref{eq:roc_dnn_2}, some  simple algebra yields that the rate $r_n$ satisfies: 
\begin{align}
\label{eq:roc_dnn_new_1}
    r_n & \lesssim N^{\frac{2s}{d}} \\
\label{eq:roc_dnn_new_2}
    r_n \left(\log{(Nr_n)}\right)^{\frac{\gamma} {2(\gamma + 1)}} & \lesssim \left(\frac{n}{N^2}\right)^{\frac{\gamma}{2(\gamma + 1)}} \,.
\end{align}
Taking the width $N$ as: 
$$
N \sim \left(\frac{n}{\log{n}}\right)^{\frac{1}{2} \frac{1}{1 + \frac{2s}{d}\left(\frac{1 + \gamma}{\gamma}\right)}}
$$
we claim that 
$$
r_n \sim \left(\frac{n}{\log{n}}\right)^{\frac{s}{d + 2s\left(\frac{1 + \gamma}{\gamma}\right)}}
$$
satisfies \eqref{eq:roc_dnn_new_1} and \eqref{eq:roc_dnn_new_2}. The first one, i.e. \eqref{eq:roc_dnn_new_1} is immediate from the choice of $N$. For \eqref{eq:roc_dnn_new_2}, observe that, our choice of $N$ yields: 
$$
\left(\frac{n}{N^2}\right)^{\frac{\gamma}{2(\gamma + 1)}}  \sim n^{\frac{s}{d + 2s\left(\frac{1+\gamma}{\gamma}\right)}}\left(\log{n}\right)^{\frac{\frac{\gamma}{2(\gamma + 1)}}{1 + \frac{2s}{d}\left(\frac{1+\gamma}{\gamma}\right)}} \,.
$$
Now, we have: 
\begin{align*}
     r_n \left(\log{(Nr_n)}\right)^{\frac{\gamma} {2(\gamma + 1)}} & \lesssim \left(\frac{n}{\log{n}}\right)^{\frac{s}{d + 2s\left(\frac{1 + \gamma}{\gamma}\right)}} \left(\log{\left(\frac{n}{\log{n}}\right)}\right)^{\frac{\gamma}{2(1 + \gamma)}} \\
     & \lesssim n^{\frac{s}{d + 2s\left(\frac{1 + \gamma}{\gamma}\right)}}(\log{n})^{\frac{\gamma}{2(1 + \gamma)} - \frac{s}{d + 2s\left(\frac{1 + \gamma}{\gamma}\right)}} \\
     & = n^{\frac{s}{d + 2s\left(\frac{1 + \gamma}{\gamma}\right)}}(\log{n})^{\frac{\frac{\gamma}{2(\gamma + 1)}}{1 + \frac{2s}{d}\left(\frac{1+\gamma}{\gamma}\right)}} \sim \left(\frac{n}{N^2}\right)^{\frac{\gamma}{2(\gamma + 1)}} \,.
\end{align*}
This completes the proof. 
\end{proof}

\begin{proof}[Proof of \cref{prop:modifyroc}]
Finally, we present a peeling argument to prove \cref{prop:modifyroc}. For notational simplicity, define $\cR_n(f) = (1/n)\sum_i (Y_i - f(X_i))^2$ to be the empirical risk and $\cR(f) = \bbE[(Y - f(X))^2]$ to be the population risk. Define $A_j := \{f: 2^{j}t r_n^{-1} \le \lVert f-f_\star\rVert_2 \le 2^{j+1}t r_n^{-1}: f \in \cF_n\}$. Then, for all $t>1$, we have: 
\allowdisplaybreaks
\begin{align}
&\;\;\;\bbP\left(r_n \|f_* - \hat f\|_2 \ge t\right) \notag \\ & = \bbP\left(\sup_{\substack{f \in \cF_n \\ r_n \|f_* - \hat f\|_2 \ge t}} \cR_n(f_n) - \cR_n(f) \ge 0 \right) \notag \\
& \le \sum_{j=1}^\infty \bbP\left(\sup_{\phi \in A_j} \cR_n(f_n) - \cR_n(f) \ge 0 \right) \notag \\
 & \le \sum_{j=1}^\infty \bbP\left(\sup_{\phi \in A_j} (\cR_n(f_n) - \cR(f_n)) - (\cR_n(f)  - \cR(f))  \ge \inf_{\phi \in A_j} \cR(f) -\cR(f_n) \right) \notag \\
\label{eq:rate_1} & \le \sum_{j=1}^\infty \frac{\bbE\left[\sup_{\phi \in A_j} \left|(\cR_n(f_n) - \cR(f_n)) - (\cR_n(f)  - \cR(f)) \right|\right]}{\inf_{\phi \in A_j}\cR(f) -\cR(f_n)}
\end{align}
We bound the numerator and denominator separately. For the denominator note that: 
\begin{align*}
    \inf_{f \in A_j} \cR(f) -\cR(f_n)  & \ge \inf_{f: \|f - f_*\|_2 \ge 2^{j-1}tr_n^{-1}}\cR(f) -\cR(f_*) + \cR(f_*) -\cR(f_n) \\
    & \ge  2^{2j-2}t^2r_n^{-2} - \inf_{f \in \cF_n} \|f- f_*\|^2 \triangleq  2^{2j-2}t^2r_n^{-2}  - \omega^2_n \,.
\end{align*}
For the numerator: 
\begin{align*}
    &\bbE\left[\sup_{f \in A_j} \left|(\cR_n(f_n) - \cR(f_n)) - (\cR_n(f)  - \cR(f)) \right|\right] \\
    &\bbE\left[\sup_{f \in A_j} \left|(\bbP_n - P)(f - f_n)^2 + (\bbP_n - P)\eps (f - f_n)\right|\right] \\
    & \le \frac{\phi_{n, 1}\left(2^j tr_n^{-1}\right) + \phi_{n, 2}\left(2^j tr_n^{-1}\right)}{\sqrt{n}} \,.
\end{align*}
Putting this bound in equation \eqref{eq:rate_1} we have: 
\begin{align*}
    & \bbP\left(r_n \|\hat f - f_\star\|_2 \ge t\right) \\
    & \le \sum_{j=1}^\infty \frac{\phi_{n, 1}\left(2^j tr_n^{-1}\right) + \phi_{n, 2}\left(2^j tr_n^{-1}\right)}{\sqrt{n}\left(2^{2j-2}t^2r_n^{-2} - \omega_n^2\right) } \\
    & \le \sum_{j=1}^\infty \frac{r_n^2 \phi_{n, 1}\left(2^j tr_n^{-1}\right) + r_n^2 \phi_{n, 2}\left(2^j tr_n^{-1}\right)}{\sqrt{n}\left(2^{2j-2}t^2 - r_n^2 \omega_n^2\right) }  \\
     & \le \sum_{j=1}^\infty \frac{(2^j t)^\alpha \left(r_n^2 \phi_{n, 1}\left(r_n^{-1}\right) + r_n^2 \phi_{n, 2}\left(r_n^{-1}\right)\right)}{\sqrt{n}\left(2^{2j-2}t^2 - r_n^2 \omega_n^2\right) } \hspace{0.1in}[\because \phi_n(t)/t^\alpha \text{ is decreasing and } 2^j t \ge 1] \\
     & \le  2\sum_{j=1}^\infty \frac{(2^j t)^\alpha}{\left(2^{2j-2}t^2 - r_n^2 \omega_n^2\right)} \hspace{0.1in} [\because r_n^2\phi_{n, 1}(r_n^{-1}) \vee r_n^2 \phi_{n, 2}(r_n^{-1}) \le \sqrt{n}, \text{ equation }\eqref{eq:roc_dnn_1}] \\
     & \le 2\sum_{j=1}^\infty \frac{(2^j t)^\alpha}{\left(2^{2j-3}t^2\right)} \hspace{0.1in} [\because r_n^2w_n^2 \le \frac12 \le 2^{2j - 3}t^2, \text{ equation }\eqref{eq:roc_dnn_2}] \\
     & \le 16 t^{2 - \alpha} \sum_{j = 1}^\infty 2^{-j(2 - \alpha)} \le Ct^{2 - \alpha} \,.
\end{align*}
This completes the proof. 
\end{proof}

\subsection{Proof of Theorem \ref{thm:additive_rate}}
We prove Theorem \ref{thm:additive_rate} in this subsection. Using a DNN-based estimator, our proof is essentially similar to that of Theorem \ref{thm:dnn_mixing}. For brevity, we highlight the key differences here. Recall the definition of the estimator in \eqref{eq:additive_model} and the collection $\cF_{\mathrm{DNN}}$ of DNN. First, we bound the approximation error based on the analysis of \cite{bhattacharya2023deep}. We prove the following lemma: 
\begin{lemma}
    \label{lem:additive_dnn_approx}
    There exists universal constants $c_1$, $c_2$, $c_3$, $N_0$ depending only on $s$, $C$, $b$ such that given any $f:\R^d\to\R$, $f\in\cF_\add([0, 1]^d, d, b, s)$ with $\lVert f\rVert_{\infty}\le b$, and any arbitrary $N\ge N_0$, it holds that 
  \begin{equation}\label{eq:dnn_approx_add}
\inf_{\phi \in \tilde{\cF}_{\mathrm{NN},\mathrm{add}}(d, 1, N, c_1, N^{c_2})}\|f - \phi\|_2 \le c_3 \sqrt{d}N^{-2s} \,,
\end{equation}
    \end{lemma}
\begin{proof}
    To prove the above lemma we use \eqref{eq:dnn_approx_add} for each of the univariate components. In particular let $\tilde \phi_j\in $ with depth $c_1$, width $N$, and weights bounded by $N^{c_2}$ such that: 
    $$
    \|\tilde \phi_j - f_{j}\|_\infty \le CN^{-2s} \,.
    $$
    Note that changing $\tilde{\phi}_j$ by $\tilde{\phi}_j-\int_{[0,1]}\tilde{\phi}_j(z)\,dz$ only amounts to altering the bias of the last layer of the neural network, without changing its architecture. Further, as $\int f(z)\,dz=0$, we have:
    $$\bigg\lVert \tilde{\phi}_j-\int_{[0,1]}\tilde{\phi}_j(z)\,dz-f_{j}\bigg\rVert_{\infty}\le 2 \lVert \tilde{\phi}_j-f_j\rVert_{\infty}\le c_3 N^{-2s},$$
    by changing $c_3$ appropriately. 
    Therefore, we can assume without loss of generality that $\int_{[0,1]} \tilde{\phi}_j(z)\,dz=0$.
    Now consider $\tilde{\phi}(x) = \sum_{j=1}^d   \tilde{\phi}_j(x_j)$ for $x=(x_1,\ldots ,x_d)$. Then $\phi\in \cF_{\mathrm{NN},\mathrm{add}}(d,1,N,c_1,N^{c_2})$.  
    Furthermore, we have: 
    \begin{align*}
        \bbE[(f(X) - \phi(X))^2] &\lesssim \int (f(x)-\tilde{\phi}(x))^2\,dx=\sum_{j=1}^d \int (\tilde{\phi}_j(x_j)-f(x_j))^2\,dx\,\le c_3 d N^{-4s}, 
    \end{align*}
    by changing $c_3$ as necessary. The conclusion now follows by changing $\phi(x)$ to $\mathrm{sign}(\phi(x))(|\phi(x)|\wedge (2b))$ and noting that $\lVert f\rVert_{\infty}\le b$, and
    $$\int \big(\mathrm{sign}(\phi(x))(|\phi(x)|\wedge (2b))-f(x)\big)^2\,dx\le \int (f(x)-\tilde{\phi}(x))^2\,dx.$$
\end{proof}
To bound the variance/stochastic error, we follow a similar approach as that of the proof of Theorem \ref{thm:dnn_mixing}. In particular, thanks to \cref{prop:modifyroc}, we aim to find $\phi_{n, 1}(\delta)$ and $\phi_{n, 2}(\delta)$ following \eqref{eq:bound_1} and \eqref{eq:bound_2}. 
Recall that we use the collection of neural networks $\cF_{\mathrm{NN},\mathrm{add}}(d,1,N,c_1,N^{c_2})$. 
Therefore, using \eqref{eq:bracketing_number_DNN} we obtain: 
$$
\mathscr{H}_{[ \ ]}(\delta, \cF_{\mathrm{DNN}}, \|\cdot\|_\infty) \lesssim dN^2 \log{\left(\frac{dN}{\delta}\right)} \,.
$$
The rest of the analysis is similar to that of the proof of Theorem \ref{thm:dnn_mixing}, where we use part (1) of Theorem \ref{thm:gamma_mixing} to bound $\phi_{n, 1}$ and $\phi_{n, 2}$. To bound $\phi_{n, 1}$, we first note that: 
$$
\mathscr{H}_{[ \ ]} (\delta, \cF_n^2, \|\cdot\|_\infty) \le 2\mathscr{H}_{[ \ ]}(\delta/8b, \cF_n, \|\cdot\|_\infty)  \lesssim dN^2 \log{\left(\frac{dN}{\delta}\right)} \,.
$$
Noting that $\phi_{n, 1}(\delta) \ge \bbE[\sup_{g \in \cF_n^2}|(\bbP_n - P)g|]$ we conclude via part (1) of Theorem 4.1 (with $\alpha = 0$ and $\tilde r = 2$): 
\begin{equation}
    \label{eq:moc_1_additive}
    \phi_{n, 1}(\delta) \lesssim \ \delta n^{\frac{1}{2(1 + \gamma)}} \left(dN^2\log{\left(\frac{N}{\delta}\right)}\right)^{\frac{\gamma}{2(\gamma + 1)}} + n^{\frac12 - \frac{\gamma}{1 + \gamma}}\left(dN^2\log{\left(\frac{N}{\delta}\right)}\right)^{\frac{\gamma}{1 + \gamma}} \,.
\end{equation}
Similarly, for $\phi_{n, 2}$, using \eqref{eq:bracketing_error} and again with part (1) of Theorem \ref{thm:gamma_mixing} (with $\alpha = 0$ and $\tilde r = 2$) we conclude: 
\begin{equation}
    \label{eq:moc_2_additive}
    \phi_{n, 2}(\delta) \lesssim \ \delta n^{\frac{1}{2(1 + \gamma)}} \left(dN^2\log{\left(\frac{N}{\delta}\right)}\right)^{\frac{\gamma}{2(\gamma + 1)}} + n^{\frac12 - \frac{\gamma}{1 + \gamma}}\left(dN^2\log{\left(\frac{N}{\delta}\right)}\right)^{\frac{\gamma}{1 + \gamma}} \,.
\end{equation}
Finally, to obtain the rate of convergence, we choose $r_n$ by solving \eqref{eq:roc_dnn_1} and \eqref{eq:roc_dnn_2} as follows: 
$$
r_n \sim \min\left\{\frac{N^{2s}}{\sqrt{d}},\ \left(\frac{n}{d N^2 (\log{(N)})}\right)^{\frac{\gamma}{2(\gamma+1)}}\right\}, 
$$
where $$N\sim n^{\frac{1}{2(1+2s(1+\gamma^{-1}))}} \ d^{\frac{1}{2(2s(1+\gamma)+\gamma)}}.$$
This completes the proof.

\subsection{Proof of \cref{thm:shapeconv}}\label{sec:pfshapeconv}

Let us elaborate first on how to extend the BCLS estimator beyond the design points. This discussion is a combination of \cite{han2016multivariate} and \cite{Seijo2011}. An alternate way to formulate the least squares problem in \cref{sec:non_param_reg_application} goes as follows: 
$$\min_{\{y_i\in\R\},\{g_i\in\R^d\}}\sum_{i=1}^n (Y_i-y_i)^2,$$
subject to 
\begin{eqnarray*}
    y_j\ge y_i+g_i^{\top}(X_j-X_i),\\ 
    y_i\le 1,\quad y_i\ge -1, \qquad \qquad & \mbox{for all}\ i,j=1,2,\ldots ,n,\\ 
    1\ge y_i+g_i^{\top}(v-X_i), \, \mbox{for all}\ v\in\partial\Omega.
\end{eqnarray*}
The existence of unique solutions in the above formulation along with the optimality conditions is derived in \cite[Lemmas 2.3 and 2.4]{Seijo2011}; also see \cite[Equations 3.2---3.4]{han2016multivariate}. With a notational abuse, let us call the optimizer from the above formulation as $\{y_i\}$ and $\{g_i\}$. We then define $\hat{f}_n$ on the whole of $\Omega$ as follows: 
$$\hat{f}_n(x):=\max_{i=1,2,\ldots ,n} (y_i+\hat{g}_i^{\top}(x-X_i)).$$
We note that with the above extension $\hat{f}_n(X_i)=\hat{y}_i$, i.e., the extension preserves the functions at the sample points. This follows immediately from the first constraint above. As a result, the following inequality continues to hold 
$$\bbP_n(\ell(Y,\hat{f}_n(X)))\le \bbP_n(\ell(Y,f_\star(X))).$$
The above is the only inequality we need to invoke \cref{thm:roc}. 

Another useful result that we will need from the literature is a bound on the bracketing entropy of the class of convex functions supported on a compact polytope. This is borrowed from \cite[Theorem 4.4]{kur2019optimality}. We reproduce their result below for completeness. 

\begin{prop}\label{prop:kurlemma}
Let $\Omega$ be a convex body in $\R^d$ with bounded volume.  Let $f_0$ be a convex function in $\Omega$ that is bounded by $1$. Also recall that $\cF\equiv \cF(\Omega)$ is the space of convex functions on $\Omega$ bounded by $1$. For a fixed $1\le r<\infty$ and $t>0$, let 
$$B_r(f_0;t;\Omega):=\left\{f\in \cF:\ \int |f(x)-f_0(x)|^r\,dx\le t^r\right\}.$$
Suppose $\Delta_1,\ldots ,\Delta_k\subseteq \Omega$ are $d$-simplices with disjoint interiors such that $\cup_{j=1}^k \Delta_j=\Omega$ and $f_0$ is affine on each $\Delta_i$. Then for every $0<u<1$ and $t>0$, we have 
$$\mathscr{H}_{[\,]}(u,B_r(f_0;t;\Omega),\lVert\cdot\rVert_r)\le C_{d,r}k\left(\log{\left(\frac{1}{u}\right)}\right)^{d+1}\left(\frac{t}{u}\right)^{\frac{d}{2}}$$
for a constant $C_{d,r}$ depending only on $r$, $d$, and the volume of $\Omega$.
\end{prop}

We are now in position to prove \cref{thm:shapeconv}.

\emph{Proof of part (a)}. Let us split the proof into two cases $\beta<1$ and $\beta>1$. 

\emph{$\beta<1$ case}. As $\beta>2/(d-2)$ with $d>4$, there exists $r$ large enough such that $$d>2r(1+\beta)/((r-1)\beta).$$ This is possible because in this regime $d\beta>2(1+\beta)$ which implies that there exists $r>2$ such that 
$$r-1>\frac{2(1+\beta)}{\beta(d-2)-2}.$$ Such a choice of $r$ satisfies the first inequality above. Further as $\beta<1$, it automatically satisfies $\beta<r/(r-2)$. 

Now, just as in the proof of \cref{thm:dnn_mixing}, it suffices to bound the local complexity of the two terms
\begin{equation}
\label{eq:boundcon_1}
    \phi_{n, 1}(\delta) \ge \bbE\left[\sup_{\|f - f_\star\|_{L_2} \le \delta} \left|\mathbb{G}_n(f - f_\star)^2\right|\right], 
\end{equation}
and
\begin{equation}
    \label{eq:boundcon_2}
    \phi_{n, 2}(\delta) \ge \bbE\left[\sup_{\|f - f_\star\|_{L_2} \le \delta} \left|\mathbb{G}_n(\xi (f - f_\star))\right|\right]. 
\end{equation}
For obtaining $\phi_{n, 1}(\sigma)$, we need to consider the function class $\cF_{\delta}^2 := \{(f - f_n)^2: f \in \cF_\delta\}$ where $\cF_{\delta}:=\{f:\cF:\ \lVert f-f_\star\rVert_{L_2}\le \delta\}$. From~\cref{lem:bracketing_property}, part 1, and \cref{prop:kurlemma} with $f_0=0$, $\Delta_1=\Omega$, we have: 
$$
\mathscr{H}_{[\,]}(u,\cF_{\delta}^2,\lVert\cdot\rVert_r) \lesssim C_{d,r}\left(\log{\left(\frac{1}{u}\right)}\right)^{d+1}u^{-\frac{d}{2}} \,,
$$
where the $\lesssim$ hides constants depending on $d$, $r$, and the volume of $\Omega$. 

Therefore, by invoking \cref{cor:nondonsk_1}, part 2, we have: 
$$ \bbE\left[\sup_{\|f - f_\star\|_{L_2} \le \delta} \left|\mathbb{G}_n(f - f_\star)^2\right|\right]\lesssim n^{\frac12-\frac2d}\left(\log(n)\right)^{d+1}=:\phi_{n,1}(\delta).$$
We note that \cref{cor:nondonsk_1} does not take into account logarithmic factors in the entropy bound but that can be incorporated to get the above bound with minor technical modifications. We leave the details to the reader.

Moving on to $\phi_{n,2}(\cdot)$, we consider the function space $\cF_{\delta,\xi}:=\{\xi(f(x)-f_\star(x)):\ f\in\cF_{\delta}\}$. By \cref{lem:bracketing_property}, part 3, and \cref{prop:kurlemma} with $f_0=0$, $\Delta_1=\Omega$, we have: 
$$
\mathscr{H}_{[\,]}(u, \cF_{\delta,\xi}, \lVert\cdot\rVert_r) \lesssim C_{d,r}\left(\log{\left(\frac{1}{u}\right)}\right)^{d+1}u^{-\frac{d}{2}} \,.
$$
The rest of the argument is same as that for $\phi_{n,1}(\cdot)$. Therefore, we again get: 
$$ \bbE\left[\sup_{\|f - f_\star\|_{L_2} \le \delta} \left|\mathbb{G}_n(\xi(f - f_\star))\right|\right]\lesssim n^{\frac12-\frac2d}\left(\log(n)\right)^{d+1}=:\phi_{n,2}(\delta).$$
Therefore $(\phi_{n,1}+\phi_{n,2})(\delta)$ is constant in $\delta$ and therefore both $\phi_{n,1}(\delta)/\delta^t$ and $\phi_{n,2}(\delta)/\delta^t$ are decreasing for $t\in (0,2)$. The conclusion follows by invoking \cref{thm:roc}.

\emph{$\beta>1$ case}. The proof is similar as in the $\beta<1$ case, except in the choice of $r$. Note that as $\beta>1$, there exists $r$ such that $\beta>r/(r-2)$. As $d>4$, we are now in the regime of \cref{cor:nondonsk_1}, part 1, instead pf \cref{cor:nondonsk_1}, part 2. The rest of the computation is the same.

\emph{Proof of part (b)}. The conclusion follows from \cref{cor:adaptbd} choosing $\alpha=d/2$. Once again \cref{cor:adaptbd} does not take into account logarithmic factors in the entropy but that can be handled with minor technical adjustments which we leave to the reader.

\subsection{Proofs of \cref{th:otbalance}}\label{sec:pfotbalance}
We begin with some preliminary notation. Given $\mu$ with Lebesgue density $f_{\mu}$, its Fisher information is denoted by 
$$I_{\mu}:=\int \lVert \nabla \log{f_{\mu}(x)}\rVert^2 f_{\mu}(x)\,dx.$$
Note that by the conditions imposed in \cref{th:otbalance}, part (ii), we have:
\begin{equation}\label{eq:assn1}
    I_{\mu}<\infty \quad , \quad I_{\nu}<\infty.
\end{equation}

Next, we recall the Brenier-McCann polar factorization theorem from \cite{brenier1991polar,mccann1995}.

\begin{prop}\label{prop:bmpfthm}
    Suppose $\mu$ and $\nu$ are absolutely continuous probability measures on $\R^d$. Then there exists a $\mu$ almost everywhere unique map $T_0:\R^d\to\R^d$ which is the gradient of a $d$-variate real-valued convex function, such that $T_0\#\mu=\nu$, i.e., $T_0(X)\sim \nu$ if $X\sim\mu$.
\end{prop}

With $T_0$ as in \cref{prop:bmpfthm}, let $\rho_t$ denote the probability density of the measure $((1-t)\mathrm{Id}+t T_0)\#\mu$. The Fisher information of the Wasserstein geodesic between $\mu$ and $\nu$ is then defined as 
$$I_{\mu,\nu}:=\int_{t\in [0,1]}\int_x \lVert \nabla \log\rho_t(x)\rVert^2\rho_t(x)\,dx\,dt.$$
Under the assumptions of \cref{th:otbalance}, we then have:
\begin{equation}\label{eq:assn2}
I_{\mu,\nu}<\infty.
\end{equation}
The following proposition taken from \cite{chizat2020faster} which will be useful in the sequel.
\begin{prop}\label{prop:debiasregwass}
    Assume $\mu$, $\nu$ are compactly supported, absolutely continuous distributions on $\R^d$. Also suppose that \eqref{eq:assn1} and \eqref{eq:assn2} hold. Then there exists a constant $C_{\mu,\nu}$ depending on $I_{\mu}$, $I_{\nu}$, and $I_{\mu,\nu}$, such that 
   $$\bigg|\mathcal{S}_{\vep,2}(\mu,\nu)-\mathcal{W}_2^2(\mu,\nu)\bigg|\le C_{\mu,\nu}\vep^2.$$
\end{prop}

The next proposition (see \cite[Theorem 1.42]{santambrogio2015optimal}) provides the dual representation of the $2$-Wasserstein distance $\mathcal{W}_2^2(\mu,\nu)$. 

\begin{prop}
\label{prop:OTdual}
    Consider $\mu$, $\nu$ which are compactly supported  on subsets of $\R^d$, say $\mathcal{X}$ and $\mathcal{Y}$ respectively (note that absolute continuity is not necessary). Then, the following holds:
    $$\frac{1}{2}\mathcal{W}_2^2(\mu,\nu)=\frac{1}{2}\int \lVert x\rVert^2\,d\mu(x)+\frac{1}{2}\int \lVert y\rVert^2\,d\nu(y)-\min_{f\in\mathfrak{f}}\mathfrak{S}_{\mu,\nu}(f),$$
    where 
    $$\mathfrak{S}_{\mu,\nu}(f):=\int f(x)\,d\mu(x)+\int f^*(y)\,d\nu(y),\quad f^*(y):=\max_{x\in\mathcal{X}}(\langle x,y\rangle-f(x)),\ y \in \mathcal{Y},$$
    and $\mathfrak{f}$ denotes the class of convex functions $f$ on $\mathcal{X}$.
\end{prop}

\begin{proof}[Proof of \cref{th:otbalance}, part (I)]
The proof proceeds begin with similar steps as the proofs of \cite[Lemmas 3 and 4]{chizat2020faster}, \cite[Prop 2]{mena2019statistical}. Finally of course we will need to use the maximal inequalities developed here, in particular \cref{thm:infmainres}. 

Without loss of generality, let us assume that both $\mathcal{X}$ and $\mathcal{Y}$ are contained in the unit ball of diameter $R\in (0,\infty)$. Also note that $\mathfrak{S}_{\mu,\nu}(f)=\mathfrak{S}_{\mu,\nu}(f+c)$ for any constant $c$. Therefore, we can invoke \cite[Step II in the proof of Theorem 2.2]{deb2021rates} and further replace $\mathfrak{f}$ in \cref{prop:OTdual} with $\mathfrak{f}_R:=\{f\in\mathfrak{f}:\ \lVert f\rVert_{\infty}\le C_R\}$, for some fixed constant $C_R>0$. 

Now given any $y_0\in \mathcal{Y}$, $f\in\mathfrak{f}_R$, choose $x_0$ such that $f^*(y_0)=\langle x_0,y_0\rangle - f(x_0)$. Given any other $y\in\mathcal{Y}$, note that $f^*(y)\geq \langle x_0,y\rangle - f(x_0)$. Therefore, 
$$f^*(y_0)-f^*(y)\leq \langle x_0,y_0-y\rangle \leq R\lVert y_0-y\rVert.$$
By reversing the roles of $y_0$ and $y$, we can now conclude that $f^*$ is $R$-Lipschitz. As $(f^*)^*=f$ (by the convexity of $f$), by reversing the roles of $f$ and $f^*$, we again have that $f$ is $R$-Lipschitz. For the sequel, let us write 
$\mathcal{F}_R$ to denote the class of convex functions which are $R$-Lipschitz, uniformly bounded by $C_R$, and supported inside the unit ball of diameter $R$. The above argument shows that 
\begin{align*}
\frac{1}{2}\mathcal{W}_2^2(\mu,\nu)&=\frac{1}{2}\int \lVert x\rVert^2\,d\mu(x)+\frac{1}{2}\int \lVert y\rVert^2\,d\nu(y)-\min_{f\in\mathcal{F}_R}\mathfrak{S}_{\mu,\nu}(f)\\ &=\frac{1}{2}\int \lVert x\rVert^2\,d\mu(x)+\frac{1}{2}\int \lVert y\rVert^2\,d\nu(y)-\mathfrak{S}_{\mu,\nu}(f_0),
\end{align*}
where $f_0\in \argmin_{f\in\mathcal{F}_R} \mathfrak{S}_{\mu,\nu}(f)$. Similarly, 
\begin{align*}
\frac{1}{2}\mathcal{W}_2^2(\hum,\hun)&=\frac{1}{2}\int \lVert x\rVert^2\,d\hum(x)+\frac{1}{2}\int \lVert y\rVert^2\,d\hun(y)-\min_{f\in\mathcal{F}_R}\mathfrak{S}_{\hum,\hun}(f)\\ &=\frac{1}{2}\int \lVert x\rVert^2\,d\hum(x)+\frac{1}{2}\int \lVert y\rVert^2\,d\hun(y)-\mathfrak{S}_{\hum,\hun}(\hat{f}_n),
\end{align*}
where $\hat{f}_n\in \argmin_{f\in\mathcal{F}_R} \mathfrak{S}_{\hum,\hun}(f)$. By the optimality of $\hat{f}_n$, we then get:
$$\frac{1}{2}W_2^2(\hum,\hun)\geq \frac{1}{2}\int \lVert x\rVert^2\,d\hum(x)+\frac{1}{2}\int \lVert y\rVert^2\,d\hun(y)-\mathfrak{S}_{\hum,\hun}(f_0).$$
Consequently, 
\begin{small}
\begin{align}\label{eq:tobound}
    &\;\;\;\;\frac{1}{2}W_2^2(\mu,\nu)-\frac{1}{2}W_2^2(\hum,\hun)\nonumber\\ &\le \bigg|\frac{1}{2}\int \lVert x\rVert^2\,d(\hum-\mu)(x)\bigg|+\bigg|\frac{1}{2}\int \lVert y\rVert^2\,d(\hun-\nu)(y)\bigg|+\sup_{f\in\mathcal{F}_R}\bigg|\int f\,d(\hum-\mu)\bigg|+\sup_{g\in\mathcal{F}_R}\bigg|\int g\,d(\hun-\nu)\bigg|\nonumber \\ &\le 2\sup_{f\in\mathcal{F}_R}\bigg|\int f\,d(\hum-\mu)\bigg|+2\sup_{g\in\mathcal{F}_R}\bigg|\int g\,d(\hun-\nu)\bigg|.
\end{align}
\end{small}
Next by using the optimality of $f_0$ instead of $\hat{f}_n$, we can replace the left-hand side above with absolute values. It therefore suffices to bound the right-hand side of \eqref{eq:tobound} above.

Towards this direction, note that by \cite{bronshtein1976varepsilon}, \cite[Theorem 1]{guntuboyina2012l1}, there exists $\delta_0$ such that $\mathscr{H}_{[ \ ]}(\delta,\mathcal{F}_R,\lVert\cdot\rVert_{\infty})\le C \delta^{-\frac{d}{2}}$ for all $\delta\le \delta_0$, and some $C$ depending on $R$. We can now invoke \cref{cor:infmainres} with $\alpha=d/2$, for $d\ge 4$, to get: 
$$\EE\sup_{f\in\mathcal{F}_R}\bigg|\int f\,d(\hum-\mu)\bigg|\lesssim \begin{cases} n^{-\frac{2}{d}} & \mbox{if}\ \beta>\frac{2}{d-2}\\ n^{-\frac{\beta}{\beta+1}} & \mbox{if}\ \beta<\frac{2}{d-2}\end{cases}.$$
Here we have used the fact that $d/2 > (1+\beta)/\beta$ is equivalent to $\beta>2/(d-2)$ for $d\ge 4$. We can bound the second term in the right hand side of \eqref{eq:tobound} in the same way and recover the same bound as in the above display. This completes the proof of \eqref{eq:wassproof}.
\end{proof}

\begin{proof}[Proof of \cref{th:otbalance}, part (II)]
    Recall from part (I) that we can assume without loss of generality $\mathcal{X}$ and $\mathcal{Y}$ are supported in the ball of diameter $R$. By rescaling, if necessary, let us assume $R=1$. The proof of this result is motivated from \cite[Proposition 4]{chizat2020faster} and \cite[Theorem 3.3]{groppe2023lower}. We will of course need to use the maximal inequalities for dependent observations developed in our paper, in particular \cref{cor:infmainres}. 

    We begin the proof with some notation. We write $$\exp_{\vep}:=\exp\left(\frac{\cdot}{(\vep/2)}\right).$$
    For $\tilde{\mu}\in\mathcal{P}(\mcx)$, define the set
    $$L_{\vep}(\tilde{\mu}):=\{\phi:\mathcal{X}\to [-\infty,\infty):\ \phi\ \mbox{is measurable with}\ 0< \int \exp_{\vep}(\phi)\,d\tilde{\mu}<\infty\}.$$
    Given any $\phi\in L_{\vep}(\tilde{\mu})$, we define 
    $$\phi^{(\vep,\tilde{\mu})}(y):=-\vep \log\int \exp_{\vep}(\phi(x)-\lVert x-y\rVert^2)\,d\tilde{\mu}(x),\qquad y\in\mathcal{Y}.$$
    Given $\tilde{\nu}\in\mathcal{Y}$, we define $L_{\vep}(\tilde{\nu})$ and $\psi^{(\vep,\tilde{\nu})}$ for $\psi\in L_{\vep}(\tilde{\nu})$ analogously. We further need two function classes, namely 
    $$\mathfrak{F}_{\vep}:=\cup_{\xi\in\mathcal{P}(\mathcal{Y})}\{\phi:\mcx\to\R \ \mbox{such that}\ \exists \ \psi:\mathcal{Y}\to\R \ \mbox{with}\ \phi=\psi^{(\vep,\xi)},\ \lVert \phi\rVert_{\infty}\vee \lVert\psi\rVert_{\infty}\le 3/2\},$$
    and
    $$\mathfrak{F}_{\vep}^{\vep,\tilde{\mu}}:=\{\phi^{\vep,\tilde{\mu}}:\ \phi\in\mathfrak{F}_{\vep}\},\qquad \tilde{\mu}\in\mathcal{P}(\mcx).$$

    Recall the definition of the entropic optimal transport cost $\mathcal{T}_{\vep,2}$ and the associated Sinkhorn divergence $\mathcal{S}_{\vep,2}$ from \eqref{eq:enwass} and \eqref{eq:sinkwass} respectively. Now, by \cite[Proposition 2]{mena2019statistical} (also see \cite[Lemma 2.4]{groppe2023lower}), we have: 
    \begin{equation*}
        \EE\big|\mathcal{T}_{\vep,2}(\hum,\hun)-\mathcal{T}_{\vep,2}(\mu,\nu)\big|\le 2\EE\bigg|\sup_{\phi\in\mathfrak{F}_{\vep}}\int \phi\,d(\hum-\mu)\bigg|+2\EE\bigg|\sup_{\psi\in \mathfrak{F}_{\vep}^{(\vep,\hum)}}\int \psi\,d(\hun-\nu)\bigg|.
    \end{equation*}
    By \cite[Theorem 3.3]{groppe2023lower}, $\mathfrak{F}_{\vep}\subseteq \mathfrak{F}_{\mcx}$ and $\mathfrak{F}_{\vep}^{(\vep,\hum)}\subseteq \mathfrak{F}_{\mathcal{Y}}$, where 
    $$\mathfrak{\mcx}:=\{\phi:\mcx\to\R,\ \lVert \phi\rVert_{\infty}\le 3/2, \ \phi \ \mbox{is}\ 1-\mbox{Lipschitz and}\ 1-\mbox{concave}\}.$$
    $\mathfrak{F}_{\mathcal{Y}}$ is defined analogously. Consequently, 
    \begin{equation}\label{eq:callat1}
        \EE\big|\mathcal{T}_{\vep,2}(\hum,\hun)-\mathcal{T}_{\vep,2}(\mu,\nu)\big|\le 2\EE\bigg|\sup_{\phi\in\mathfrak{F}_{\mcx}}\int \phi\,d(\hum-\mu)\bigg|+2\EE\bigg|\sup_{\psi\in \mathfrak{F}_{\mathcal{Y}}}\int \psi\,d(\hun-\nu)\bigg|.
    \end{equation}
    Importantly, the bound in the right hand side of \eqref{eq:callat1} is free of $\vep$. We now bound the first term on the right hand side above and the second term can be bounded similarly. By \cite[Theorem 6]{bronshtein1976varepsilon} and \cite[Theorem 3.2]{guntuboyina2012l1}, there exists $\delta_0>0$ such that we have:
    $$\mathscr{H}_{[ \ ]}(\delta,\mathcal{F}_{\mcx},\lVert \cdot \rVert_{\infty})\le C \delta^{-d/2},$$
    for $\delta\le\delta_0$, and $C$ only depends on $\mcx$. By \cref{cor:infmainres} with $\alpha=d/2$  and \eqref{eq:callat1}, we then have for $d\ge 4$:
    $$\EE\big|\mathcal{T}_{\vep,2}(\hum,\hun)-\mathcal{T}_{\vep,2}(\mu,\nu)\big|\lesssim \begin{cases} \ \ n^{-\frac{2}{d}} & \mbox{if} \ \beta>\frac{2}{d-2} \\ n^{-\frac{\beta}{\beta+1}} & \mbox{if} \ \beta<\frac{2}{d-2}\end{cases}.$$
    We can repeat the same argument to get 
    $$\EE\big|\mathcal{T}_{\vep,2}(\hum,\hum)-\mathcal{T}_{\vep,2}(\mu,\mu)\big|+\EE\big|\mathcal{T}_{\vep,2}(\hun,\hun)-\mathcal{T}_{\vep,2}(\nu,\nu)\big|\lesssim \begin{cases} \ \ n^{-\frac{2}{d}} & \mbox{if} \ \beta>\frac{2}{d-2} \\ n^{-\frac{\beta}{\beta+1}} & \mbox{if} \ \beta<\frac{2}{d-2}\end{cases}.$$
    Combining the above observations with the definition of the Sinkhorn divergence in \eqref{eq:sinkwass}, we get:
    $$\EE\bigg|\mathcal{S}_{\vep,2}(\hum,\hun)-\mathcal{S}_{\vep,2}(\mu,\nu)\bigg|\lesssim \begin{cases} \ \ n^{-\frac{2}{d}} & \mbox{if} \ \beta>\frac{2}{d-2} \\ n^{-\frac{\beta}{\beta+1}} & \mbox{if} \ \beta<\frac{2}{d-2}\end{cases}.$$
    Finally by choosing $\vep\equiv\vep_n$ as in the statement of \cref{th:otbalance}, 
    we have by \cref{prop:debiasregwass}, 
    $$\EE\bigg|\mathcal{S}_{\vep_n(\beta,d),2}(\mu,\nu)-W_2^2(\mu,\nu)\bigg|\lesssim \begin{cases} \ \ n^{-\frac{2}{d}} & \mbox{if} \ \beta>\frac{2}{d-2} \\ n^{-\frac{\beta}{\beta+1}} & \mbox{if} \ \beta<\frac{2}{d-2}\end{cases}.$$
    Finally by using the choice of $\vep\equiv \vep_n$ and $k\equiv k_n$ as in the statement of \cref{th:otbalance}, coupled with \cite[Proposition 2]{chizat2020faster}, we get:
    $$\EE\bigg|\mathcal{S}_{\vep_n,2}(\hum,\hun)-\mathcal{S}^{(k_n)}_{\vep_n,2}(\hum,\hun)\bigg|\lesssim \begin{cases} \ \ n^{-\frac{2}{d}} & \mbox{if} \ \beta>\frac{2}{d-2} \\ n^{-\frac{\beta}{\beta+1}} & \mbox{if} \ \beta<\frac{2}{d-2}\end{cases}.$$
    Combining the three displays above establishes \eqref{eq:sinkproof}.

    Finally, as each step of the Sinkhorn algorithm (as described in \cref{sec:OT}) has $\tilde{O}(n^2)$ complexity, the total time complexity for obtaining $\mathcal{S}_{\vep_n,2}^{(k_n)}$ is $\tilde{O}(n^2 k_n)$. This completes the proof of \cref{th:otbalance}, part (II).
\end{proof}

\subsection{Proof of \cref{prop:classerr}}
The proof proceeds along the same lines as \cite[Theorem 7.1]{koltchinskii2010sparsity} and \cite[Theorem 3.5]{han2021set}. Choose $r_n^2$ such that $r_n^2 n^{\frac{1}{\alpha+1+\gamma^{-1}}}\to\infty$.

First, let us introduce some relevant notation. For $g\in\cG$, write $f_g(x,y)=\mathbf{1}(y\neq g(x))$. Let $\cG_{\delta}:=\{g\in\cG:\ \cE(\cG)\le \delta\}$. Let $K$ be the smallest integer such that $r_n^2 2^K\ge 1$. For any $1\le j\le K$, let $\cF_j:=\{f_{g_1}-f_{g_2}:\ g_1,g_2\in \cG(r_n^2 2^j)\}$. Fix $M>0$ and consider the event 
$$E_{n}:=\left\{\max_{1\le j\le K} \frac{\sup_{f\in\cF_j}|(\bbP_n-P)(f)|}{r_n^2 2^j}\ge \frac{1}{4}\right\},$$
We will bound 
\begin{align}\label{eq:trunclass}
    \bbP(\cE(\hat{g}_n)\ge r_n^2)\le \bbP(\cE(\hat{g}_n)\ge r_n^2, \ E_{n})+\bbP(\cE(\hat{g}_n)\ge r_n^2, \ E_{n}^c).
\end{align}
We will bound the two terms above one by one. For the first term of \eqref{eq:trunclass}, note that by an application of Markov's inequality, we get:
\begin{align*}
    \bbP(\cE(\hat{g}_n)\ge r_n^2, \ E_{n})\le \bbP(E_{n})\le 4\sum_{j=1}^K \frac{\bbE\sup_{f\in\cF_j}|(\bbP_n-P)(f)|}{r_n^2 2^j}.
\end{align*}
Next we bound the expected supremum term above. To wit, recall that for every $g\in\cG$, there exists $C\in\mathcal{C}$ such that $g=\mathbf{1}(C)$. Next, let $\mathcal{S}:=\{S:f_g=\mathbf{1}(S)\}$. For $g_1=\mathbf{1}(C_1)$ and $g_2=\mathbf{1}(C_2)$, we then have $f_{g_1}=\mathbf{1}(S_1)$ and $f_{g_2}=\mathbf{1}(S_2)$ for some $S_1,S_2\in\mathcal{S}$. Therefore, 
$$\bbP(S_1\Delta S_2)=\bbP(f_{g-1}-f_{g_2})^2\le \bbP(g_1-g_2)^2=\bbP(C_1\Delta C_2).$$
Therefore, $\mathscr{H}_{[\,]}(u,\mathcal{S},\lVert\cdot\rVert_1)\le \mathscr{H}_{[\,]}(u,\mathcal{C},\lVert\cdot\rVert_1)$ for $0<u<1$. Further given any $g\in\cG_{r_n^2 2^j}$, let $S\in\mathcal{S}$ be such that $f_g=\mathbf{1}(S)$. Recall that $g_\star(x)=\mathbf{1}(\bbE[Y|X=x]\ge 1/2)$. Then similar to the above display, we get:
$$\bbP(S\Delta S_\star)\le \bbP(g-g_\star)^2\le c^{-1}\cE(g)\le c^{-1}r_n^2 2^j.$$
Note that we have used the margin condition \eqref{eq:margincon} above. Therefore the following inequalities are immediate for $1\le j\le K$: 
$$\bbE\sup_{f\in\cF_j}|\mathbb{G}_n(f)|\lesssim \bbE\sup_{g\in \cG(r_n^2 2^j)} |\mathbb{G}_n(f_g-f_{g_\star})|\lesssim \bbE\sup_{S\in\mathcal{S}:\ \bbP(S\Delta S_\star)\le c^{-1}r_n^2 2^j}|\mathbb{G}_n(S\setminus S_\star)|.$$
Now we use \cref{thm:gamma_mixing} to get: 
$$\bbE\sup_{f\in\cF_j}|\mathbb{G}_n(f)|\lesssim n^{\frac{\gamma(\alpha-1)+1}{2(\alpha \gamma+\gamma+1)}}+n^{\frac12-\frac{\gamma}{\gamma+1}}(r_n^2 2^j)^{-\frac{\alpha \gamma}{1+\gamma}}+n^{\frac12-\frac{\gamma}{2(\gamma+1)}}(r_n^2 2^j)^{\frac12-\frac{\alpha\gamma}{2(1+\gamma)}}.$$
Combining everything, we get:
\begin{align*}
    \sum_{j=1}^K \frac{\bbE\sup_{f\in\cF_j}|(\bbP_n-P)(f)|}{r_n^2 2^j}&\lesssim \frac{1}{r_n^2 n^{\frac{1}{\alpha+1+\gamma^{-1}}}}\sum_{j=1}^K 2^{-j}\, + \frac{1}{(r_n^2)^{1+\frac{\alpha\gamma}{1+\gamma}}n^{\frac{\gamma}{1+\gamma}}}\sum_{j=1}^K (2^{-j})^{1+\frac{\alpha\gamma}{1+\gamma}}\\ &+\frac{1}{r_n^{1+\frac{\alpha\gamma}{1+\gamma}} n^{\frac{\gamma}{2(\gamma+1)}}}\sum_{j=1}^K 2^{-\frac{j}{2}\left(\frac{\alpha\gamma}{1+\gamma}-1\right)}\\ &\lesssim \frac{1}{r_n^2 n^{\frac{1}{\alpha+1+\gamma^{-1}}}}+\frac{1}{(r_n^2)^{1+\frac{\alpha\gamma}{1+\gamma}}n^{\frac{\gamma}{1+\gamma}}}+\frac{1}{r_n^{1+\frac{\alpha\gamma}{1+\gamma}} n^{\frac{\gamma}{2(\gamma+1)}}},
\end{align*}
where we have used $\alpha>1+\gamma^{-1}$ in the last display. Finally as $r_n^2 n^{\frac{1}{\alpha+1+\gamma^{-1}}}\to\infty$, we also have $r_n^{1+\frac{\alpha\gamma}{1+\gamma}} n^{\frac{\gamma}{2(\gamma+1)}}\to\infty$, and  $r_n^{2+2\frac{\alpha\gamma}{1+\gamma}} n^{\frac{\gamma}{\gamma+1}}\to\infty$. As a result, we get: 
$$\lim\limits_{n\to\infty}\bbP(\cE(\hat{g}_n)\ge r_n^2,\ E_n)=0.$$

Now we move on to the second term of \eqref{eq:trunclass}. We will show that on $E_n$, we have $\cE(\hat{g}_n)\le r_n^2$ almost surely. This will clearly complete the proof. Towards this direction, pick any $g\in\cG$ such that $\cE(g)>r_n^2$. Then there exists $j\in [K]$ such that $g\in \cG_{r_n^2 2^j}\setminus \cG_{r_n^2 2^{j-1}}$. Note that this implies the following important observation:
\begin{equation}\label{eq:impob}
    \cE(g)\le r_n^2 2^j \le 2\cE(g).
\end{equation}
Let us define 
$$\hat{\cE}_n(g)=\frac{1}{n}\sum_{i=1}^n \mathbf{1}(Y_i\neq g(X_i))\, - \, \frac{1}{n}\sum_{i=1}^n \mathbf{1}(Y_i\neq \hat{g}_n(X_i)(.$$
Now pick an arbitrary $u>0$ such that $0<u<r_n^2 2^j$ and fix $g'\in\cG(u)$. Then we have:
\begin{align*}
    \cE(g)&=\bbP(f_g-f_{g'})+\bbP(f_{g'}-f_{g_\star})\\ &=\bbP(f_g-f_{g'})+\cE(g')\\ &\le \bbP(f_g-f_{g'})+u\\ &\le \bbP_n(f_g-f_{g'})+u+\sup_{f\in\cF_j}|(\bbP_n-P)(f)|\\ &\le \hat{\cE}_n(g)+u+\frac{1}{4}r_n^2 2^j\\ &\le \hat{\cE}_n(g)+u+\frac{1}{2}\cE(g).
\end{align*}
As $u>0$ can be arbitrarily small, we observe that on the event $E_n$, it holds that 
$$\frac{\hat{\cE}_n(g)}{\cE(g)}\ge \frac12,$$
for all $g$ such that $\cE(g)>r_n^2$. As $\hat{\cE}_n(\hat{g}_n)=0$, clearly $\hat{g}_n$ does not satisfy the above inequality on $E_n$. Therefore we must have $\cE(\hat{g}_n)\le r_n^2$ on the even $E_n$. This completes the proof.
\section{Proofs of supplementary results}\label{sec:auxres}
\begin{lemma}
\label{lem:bracketing_property}
Given a function class $\cF$ and a norm $\lVert \cdot\rVert$, suppose $\bbH_{\| \cdot\|}(\delta, \cF)$ denotes the $\delta$-bracketing entropy of $\cF$ w.r.t. norm $\| \cdot \|$. Then, we have the following: 
\begin{enumerate}
    \item For any monotone and $L$-Lipschitz function $\psi$, we have $\bbH_{\| \cdot\|}(\delta, \psi \circ \cF) \le \bbH_{\| \cdot\|}(\delta/L, \cF)$. 
    \item For any two function class $\cF$ and $\cG$, $\bbH_{\| \cdot\|}(\delta, \cF + \cG) \le \bbH_{\| \cdot\|}(\delta/2, \cF) + \bbH_{\| \cdot\|}(\delta/2, \cG)$. 
    \item For any fixed function $g$, we have $\bbH_{\| \cdot\|}(\delta/\|g\|_\infty, g\cF) \le \bbH_{\| \cdot\|}(\delta, \cF)$.
    \item Let $f_+=f\vee 0$ and $\cF_+=\{f_+:\ f\in\cF\}$. Then $\bbH_{\| \cdot\|}(\delta,\cF_+)\le \bbH_{\| \cdot\|}(\delta,\cF)$.
\end{enumerate}
Here $\psi \circ \cF = \{\psi \circ f, f \in \cF\}$, $g\cF = \{gf: f \in \cF\}$ and $\cF + \cG = \{f + g: f \in \cF, g \in \cG\}$. 
\end{lemma}

\begin{proof}
    For the entire proof, first fix some $\delta > 0$. 
    \begin{enumerate}
        \item Let $\psi$ be a monotone Lipschitz function on the range of $\cF$ with Lipschitz constant $L$. Let $N = \bbH_{\left\| \cdot \right\|}(\delta/L, \cF)$ with the brackets being $\left\{(l_1,  u_1), \dots, (l_N, u_N)\right\}$. By definition we have $\|u_j - l_j\| \le \delta/L$ for all $1 \le j \le N$. Now given any $f$, there exists $1 \le j \le N$ such that $l_j(x) \le f(x) \le u_j(x)$ for all $x$. Therefore, the monotonicity of $\phi$ implies $\phi \circ l(x) \le \phi \circ f(x) \le \phi \circ u(x)$. Furthermore as $\phi$ is $L$-Lipschitz, we have: 
        $$
        \|\phi \circ u - \phi \circ l\| \le L \|u - l\| \le L(\delta/L) = \delta. 
        $$
        Therefore the first claim is established. 

        \item Let $\{(l_1, u_1), \dots, (l_{N_1}, u_{N_1})\}$ be a $\delta/2$ cover of $\cF$ and $\{(\tilde l_1, \tilde u_1), \dots, (\tilde l_{N_2}, \tilde u_{N_2})\}$ be a $\delta/2$ cover of $\cG$. Then pick any function $f + g$ from $\cF + \cG$. There exists $1 \le j_1 \le N_1$ and $1 \le j_2 \le N_2$ such that: 
        \begin{align*}
        l_{j_1}(x) \le f(x) \le u_{j_1}(x), & \ \ \tilde l_{j_2}(x) \le g(x) \le \tilde u_{j_2}(x)  \\
        & \implies l_{j_1}(x) + \tilde l_{j_2}(x) \le f(x) + g(x) \le u_{j_1}(x) + \tilde u_{j_2}(x) \,. 
        \end{align*}
        Furthermore, we have: 
        $$
        \|u_{j_1} + \tilde u_{j_2} - l_{j_1} - \tilde l_{j_2}\| \le \|u_{j_1} - l_{j_1}\| + \|\tilde u_{j_2} - \tilde l_{j_2}\| \le \delta/2 + \delta/2 \le \delta \,.
        $$
        This completes the proof. 

        \item The proof is a combination of proof of (1.) and (2.). Any fixed function $g$ can be written as $g_+ - g_-$ where both $g_+$ and $g_-$ are non-negative. In particular $g_+(x) = \max\{g(x), 0\}$ and $g_-(x) = \max\{-g(x), 0\}$. As before, let $\{(l_1, u_1), \dots, (l_{N_1}, u_{N_1})\}$ be a $\delta/2$ cover of $\cF$. Now for any $f \in \cF$, there exists $1 \le j \le N$ such that $l(x) \le f(x) \le u(x)$ for all $x$ almost surely. Therefore, we have: 
        $$
        l(x)g_+(x) \le g_+(x) f(x) \le u(x) g_+(x), \ \ \text{ and } \ \ l(x)g_-(x) \le g_-(x) f(x) \le u(x) g_-(x) \,,
        $$
        for all $x$. Therefore: 
        $$
        l(x)g_+(x) - u(x)g_-(x) \le g(x)f(x) \le  u(x) g_+(x) -  l(x)g_-(x) \,.
        $$
        Furthermore: 
        \begin{align*}
            \left\|u(x) g_+(x) -  l(x)g_-(x)  - l(x)g_+(x) + u(x)g_-(x) \right\| & = \left\|(u(x)- l(x))g(x)\right\| \\
            & \le \|g\|_\infty \frac{\delta}{\|g\|_\infty} \le \delta \,.
        \end{align*}
        This completes the proof. 

    \item Given $\delta>0$, let $[l_1,u_1],\ldots , [l_N,u_N]$ denote the $\delta$-bracketing set of $\cF$ where $N=\exp(\bbH_{\| \cdot\|}(\delta,\cF))$. It suffices to show that $[l_{1+},u_{1+}], \ldots , [l_{N+},u_{N+}]$ is a $\delta$-bracket of $\cF_+$. To see this, given an $f\in\cF$, let $1\le i_f\le N$ be such that  $l_{i_f}\le f\le u_{i_f}$ and $\lVert u_{i_f}-l_{i_f}\rVert_{\infty}\le \delta$. Then $l_{i_f+}\le f_+\le u_{i_f+}$ and $\lVert u_{i_f+}-l_{i_f+}\rVert_{\infty}\le \lVert u_{i_f}-l_{i_f}\rVert_{\infty}\le \delta$. This completes the proof.
    \end{enumerate}
\end{proof}

\begin{proof}[Proof of \cref{lem:basefn}]
Observe that the above function is well-defined as the function 
$$\beta_q-\frac{q}{n}\left(1+\sum_{k\ge 0\ :\ 2^{-k}\sigma\ge \delta}\mathbb{H}_{\phi}(2^{-k}\sigma,\cF)\right)$$
is $> 0$ at $q=0$ and $\le 0$ for $q=n$. This also implies $\tq{\delta}\ge 1$. Further note that if $\delta_1\le \delta_2$, then
$\mbH{\delta_1}\ge \mbH{\delta_2}$. Consequently,
\begin{align*}
\beta(\tq{\delta_2})&\le \frac{\tq{\delta_2}}{n} \left(1+\sum_{k\ge 0\ :\ 2^{-k}\sigma\ge \delta_2}\mathbb{H}_{\phi}(2^{-k}\sigma,\cF)\right)\\ &\le \frac{\tq{\delta_1}}{n} \left(1+\sum_{k\ge 0\ :\ 2^{-k}\sigma\ge \delta_2}\mathbb{H}_{\phi}(2^{-k}\sigma,\cF)\right).
\end{align*}
Therefore $\tq{\delta_1}\le \tq{\delta_2}$. As $\Lbp(q)$ is non-decreasing in $q$, the conclusion follows.
\end{proof}

\begin{proof}[Proof of \cref{lem:qdbd}]
Choose any arbitrary $t> \ornm{h}$ and sample $X\sim P$. By using the inequality $$\phi^c(1)+\phi\left(\frac{h^2(x)}{t^2}\right)\ge \frac{h^2(x)}{t^2},$$
we get:
\begin{align*}
    \EE h^2(X) \le t^2\left(\phi^c(1)+1\right).
\end{align*}
As $\phi\in\Phi$, we have $\phi^c(1)\in [0,\infty)$. Consequently, setting $\cph:=\sqrt{1+\phi^c(1)}$ and taking infimum over $t\ge \ornm{h}$  completes the proof.
\end{proof}

\begin{proof}
We will use the covariance inequality from \cite[Theorem 1.1]{rio1993covariance} which implies that 
$$
|\cov(h(X_i),h(X_j))|\le \int_0^{\beta(|i-j|)} Q_h^2(u)\,du.$$
Note that $Q_h(U)$, $U\sim \mathrm{Unif}[0,1]$ has the same distribution as $|h(X_1)|$. Therefore, by \cref{lem:qdbd}, the right hand side of the above display is finite. Also note that by \cref{lem:qdbd}, it follows that:
$$
\sum_{i=1}^q \var(h(X_i))\le q\cph^2 \ornm{h}^2.
$$
Combining the two displays above, we get:
\begin{align}\label{eq:davy1}
    \var\left(\sum_{i=1}^q h(X_i)\right) & = \sum_i \var(h(X_i)) + 2 \sum_{i < j} \cov(h(X_i), h(X_j)) \notag \\
    & \le q\cph^2\ornm{h}^2 + 2 \sum_{i < j}\int_0^{\beta(|i-j|)} Q_h^2(u)\,du \notag \\
    & = q\cph^2\ornm{h}^2 + 2 \sum_{i < j}\int_0^{1} \mathds{1}\left(u \le \beta(|i-j|)\right) Q_h^2(u)\,du \notag \\
    & \le q\cph^2\ornm{h}^2+2q\sum_{k=1}^{q-1}\int_{0}^1 \mathds{1}(u\le \beta(k))Q_h^2(u)\,du.
\end{align}
Note that by Markov's inequality, given any $s>0$ and any $t> \ornm{h}$, we have:
\begin{align*}
    P(|h(X)|\ge s)&\le \frac{\EE\phi\left(\frac{h^2(X)}{t^2}\right)}{\phi\left(\frac{s^2}{t^2}\right)}\le \frac{1}{\phi\left(\frac{s^2}{t^2}\right)}.
\end{align*}
Next, given any $u>0$, we choose $$s=t\sqrt{\phi^{-1}\left(\frac{1}{u}\right)}$$
in which case,
$$P(|h(X)|\ge s)\le u \implies Q_h(u)\le t\sqrt{\phi^{-1}\left(\frac{1}{u}\right)}.$$
Now by taking infimum over $t>\ornm{h}$, we get:
\begin{equation}\label{eq:quanbd}
Q_h^2(u)\le \ornm{h}^2\phi^{-1}\left(\frac{1}{u}\right).
\end{equation}
The above observation implies
$$
\sum_{k=1}^{q-1}\int_{0}^1 \mathds{1}(u\le \beta(k))Q_h^2(u)\,du\le \ornm{h}^2\sum_{k=1}^{q-1} \int_0^{\beta(k)}\phi^{-1}\left(\frac{1}{u}\right)\,du\le \ornm{h}^2\Lbp(q).
$$
Plugging the above observation into \eqref{eq:davy1} completes the proof.
\end{proof}

\begin{proof}[Proof of \cref{thm:maximal_finite_davy}]
In order to prove \cref{thm:maximal_finite_davy}, we will need two existing results from the literature which we state here without proof. The first result is a variant of Berbee's coupling Lemma, which was proved independently in \cite{Berbee1979} and \cite{Goldstein1979}.

\begin{prop}\label{lem:berbee}
Let $\{X_i\}_{i\ge 1}$ be a countable sequence of random elements taking values in some Polish space $\mcx$. For any $j\ge 1$, let $b_j:=\beta(X_j,\ \{X_i\}_{i>j})$ (see \cref{def:beta_mixing}). Then there exists a sequence $\{X_i^*\}_{i\ge 1}$ of independent random variables such that for any $j\ge 1$, $X_j^*$ has the same distribution as $X_j$ and $P(X_j\neq X_j^*)\le b_j$.
\end{prop}

The next result is the celebrated Bernstein's inequality (see \cite{bernstein1924modification}).

\begin{prop}\label{prop:bern}
Suppose $Y_1, \dots, Y_n$ are centered independent random variables with $|Y_i| \le b$ and $\var(Y_i) \le \sigma^2_Y$. Then we have: 
$$
\bbP\left(\left|\sum_{i=1}^n Y_i\right| \ge t\right) \le 2\exp{\left(-\frac{t^2}{2\left(n\sigma^2_Y + \frac13 b t\right)}\right)}
$$
for all $t>0$.
\end{prop}

We are now in position to prove \cref{thm:maximal_finite_davy}.

Fix $q > 0$. By Berbee's coupling \cref{lem:berbee}, we can construct another sequence $\{X_i^0\}_{i\ge 1}$ of random elements in $\mcx$ such that: 
\begin{enumerate}
    \item The sub-sequence $\bX^0_k := \{X^0_{qk + 1}, \dots, X^0_{q(k+1)}\} \overset{\mathscr{L}}{=} \{X_{qk + 1}, \dots, X_{q(k+1)}\} =: \bX_k$.
    \item The sequence $\{\bX^0_{2k}\}_{k\ge 1}$ (respectively  $\{\bX^0_{2k - 1}\}_{k\ge 1}$) are independent. 
    \item $\bbP\left(\bX^0_k \neq \bX_k\right) \le \beta_q$. 
\end{enumerate}
For the rest of the calculation, assume without loss of generality $Pf = 0$ for all $f \in \cF$. At the expense of potentially changing the universal constant $K$, we can assume without loss of generality that $n$ is an even multiple of $q$. For $k\le (n/q)-1$, define: 
\begin{enumerate}
    \item $Y_k(f) := \sum\limits_{i=q(k-1) + 1}^{qk} f(X_i)$
    \item  $Y^0_k(f) = \sum\limits_{i=q(k-1) + 1}^{qk} f(X^0_i)$
\end{enumerate}
A simple application of triangle inequality yields: 
\begin{align}
\EE\left[\max_{f \in \cF} \left|\frac{1}{\sqrt{n}}\sum_{i=1}^nf(X_i)\right|\right] & \le \EE\left[\max_{f \in \cF} \left|\frac{1}{\sqrt{n}}\sum_{i=1}^nf(X^0_i)\right|\right] + \EE\left[\max_{f \in \cF} \left|\frac{1}{\sqrt{n}}\sum_{i=1}^n\left(f(X_i) - f(X^0_i)\right)\right|\right] \nonumber \\
& \le \EE\left[\max_{f \in \cF} \left|\frac{1}{\sqrt{n}}\sum_{i=1}^nf(X^0_i)\right|\right] + \frac{2b}{\sqrt{n}}\EE\left[\sum_{i=1}^n \mathds{1}_{X_i \neq X_i^0}\right] \nonumber\\
\label{eq:break_1_maximal_finite}
& \le \EE\left[\max_{f \in \cF} \left|\frac{1}{\sqrt{n}}\sum_{i=1}^nf(X^0_i)\right|\right] + 2b\sqrt{n}\beta_q \,.
\end{align}
We next bound the first term of the above inequality. 

From Lemma \ref{lem:var_bound_improved_davy} we have: 
\begin{equation}
\label{eq:var_bound_finite_beta}
\var\left(\frac{Y_k^0}{\sqrt{q}}\right) = \var\left(\frac{Y_k}{\sqrt{q}}\right) \le \sigma^2\left(\cph^2+2(1+\Lbp(q))\right) = \sigma^2 \pi^2_{\phi, \beta}(q) := 2\tau_q^2\, ,
\end{equation}
where $\pi_{\phi,\beta}(q)$ is defined in \eqref{eq:piq}. Further $$\max_{k\ge 1}\max_{f\in\cF} |Y_k(f)|\vee |Y_k^0(f)|\le qb\, .$$ 
Therefore, for any $f \in \cF$ we have by the Bernstein's inequality \cref{prop:bern}: 
\begin{align*}
    \bbP\left(\left|\frac{1}{\sqrt{n}}\sum_{k=1}^{n/2q}Y^0_{2k}\right| > t\right) & \le 2\exp{\left(-\frac12\frac{nt^2}{\frac{n}{2q}2q\tau_q^2 + \frac13 qb \sqrt{n}t}\right)} \\
    & = 2\exp{\left(-\frac12\frac{t^2}{\tau_q^2 + \frac13 \frac{qb}{\sqrt{n}} t}\right)} 
    \\
    & = 2\exp{\left(-\frac{t^2}{2\tau_q^2 + \frac23 \frac{qb}{\sqrt{n}} t}\right)} \\
    & \le 2\exp{\left(-\frac{t^2}{2 \max\left\{2\tau_q^2, \frac23 \frac{qb}{\sqrt{n}} t\right\}}\right)} \\
    & =  2\exp{\left(-\min\left\{\frac{t^2}{4\tau_q^2}, \frac{3t\sqrt{n}}{4qb}\right\}\right)} \\
    & \le 2\exp{\left(-\frac{t^2}{4\tau_q^2}\right)} + 2 \exp{\left(-\frac{3t\sqrt{n}}{4qb}\right)}
\end{align*}
Therefore, taking union bound, we have: 
$$
\bbP\left(\max_{f \in \cF}\left|\frac{1}{\sqrt{n}}\sum_{k=1}^{n/2q}Y^0_{2k}\right| > t\right) \le 2\exp{\left(-\left[\frac{t^2}{4\tau_q^2}-\log{|\cF|}\right]_+\right)} + 2\exp{\left(-\left[\frac{3t\sqrt{n}}{4qb}-\log{|\cF|}\right]_+\right)}
$$
Same calculation will go through for the odd sequence $\{Y^0_{2k-1}\}$. Therefore we have: 
$$
\bbP\left(\max_{f \in \cF}\left|\frac{1}{\sqrt{n}}\sum_{k=1}^{n/q}Y^0_{k}\right| > t\right) \le 4\exp{\left(-\left[\frac{t^2}{16\tau_q^2}-\log{|\cF|}\right]_+\right)} + 4\exp{\left(-\left[\frac{3t\sqrt{n}}{8qb}-\log{|\cF|}\right]_+\right)}
$$
Integrating with respect to $t$, we have: 
\begin{align*}
    & \EE\left[\max_{f \in \cF}\left|\frac{1}{\sqrt{n}}\sum_{k=1}^{n/q}Y^0_{k}\right|\right] \\
    & = \int_0^\infty \bbP\left(\max_{f \in \cF}\left|\frac{1}{\sqrt{n}}\sum_{k=1}^{n/q}Y^0_{k}\right| > t\right) \ dt \\
    & \le 4 \int_0^\infty \exp{\left(-\left[\frac{t^2}{16\tau_q^2}-\log{|\cF|}\right]_+\right)} \ dt + 4\int_0^\infty \exp{\left(-\left[\frac{3t\sqrt{n}}{8qb}-\log{|\cF|}\right]_+\right)} \ dt \\
    & \le 4 \int_0^{4\tau_q\sqrt{\log{|\cF|}}} \exp{\left(-\left[\frac{t^2}{16\tau_q^2}-\log{|\cF|}\right]_+\right)} \ dt + 4\int_{4\tau_q\sqrt{\log{|\cF|}}}^\infty \exp{\left(-\left[\frac{t^2}{16\tau_q^2}-\log{|\cF|}\right]_+\right)} \ dt \\
    & \qquad + 4\int_0^{\frac{8qb\log{|\cF|}}{3\sqrt{n}}}\exp{\left(-\left[\frac{3t\sqrt{n}}{8qb}-\log{|\cF|}\right]_+\right)} \ dt + 4\int_{\frac{8qb\log{|\cF|}}{3\sqrt{n}}}^\infty\exp{\left(-\left[\frac{3t\sqrt{n}}{8qb}-\log{|\cF|}\right]_+\right)} \ dt \\
    & \le 16\tau_q\sqrt{\log{|\cF|}} + 4|\cF|\int_{4\tau_q\sqrt{\log{|\cF|}}}^\infty \exp{\left(-\frac{t^2}{16 \tau_q^2}\right)} \ dt + \frac{32qb\log{|\cF|}}{3\sqrt{n}} \\ &\qquad + 4|\cF|\int_{\frac{8qb\log{|\cF|}}{3\sqrt{n}}}^\infty \exp{\left(-\frac{3t\sqrt{n}}{8qb}\right)} \ dt \\
    & \le 16\tau_q\sqrt{\log{|\cF|}} + 16\tau_q + \frac{32qb\log{|\cF|}}{3\sqrt{n}} + \frac{32qb}{3\sqrt{n}} \\
    & = 32\tau_q\sqrt{1+\log{|\cF|}} + \frac{32qb}{3\sqrt{n}}\left(1 + \log{|\cF|}\right) \,.
\end{align*}
This along with \eqref{eq:break_1_maximal_finite} completes the proof. 
\end{proof}

\begin{proof}[Proof of \cref{lem:l1phin}]
We begin by noting that for $t\le \vep$, we have:
\begin{align}\label{eq:phin101}
    Q_h(t)\le \frac{\omega_{q,h}(\vep)}{\sqrt{B_q(t)}}.
\end{align}
As $Q_h^{-1}(Q_h(\vep))\le \vep$, we consequently get,
\begin{align*}
    \lVert h\mathds{1}(|h|> Q_h(\vep))\rVert_{L_1(P)}=\int_0^{Q_h^{-1}(Q_h(\vep))} Q_h(t)\,dt \le \omega_{q,h}(\vep)\int_0^{\vep} \frac{1}{\sqrt{B_q(t)}}\,dt.
\end{align*}
The last inequality above follows from \eqref{eq:phin101}. 
Next we observe that $B_q$ is an integral of a bounded decreasing function and hence, is concave, with $B_q(0)=0$. Hence we have: 
$$
B_q(t) = B_q\left(\vep \frac{t}{\vep} + \left(1 - \frac{t}{\vep}\right) 0\right) \ge \frac{t}{\vep}B_q(\vep) \,.
$$
Using this we obtain:  
$$\lVert h\mathds{1}(|h|> Q_h(\vep))\rVert_{L_1(P)}\le \frac{\omega_{q,h}(\vep)\sqrt{\vep}}{\sqrt{B_q(\vep)}}\int_0^{\vep}\frac{\,dt}{\sqrt{t}}=\frac{2\omega_{q,h}(\vep)\vep}{\sqrt{B_q(\vep)}}.$$
This proves \eqref{eq:phin1}. The proof of \eqref{eq:phin3} follows from \eqref{eq:phin1} and the following observations. Note that as $Q_h$ is decreasing:  
\begin{align*}
    \omega_{q,h}^2(\vep)\le \sup_{t\le \vep}\  Q_h^2(t)B_q(t)\le \sup_{t\le \vep}\ \int_0^t Q_h^2(u)\left(\sum_{k=0}^q \mathds{1}(u\le \beta(k))\right)\,du.
\end{align*}
By invoking \eqref{eq:quanbd}, we then have for all $\vep \in [0, 1]$:
\begin{align}\label{eq:phin102}
\omega_{q,h}(\vep)\le \lVert h\rVert_{\phi,P}\sqrt{\Lbp(q)}.
\end{align}
Furthermore,  
\begin{align}\label{eq:phin103}
B_q(\vep)\ge \vep \sum_{k=0}^q \mathds{1}(\vep\le \beta(k)).
\end{align}
Using \eqref{eq:phin103} and \eqref{eq:phin2}, we then get:
$$
v\sqrt{B_{q}(\vep)}\ge \upsilon\sqrt{\vep}\sqrt{\sum_{k=0}^q \mathds{1}(\vep\le \beta(k))}\ge \lVert h\rVert_{\phi,P}\sqrt{\Lbp(q)}\ge \omega_{q,h}(\vep)\ge Q_h(\vep)\sqrt{B_q(\vep)}\,.
$$
Here the first inequality follows from \eqref{eq:phin103}, second inequality from the condition \eqref{eq:phin2}, third inequality from \eqref{eq:phin102} and the last one from the definition of $\omega_{q,h}$. Consequently, for any $v\ge Q_h(\vep)$ we have: 
$$
\lVert h\mathds{1}(|h|> v)\rVert_{L_1(P)}\le \lVert h\mathds{1}(|h|> Q_h(\vep))\rVert_{L_1(P)}\,.
$$
The conclusion then follows by combining \eqref{eq:phin102}, \eqref{eq:phin103}, and \eqref{eq:phin1}.
\end{proof}

\begin{proof}[Proof of \cref{prop:maxfindavy2}]
    The proof is similar to that of \cref{thm:maximal_finite_davy} except that, instead of using \cref{lem:var_bound_improved_davy} in \eqref{eq:var_bound_finite_beta}, we will use the following:
    \begin{align}\label{eq:davyinf}
        \var\left(\sum_{i=1}^q h(X_i)\right) & = \sum_i \var(h(X_i)) + 2 \sum_{i < j} \cov(h(X_i), h(X_j)) \notag \\
    & \le q\lVert h\rVert_{\infty,\mcx}^2 + 2 \sum_{i < j}\int_0^{\beta(|i-j|)} Q_h^2(u)\,du\notag \\
    & \le q\lVert h\rVert_{\infty,\mcx}^2+2q\lVert h\rVert_{\infty,\mcx}^2\sum_{k=1}^{q-1}\beta_k.
    \end{align}
    Here the second inequality follows from \cite[Theorem 1.1]{rio1993covariance} and the third inequality follows from the fact that $\lVert Q_h\rVert_{\infty,(0,1)}\le \lVert h\rVert_{\infty,\mcx}$. 
    The rest of the proof follows verbatim from the proof of \cref{thm:maximal_finite_davy}.
\end{proof}

\begin{proof}[Proof of \cref{lem:liphin2}]
    The proof is is the same as that of \cref{lem:l1phin} except that we will replace \eqref{eq:phin102} with the following bound:
    $$\omega_{q,h}(\vep)\le \sup_{t\le \vep}\sqrt{Q_h^2(t)B_q(t)}\le \lVert h\rVert_{\infty,\mcx}\sqrt{B_q(1)}\le \lVert h\rVert_{\infty,\mcx}\sqrt{\sum_{k=0}^q \beta_k},$$
    for all $\vep\in [0,1]$. The rest of the proof follows verbatim from the proof of \cref{lem:l1phin}.
\end{proof}

\begin{proof}[Proof of \cref{lem:finite_maximal_gamma}]
    The proof is similar to that of \cref{thm:maximal_finite_davy} except that, instead of using \cref{lem:var_bound_improved_davy} in \eqref{eq:var_bound_finite_beta}, we will use the following:
    \begin{align*}
        \var\left(\sum_{i=1}^q h(X_i)\right) & = \sum_i \var(h(X_i)) + 2 \sum_{i < j} \cov(h(X_i), h(X_j)) \notag \\
    & \le q\lVert h\rVert_{L_2(P)}^2 + 2 \lVert h\rVert_{L_2(P)}^2\sum_{i < j}\gamma_{|i-j|}\notag \\
    & \le q b ^{2-\tilde{r}}\sigma^{\tilde{r}}+2q b^{2-\tilde{r}}\sigma^{\tilde{r}}\sum_{k=1}^{q-1}\gamma_k.
    \end{align*}
    Here the second inequality follows from the definition of $\rho$-mixing (see \cref{def:rho_mixing}) and the third inequality follows from the elementary inequality $\int h^2\,dP\le \lVert h\rVert_{\infty,\mcx}^{2-\tilde{r}}\int |h|^{\tilde{r}}\,dP$.  
    The rest of the proof follows verbatim from the proof of \cref{thm:maximal_finite_davy}.
\end{proof}
\end{document}